\documentclass{amsart}%
\usepackage{amsfonts}
\usepackage{graphicx}
\usepackage{hyperref}
\usepackage{amsmath}
\usepackage{amssymb}
\usepackage{color}
\usepackage{wrapfig}%
\providecommand{\U}[1]{\protect\rule{.1in}{.1in}}
\newtheorem{theorem}{Theorem}[section]
\theoremstyle{plain}

\newtheorem{claim}[theorem]{Claim}

\newtheorem{condition}[theorem]{Condition}

\newtheorem{corollary}[theorem]{Corollary}

\newtheorem{definition}[theorem]{Definition}
\newtheorem{example}[theorem]{Example}

\newtheorem{lemma}[theorem]{Lemma}

\newtheorem{Argument}[theorem]{Argument}

\newtheorem{remark}[theorem]{Remark}

\newtheorem{summary}[theorem]{Summary}
\newtheorem{discussion}[theorem]{Discussion}
\numberwithin{equation}{section}
\numberwithin{figure}{section}
\begin{document}
\title[\textbf{Movement of branch points}]{\textbf{Movement of branch points in Ahlfors' theory of covering surfaces}}
\author{Yunling \medskip Chen}
\address{Department of Mathematical Sciences, Tsinghua University, Beijing 100084, P.
R. China}
\email{chenyl20@mails.tsinghua.edu.cn,}
\author{Tianrun \medskip Lin}
\address{Department of Mathematical Sciences, Tsinghua University, Beijing 100084, P.
R. China }
\email{ltr17@mails.tsinghua.edu.cn}
\author{Guangyuan\medskip\ Zhang }
\address{Department of Mathematical Sciences, Tsinghua University, Beijing 100084, P.
R. China}
\email{gyzhang@mail.tsinghua.edu.cn}
\thanks{Project 10971112 and 12171264 supported by NSFC}

\begin{abstract}
In this paper, we will prove a result which is asserted in \cite{Zh2} and is
used in the proof of the existence of extremal surfaces in \cite{Zh2}.

\end{abstract}
\maketitle
\tableofcontents

\section{Introduction}

In 1935, Lars Ahlfors \cite{Ah} introduced the theory of covering surfaces and
gave a geometric illustration of Nevanlinna's value distribution theory.
Depending on the application of the Length -- Area principle (\cite{Ere}%
,p.14), Ahlfors' theory has a metric-topological nature. The most crucial
result in the theory of covering surfaces is Ahlfors' Second Fundamental
Theorem (SFT), which corresponds to Nevanlinna's Second Main Theorem. As the
most important constant in Ahlfors' SFT, the precise bound of the constant
$H(E_{q})$ (we will give the definition later) has not been sufficiently
studied yet. This leads to our work.

We start with several definitions and elementary facts in the theory of
covering surfaces. The unit sphere $S$ is identified with the extended complex
plane $\overline{\mathbb{C}}$ under the stereographic projection
$P:S\to\overline{\mathbb{C}}$ as in \cite{Ah0}. Endowed with the spherical
metric on $S$, the spherical length ${L}$ and the spherical area ${A}$ on $S$
have natural interpretations on ${\mathbb{C}}$ as
\begin{align*}
d{L}  &  =\frac{2|dz|}{1+|z|^{2}},
\end{align*}
and
\begin{align*}
d{A}  &  =\frac{4dxdy}{(1+|z|^{2})^{2}}%
\end{align*}
for any $z\in\mathbb{C}$.

For a closed set $K$ on $\overline{\mathbb{C}}$, a mapping $f:K\rightarrow S$
is called continuous and open if $f$ can be extended to a continuous and open
mapping from a neighborhood of $K$ to $S$. Now we can define the covering surface.

\begin{definition}
\label{defcover} Let $U$ be a domain on $\overline{\mathbb{C}}$ whose boundary
consists of a finite number of disjoint Jordan curves $\alpha_{1}%
,\ldots,\alpha_{n}$. Let $f:\overline{U}\rightarrow S$ be an
orientation-preserving, continuous, open, and finite-to-one map (OPCOFOM).
Then the pair $\Sigma=(f,\overline{U})$ is called a covering surface over $S$,
and the pair $\partial\Sigma=(f,\partial U)$ is called the boundary of
$\Sigma$.

For each point $w \in S$, the covering number $\overline{n}(f,w)$ is defined
as the number of all $w$-points of $f$ in $U$ without counting multiplicity.
That is, $\overline{n}(f,w) = \overline{n}(\Sigma,w) = \sharp\{f^{-1}(w)\cap
U\}$.
\end{definition}

All surfaces in this paper are covering surfaces defined above.

The area of a surface $\Sigma=(f,\overline{U})$ is defined as the spherical
area of $f:\overline{U}\to S$, say,
\begin{align*}
A(\Sigma)=A(f,U)  &  = \int\int_{S} \overline{n}(\Sigma,w) dA(w)\\
&  = \int\int_{\mathbb{C}} \frac{4} {(1+u^{2}+v^{2})^{2}}\overline{n}%
(\Sigma,u+\sqrt{-1}v) dudv.
\end{align*}
And the perimeter of $\Sigma=(f,\overline{U})$ is defined as the spherical
length of $f:\partial{U}\to S$ and write
\begin{align*}
L(\partial\Sigma) = L(f, \partial U).
\end{align*}

\begin{definition}
\label{family F,D} Let $\Sigma=(f,\overline{U})$ be a covering surface.

(1) $\Sigma$ is called a \emph{closed surface,} if $U=S$. For a closed surface
$\Sigma,$ we have $\partial\Sigma=\emptyset,$ and then $L(\partial\Sigma)=0.$

(2) $\Sigma$ is called a \emph{simply-connected} \emph{surface,} if $U$ is a
simply connected domain.

(3) $\mathbf{F}$ denotes all surfaces such that for each $\Sigma=\left(
f,\overline{U}\right)  \in\mathbf{F},$ $U$ is a Jordan domain.
\end{definition}

\begin{remark}
\label{2023-03-28 copy(6)}(A) Let $K_{1}$ and $K_{2}$ be two domains or two
closed domains on $S,$ such that $\partial K_{1}$ and $\partial K_{2}$ are
both consisted of a finite number of disjoint Jordan curves. A mapping
$f:K_{1}\rightarrow K_{2}$ is called a \label{CCM} \emph{complete covering
mapping} (CCM), if (a) for each $p\in K_{2}$ there exists a neighborhood $V$
of $p$ in $K_{2}$ such that $f^{-1}(V)\ $can be expressed as a union
$\cup_{j\in\mathcal{A}}U_{j}$ of disjoint (relative) open sets of $K_{1}$, and
(b) $f|_{U_{j}}:U_{j}\rightarrow V$ is a homeomorphism for each $j\in
\mathcal{A}$.

(B) \label{2023-03-28 copy(7)}We call $f$ a \label{BCCM} \emph{branched
complete covering mapping} (BCCM), if all conditions of (A) hold, except that
(b) is replaced with (b1) or (b2): (b1) If both $K_{1}$ and $K_{2}$ are
domains, then for each $j\in\mathcal{A},$ $U_{j}\cap f^{-1}(p)$ contains only
one point $a_{j}$ of $f^{-1}(p),$ and there exist two homeomorphisms
$\varphi_{j}:U_{j}\rightarrow\Delta,\psi_{j}:V\rightarrow\Delta$ with
$\varphi_{j}\left(  a_{j}\right)  =\psi_{j}\left(  p\right)  =0,$ such that
$\psi_{j}\circ f|_{U_{j}}\circ\varphi_{j}^{-1}(\zeta)=\zeta^{k_{j}},\zeta
\in\Delta,\ $where $k_{j}$ is a positive integer; or (b2) if both $K_{1}$ and
$K_{2}$ are closed domains, then $f|_{K_{1}^{\circ}}:K_{1}^{\circ}\rightarrow
K_{2}^{\circ}$ satisfies (b1) and moreover, $f$ restricted to a neighborhood
of $\partial K_{1}$ in $K_{1}$ is a CCM onto a neighborhood of $\partial
K_{2}$ in $K_{2}.$

(C) \label{2023-03-28 copy(8)}For a surface $\Sigma=\left(  f,\overline
{U}\right)  \ $over $S,$ $f$ is in general not a CCM or BCCM. When $f\left(
z\right)  =z^{2},$ both $f:\overline{\Delta}\rightarrow\overline{\Delta}\ $and
$f:\Delta\rightarrow\Delta$ are BCCMs, but when $f\left(  z\right)  =z\left(
\frac{z-a}{1-\bar{a}z}\right)  ^{2},$ $f:\Delta\rightarrow f(\Delta)$ is
neither a CCM nor a BCCM.
\end{remark}

Ahlfors' Second Fundamental Theorem gives the relationship between $A(\Sigma
)$, $\overline{n}(\Sigma)$ and $L(\partial\Sigma)$.

\begin{theorem}
[Ahlfors' Second Fundamental Theorem]\label{sft} Given an integer $q\ge3$, let
$E_{q}=\{a_{1},\ldots,a_{q}\}$ be a set of distinct $q$ points on $S$. Then
there exists a positive constant $h$ depending only on $E_{q}$, such that for
any covering surface $\Sigma= (f,\overline{U})\in\mathbf{F} $, we have
\begin{align}
\label{h1}(q-2)A(\Sigma) \le4\pi\sum\limits_{j=1}^{q} \overline{n}%
(\Sigma,a_{j}) + h L(\partial\Sigma).
\end{align}
In particular, if $f(\overline{U}) \cap\{0,1,\infty\}=\emptyset$, then we
have
\begin{align}
\label{h0}A(\Sigma) \le h L(\partial\Sigma).
\end{align}

\end{theorem}

It is a natural question that whether we can find a precise lower bound for
the constants $h$ in Theorem \ref{sft}. For this purpose, we need to define
the remainder-perimeter ratio $H(\Sigma)$ as follows.

\begin{definition}
\label{1.6} For a covering surface $\Sigma=(f,\overline{U})\in\mathbf{F} $ and
a set $E_{q}=\{a_{1},\ldots,a_{q}\}$ on $S$, we define the total covering
number over $E_{q}$ as
\[
\overline{n}(f,E_{q}) = \overline{n}(\Sigma,E_{q}) = \sum\limits_{j=1}^{q}
\overline{n}(\Sigma,a_{j}) = \sharp\{f^{-1}(E_{q})\cap U\},
\]
the remainder as
\[
R(\Sigma,E_{q})=(q-2)A(\Sigma) - 4\pi\overline{n}(\Sigma,E_{q}),
\]
and the remainder-perimeter ratio as
\[
H(\Sigma,E_{q}) = \frac{R(\Sigma,E_{q})}{L(\partial\Sigma)}.
\]

\end{definition}

\begin{remark}
In the sequel, we always use $R(\Sigma)$ and $H(\Sigma)$ without emphasizing
the set $E_{q}$.
\end{remark}

We can observe that to estimate the constants $h$ in Theorem \ref{sft}, we are
supposed to give an upper bound of $H(\Sigma)$. In \cite{Zh1}, the last author
developed an innovative method to compute the precise value of the constant
$h$ in (\ref{h0}).

\begin{theorem}
[Theorem 1.1 in \cite{Zh1}]For any surface $\Sigma=(f,\overline{U}%
)\in\mathbf{F}$ with $f(\overline{U}) \cap\{0,1,\infty\}=\emptyset$, we have
\[
A(\Sigma)< h_{0}L(\partial\Sigma),
\]
where
\begin{equation}
h_{0}=\max_{\theta\in\left[  0,\frac{\pi}{2}\right]  } \left\{  \frac{\left(
\pi+\theta\right)  \sqrt{1+\sin^{2}\theta}}{\arctan\dfrac{\sqrt{1+\sin
^{2}\theta}}{\cos\theta}}-\sin\theta\right\}  .
\end{equation}
Moreover, the constant $h_{0}$ is sharp: there exists a sequence of covering
surface $\{\Sigma_{n}\}$ in $\mathbf{F}$ with $f(\overline{U_{n}})
\cap\{0,1,\infty\}=\emptyset$ such that $A(\Sigma_{n})/L(\partial\Sigma
_{n})\rightarrow h_{0}$ as $n\rightarrow\infty.$
\end{theorem}

However, in general cases, it will be very difficult to estimate the precise
bound of the constant $h$. Since the branch points (See definition in Remark
\ref{notation}) outside of $f^{-1}(E_{q})$ of a surface bring a lot of trouble
in the research, Sun and the last author tried overcoming such problems in
\cite{S-Z}. Unfortunately, we observe that the published result in \cite{S-Z}
does not work well enough. Before establishing our main theorem, we introduce
more terminologies and definitions.

All paths and curves considered in this paper are oriented and any subarc of a
path or closed curve inherits this orientation. Sometimes paths and curves
will be regarded as sets, but only when we use specific set operations and set
relations. For an oriented circular arc $c$, the circle $C$ containing $c$ and
oriented by $c$ is called the circle determined by $c$.

\begin{definition}
\label{geodesic}\textbf{\label{2023-03-28 copy(18)}}For any two non-antipodal
points $p$ and $q$ on $S,$ $\overline{pq}$ is the geodesic on $S$ from $p$ to
$q:$ the shorter of the two arcs with endpoints $p$ and $q$ of the great
circle on $S$ passing through $p$ and $q.$ Thus $d(p,q)<\pi$ and
$\overline{pq}$ is uniquely determined by $p$ and $q$. An arc of a great
circle on $S$ is called a \emph{line segment} on $S,$ and to emphasize this,
we also refer to it as a \emph{straight line segment. }For the notation
$\overline{pq},$ when $p$ and $q$ are explicit complex numbers we write
$\overline{p,q},$ to avoid ambiguity such as $\overline{123}=\overline{12,3}$
or $\overline{1,23}.$ When $p$ and $q$ are two antipodal points of $S,$
$\overline{pq}$ is not unique and $d\left(  p,q\right)  =\pi.$ To avoid
confusions, when we write $\overline{pq},$ or say $\overline{pq}$ is well
defined, we always assume $d\left(  p,q\right)  <\pi.$
\end{definition}

\begin{definition}
\label{in}(1) For a Jordan domain $D$ in $\overline{\mathbb{C}},$ let $h$ be a
M\"{o}bius transformation with $h(D)\subset\Delta.$ Then $\partial D$ is
oriented by $h$ and the anticlockwise orientation of $\partial h(D).$ The
boundary of every Jordan domain on $S$ is oriented in the same way, via
stereographic projection.

(2) For a Jordan curve $C$ on $\overline{\mathbb{C}}$ or $S,$ the domain
$T_{C}$ bounded by $C$ is called \emph{enclosed} by $C$ if the boundary
orientation of $T_{C}$ agrees with the orientation of $C.$

(3) A domain $D$ on $S$ is called \emph{convex} if for any two points $q_{1}$
and $q_{2}$ in $D$ with $d(q_{1},q_{2})<\pi$, $\overline{q_{1}q_{2}}\subset D
$; a Jordan curve on $S$ is called \emph{convex} if it encloses a \emph{convex
domain }on $S$; a path on $S$ is called \emph{convex} if it is an arc of a
\emph{convex} Jordan curve.

(4) Let $\gamma:[a,b]\rightarrow S$ be a path on $S$ and $p_{0}\in(a,b).$
$\gamma$ is called \emph{convex at} $p_{0},$ if $\gamma$ restricted to a
neighborhood $(p_{0}-\delta,p_{0}+\delta)$\label{6868-5} of $p_{0}$ in $(a,b)$
is a \emph{convex} Jordan path, with respect to the parametrization giving
$\gamma$ ($t$ increases). $\gamma$ is called \emph{strictly convex} at $p_{0}$
if $\gamma$ is convex at $p_{0}$ and restricted to a neighborhood $N_{p_{0}}$
of $p_{0}$ in $(a,b)$ is contained in some closed hemisphere $S_{1}$ on $S$
with $\gamma_{N_{p_{0}}}\cap S_{1}=\gamma(p_{0}).$
\end{definition}

Recall that $\mathbf{F}$ is the space of covering surfaces $\Sigma
=(f,\overline{U})$,where $U$ is a Jordan domain on $\overline{\mathbb{C}}$.
Before introducing a subspace of $\mathbf{F}$, we need to give the definition
of partition. For a Jordan curve $\alpha$ in $\mathbb{C}$, its partition is a
collection $\{\alpha_{j}\}_{j=1}^{n}$ of its subarcs such that $\alpha
=\cup_{j=1}^{n}\alpha_{j}$ and $\alpha_{j}^{\circ}$ are disjoint and arranged
anticlockwise. In this setting we write $\alpha=\alpha_{1}+\alpha_{2}%
+\cdots+\alpha_{n}.$ Here $\alpha_{j}^{\circ}$ is the interior of $\alpha_{j}%
$, which is $\alpha_{j}$ without endpoints. A partition
\begin{equation}
\partial\Sigma=\gamma_{1}+\gamma_{2}+\cdots+\gamma_{n} \label{6662}%
\end{equation}
of $\partial\Sigma$ for a surface $\Sigma=(f,\overline{U})\in\mathbf{F}$ is
equivalent to a partition
\begin{equation}
\partial U=\alpha_{1}+\alpha_{2}+\cdots+\alpha_{n}%
\end{equation}
of $\partial U$ such that $\gamma_{j}=(f,\alpha_{j})$ for $j=1,\dots,n.$

\begin{definition}
We denote by $\mathcal{F}$ the subspace of $\mathbf{F}$ such that for each
$\Sigma=\left(  f,\overline{U}\right)  $, $\partial\Sigma$ has a partition
\[
\partial\Sigma=c_{1}+c_{2}+\dots+c_{n},
\]
where $c_{1},\dots,c_{n}$ are simple convex circular (SCC) arcs. This means
that $\partial U$ has a partition
\[
\partial U=\alpha_{1}+\alpha_{2}+\dots+\alpha_{n}%
\]
such that $\alpha_{j},1\leq j\leq n$, are arranged anticlockwise and $f$
restricted to each $\alpha_{j}$ is a homeomorphism onto the convex circular
arc $c_{j}$.
\end{definition}

Now we introduce some subspaces of $\mathcal{F}$ which can describe some
properties of the covering surfaces precisely.

\begin{definition}
\label{sp} For given positive number $L$, $\mathcal{F}(L)$ denotes the
subspace of $\mathcal{F}$ in which every surface has boundary length
$L(\partial\Sigma)\leq L$.

$\mathcal{C}(L,m)$ denotes the subspace of $\mathcal{F}(L)$ such that
$\Sigma=\left(  f,\overline{\Delta}\right)  \in\mathcal{C}(L,m)$ if and only
if $\partial\Delta$ and $\partial\Sigma$ have $\mathcal{C}(L,m)$-partitions.
This means that $\partial\Delta$ and $\partial\Sigma$ have partitions
\begin{equation}
\partial\Delta=\alpha_{1}\left(  a_{1},a_{2}\right)  +\alpha_{2}\left(
a_{2},a_{3}\right)  +\dots+\alpha_{m}\left(  a_{m},a_{1}\right)  \label{part}%
\end{equation}
and
\begin{equation}
\partial\Sigma=c_{1}\left(  q_{1},q_{2}\right)  +c_{2}\left(  q_{2}%
,q_{3}\right)  +\dots+c_{m}\left(  q_{m},q_{1}\right)  \label{part1}%
\end{equation}
respectively, such that $c_{j}\left(  q_{j},q_{j+1}\right)  =\left(
f,\alpha_{j}\left(  a_{j},a_{j+1}\right)  \right)  $ is an SCC arc for each
$j=1,\dots,m.$

Given $q\geq3$, let $E_{q}=\{a_{1},\ldots,a_{q}\}$ be a set of $q$ distinct
points. $\mathcal{C}^{\ast}(L,m)$ denotes the subspace of $\mathcal{C}%
(L,m)\mathcal{\ }$such that $\Sigma=\left(  f,\overline{\Delta}\right)
\in\mathcal{C}^{\ast}\left(  L,m\right)  $ if and only if $\partial\Delta$ and
$\partial\Sigma$ have $\mathcal{C}^{\ast}(L,m)$-partitions. That is, the
partitions are $\mathcal{C}(L,m)$-partitions in (\ref{part}) and (\ref{part1})
so that $f$ has no branch points in $\alpha_{j}^{\circ}\cap f^{-1}(E_{q})$ for
every $j=1,\dots,m$.

$\mathcal{F}(L,m)$ denotes the subspace of $\mathcal{C}(L,m)$ such
that$\mathcal{\ }\Sigma=\left(  f,\overline{\Delta}\right)  \in\mathcal{F}%
\left(  L,m\right)  $ if and only if $\partial\Delta$ and $\partial\Sigma$
have $\mathcal{F}(L,m)$-partitions (\ref{part}) and (\ref{part1}), that is,
the partitions are $\mathcal{C}(L,m)$-partitions such that, for each
$j=1,2,\dots,m,$ $f$ has no branch point in $\alpha_{j}^{\circ}.$

$\mathcal{F}_{r}$ denotes the subspace of $\mathcal{F}$ such that
$\Sigma=\left(  f,\overline{\Delta}\right)  \in\mathcal{F}_{r}$ if and only if
$f$ has no branch point in $\overline{\Delta}\backslash f^{-1}(E_{q}),$ say,
$C_{f}^{\ast}\left(  \overline{\Delta}\right)  =\emptyset,$ and define
\[
\mathcal{F}_{r}(L)=\mathcal{F}_{r}\cap\mathcal{F}(L),
\]%
\[
\mathcal{F}_{r}(L,m)=\mathcal{F}_{r}\cap\mathcal{F}(L,m).
\]

\end{definition}

\begin{remark}
The condition in the definition of $\mathcal{F}(L,m)$ is equivalent to say
that, for each $j=1,\dots,m,$ $f$ restricted to a neighborhood of $\alpha
_{j}^{\circ}$ in $\overline{\Delta}$ is a homeomorphism onto a one-side
neighborhood of $c_{j}^{\circ}$, which is the part of a neighborhood of
$c_{j}^{\circ}$ contained in the closed disk enclosed by the circle determined
by $c_{j}.$
\end{remark}

By definition, we have
\[
\mathcal{F}_{r}\left(  L,m\right)  \subsetneqq\mathcal{F}\left(  L,m\right)
\subsetneqq\mathcal{C}^{\ast}\left(  L,m\right)  \subsetneqq\mathcal{C}\left(
L,m\right)  ,
\]
and
\[
\mathcal{F}(L)=\cup_{m=1}^{\infty}\mathcal{F}\left(  L,m\right)  =\cup
_{m=1}^{\infty}\mathcal{C}\left(  L,m\right)  .
\]
For each $\Sigma\in\mathcal{C}\left(  L,m\right)  ,$ there exists an integer
$m_{1}>m$ such that $\Sigma\in\mathcal{F}\left(  L,m_{1}\right)  .$

Analogous to Definition\ref{1.6}, we define the Ahlfors' constants in
different subspaces of covering surfaces.

\begin{definition}
\label{hs1}\ Given $q\geq3$, for any set $E_{q}=\{a_{1},\ldots,a_{q}\}$ of $q$
distinct points, we define
\[
H_{0}=\sup_{\Sigma\in\mathcal{F}}H(\Sigma)=\sup_{\Sigma\in\mathcal{F}}%
H(\Sigma,E_{q}),
\]%
\[
H_{L}=H_{L}(E_{q})=\sup_{\Sigma\in\mathcal{F}(L)}H(\Sigma)=\sup_{\Sigma
\in\mathcal{F}(L)}H(\Sigma,E_{q}),
\]%
\[
H_{L,m}=\sup_{\Sigma\in\mathcal{F}(L,m)}H(\Sigma)=\sup_{\Sigma\in
\mathcal{F}(L,m)}H(\Sigma,E_{q}),
\]

\end{definition}

\begin{remark}
For any surface $\Sigma\in\mathcal{F}$ and any $\varepsilon>0\mathbf{,}$ to
estimate $H(\Sigma)$ we may assume $L(\partial\Sigma)<+\infty$. Otherwise, we
have $H(\Sigma)=0.$
\end{remark}

\begin{definition}
\label{L}Let $\mathcal{L}$ be the set of continuous points of $H_{L}%
=H_{L}(E_{q}),$ with respect to $L.$
\end{definition}

\begin{remark}
\label{RL} By Ahlfors' SFT, we can see that
\[
H_{0}=\lim\limits_{L\to+\infty}H_{L}<+\infty.
\]
Since $H_{L}$ increase with respect to $L,$ it is clear that $\left(
0,+\infty\right)  \backslash\mathcal{L}$ is just a countable set. Thus for
each $L\in\mathcal{L}$, there exists a positive number $\delta_{L}$ such that
for each $L^{\prime}\in(L-\delta_{L},L+\delta_{L}),$ we have
\[
H_{L}-\frac{\pi}{2L}<H_{L^{\prime}}<H_{L}+\frac{\pi}{2L}.
\]

\end{remark}

Now we can state our main theorem as follows.

\begin{theorem}
[Main Theorem]\label{re} Let $L\in\mathcal{L}$ and let $\Sigma=\left(
f,\overline{\Delta}\right)  $ be a covering surface in $\mathcal{C}^{\ast
}(L,m)$. Assume that
\begin{equation}
H(\Sigma)>H_{L}-\frac{\pi}{2L(\partial\Sigma)}. \label{210615}%
\end{equation}
Then there exists a surface $\Sigma^{\prime}=\left(  f^{\prime},\overline
{\Delta}\right)  $ such that

(i) $\Sigma^{\prime}\in\mathcal{F}_{r}(L,m)$.

(ii) $H(\Sigma^{\prime})\geq H(\Sigma)$ and $L(\partial\Sigma^{\prime})\leq
L(\partial\Sigma).$ Moreover, at least one of the inequalities is strict if
$\Sigma\notin\mathcal{F}_{r}(L,m)$.

(iii) When $L(\partial\Sigma^{\prime})=L(\partial\Sigma)$, we have
$\partial\Sigma^{\prime}=\partial\Sigma$ and they share the same
$\mathcal{F}(L,m)$-partitions (\ref{part}) and (\ref{part1}).
\end{theorem}

Now we outline the structure of this paper. Section 2 introduces some
fundamental properties of covering surfaces, especially the surgeries to sew
two surfaces along the equivalent boundary arcs. In Section 3, we remove the
non-special branch points of the given surface, and in Section 4 we finish our
proof of the main theorem.

\section{Elementary properties of covering surfaces}

This section consists of some useful properties of covering surfaces. For a
path $\Gamma$ on $S$ given by $z=z(t),t\in\lbrack t_{1},t_{2}],$ $-\Gamma$ is
the opposite path of $\Gamma$ given by $z=z(-t),t\in\lbrack-t_{2},-t_{1}]$.

\begin{definition}
\label{lune-lens}A convex domain enclosed by a convex circular arc $c$ and its
chord $I$ is called a \emph{lune} and is denoted by $\mathfrak{D}^{\prime
}\left(  I,c\right)  ,\mathfrak{D}^{\prime}\left(  I,\theta(c)\right)  ,$
$\mathfrak{D}^{\prime}\left(  I,L(c)\right)  ,$ or $\mathfrak{D}^{\prime
}\left(  I,k(c)\right)  $ where $\theta$ is the interior angle at the two
cusps, $k$ is the curvature of $c$ and $I$ is oriented such that\footnote{The
initial and terminal points of $I$ and $c$ are the same, respectively, in the
notation $\mathfrak{D}^{\prime}(I,\theta),$ in other words, $\mathfrak{D}%
^{\prime}(I,\theta)$ is on the right hand side of $I.$} $\partial
\mathfrak{D}^{\prime}\left(  I,\theta\right)  =c-I.$

For two lunes $\mathfrak{D}^{\prime}\left(  I,\theta_{1}\right)  $ and
$\mathfrak{D}^{\prime}\left(  -I,\theta_{2}\right)  $ sharing the common chord
$I$ we write%
\[
\mathfrak{D}\left(  I,\theta_{1},\theta_{2}\right)  =\mathfrak{D}^{\prime
}\left(  I,\theta_{1}\right)  \cup I^{\circ}\cup\mathfrak{D}^{\prime}\left(
-I,\theta_{2}\right)
\]
and called the Jordan domain $\mathfrak{D}\left(  I,\theta_{1},\theta
_{2}\right)  $ a lens. Then the notations $\mathfrak{D}\left(  I,l_{1}%
,l_{2}\right)  $, $\mathfrak{D}\left(  I,c_{1},c_{2}\right)  \ $and
$\mathfrak{D}\left(  I,k_{1},k_{2}\right)  $ are in sense and denote the same
lens, when $l_{j}=L(c_{j})$ and $k_{j}$ is the curvature of $c_{j},$ $j=1,2$,
say,%
\begin{align*}
\mathfrak{D}\left(  I,c_{1},c_{2}\right)   &  =\mathfrak{D}\left(
I,l_{1},l_{2}\right)  =\mathfrak{D}\left(  I,k_{1},k_{2}\right)
=\mathfrak{D}^{\prime}\left(  I,l_{1}\right)  \cup I^{\circ}\cup
\mathfrak{D}^{\prime}\left(  -I,l_{2}\right) \\
&  =\mathfrak{D}^{\prime}\left(  I,c_{1}\right)  \cup I^{\circ}\cup
\mathfrak{D}^{\prime}\left(  -I,c_{2}\right)  =\mathfrak{D}^{\prime}\left(
I,k_{1}\right)  \cup I^{\circ}\cup\mathfrak{D}^{\prime}\left(  -I,k_{2}%
\right)  .
\end{align*}

\end{definition}

For a lune $\mathfrak{D}^{\prime}\left(  I,\tau\right)  ,$ whether $\tau$
denotes the length $l$, the angle $\theta,$ or the curvature $k$ is always
clear from the context, and so is for the lens $\mathfrak{D}\left(  I,\tau
_{1},\tau_{2}\right)  .$ By definition, we have $0<\theta_{j}\leq\pi$ for
$j=1,2,$ but for the domain $\mathfrak{D}\left(  I,\theta_{1},\theta
_{2}\right)  $ it is permitted that $\theta_{1}$ or $\theta_{2}$ is zero, say
$\mathfrak{D}\left(  I,\theta_{1},\theta_{2}\right)  $ reduces to
$\mathfrak{D}^{\prime}\left(  I,\theta_{1}\right)  $ or $\mathfrak{D}^{\prime
}\left(  -I,\theta_{2}\right)  .$ By definition of $\mathfrak{D}%
(I,\theta,\theta)$ we have
\[
\mathfrak{D}(I,\theta,\theta)=\mathfrak{D}^{\prime}\left(  I,\theta\right)
\cup\mathfrak{D}^{\prime}\left(  -I,\theta\right)  \cup I^{\circ},
\]
and $\theta\in(0,\pi].$ If $I=\overline{1,0,-1}$ and $\theta=\pi/2,$ for
example, $\mathfrak{D}\left(  I,\theta,\theta\right)  =\Delta$ and
$\mathfrak{D}^{\prime}\left(  I,\theta\right)  =\Delta^{+}$ is the upper half
disk of $\Delta.$

Let $\Sigma=\left(  f,\overline{U}\right)  \in\mathcal{F}$ and let
$p\in\partial U$. If $f$ is injective near $p$, then $f$ is homeomorphic in a
closed Jordan neighborhood $N_{p}$ of $p$ in $U$, and then $f(N_{p})$ is a
closed Jordan domain on $S$ whose boundary near $f(p$) is an SCC arc, or two
SCC arcs joint at $f(p)$, and thus the interior angle of $f(N_{p})$ at $f(p)$
is well defined, called the interior angle of $\Sigma$ at $p$ and denoted by
$\angle(\Sigma,p).$

In general, we can draw some paths $\{\beta_{j}\}_{j=1}^{k}$ in $\overline{U}$
with $\cup_{j=1}^{k}\beta_{j}\backslash\{p\}\subset U$ and $\beta_{j}\cap
\beta_{i}=\{p\}\ $if $i\neq j,$ such that each $\left(  f,\beta_{j}\right)  $
is a simple line segment on $S$, $\cup_{j=1}^{k}\beta_{j}$ divides a closed
Jordan neighborhood $N_{p}$ of $p$ in $\overline{U}$ into $k+1$ closed Jordan
domains $\overline{U_{j}}\ $with $p\in\overline{U_{j}},j=1,\dots,k+1,$ and
$U_{i}\cap U_{j}=\emptyset\ $if $i\neq j,$ and $f$ restricted to
$\overline{U_{j}}$ is a homeomorphism with $\left(  f,\overline{U_{j}}\right)
\in\mathcal{F}$ for each $j.$ Then the interior angle of $\Sigma$ at $p$ is
defined by
\[
\angle\left(  \Sigma,p\right)  =\sum_{j=1}^{k+1}\angle\left(  \left(
f,\overline{U_{j}}\right)  ,p\right)  .
\]

\begin{theorem}
\label{st}\label{230330-11:14 copy(1)}(i). (Stoilow's Theorem \cite{S}
pp.120--121) Let $U$ be a domain on $\overline{\mathbb{C}}$ and let
$f:U\rightarrow S$ be an open, continuous and discrete mapping. Then there
exist a domain $V$ on $\overline{\mathbb{C}}$ and a homeomorphism
$h:V\rightarrow U,$ such that $f\circ h:V\rightarrow S$ is a holomorphic mapping.

\label{230330-11:14 copy(2)}(ii). Let $\Sigma=(f,\overline{U})$ be a surface
where $U$ is a domain on $\overline{\mathbb{C}}.$ Then there exists a domain
$V$ on $\overline{\mathbb{C}}$ and an OPH $h:\overline{V}\rightarrow
\overline{U}$ such that $f\circ h:\overline{V}\rightarrow S$ is a holomorphic mapping.

\label{230330-11:14 copy(3)}(iii) Let $\Sigma=(f,\overline{U})\in\mathbf{F}.$
Then there exists an OPH $\varphi:\overline{U}\rightarrow\overline{U}$ such
that $f\circ\varphi$ is holomorphic on $U$.
\end{theorem}

What $f$ is discrete means that $f^{-1}(w)\cap K$ is finite for any compact
subset $K$ of $U.$

\begin{proof}
\label{230330-11:14 copy(4)}\label{ST1}Let $\Sigma=(f,\overline{U})$ be a
surface where $U$ is a domain on $\overline{\mathbb{C}}.$ Then $f:\overline
{U}\rightarrow S$ is the restriction of an OPCOFOM $g$ defined in a
neighborhood $U_{1}$ of $\overline{U},$ and thus by Stoilow's theorem, there
exists a domain $V_{1}$ on $\overline{\mathbb{C}}$ and an OPH $h:V_{1}%
\rightarrow U_{1}$ such that $g\circ h$ is holomorphic on $V_{1}$ and then for
$\overline{V}=h^{-1}(\overline{U}),$ $f\circ h$ is holomorphic on
$\overline{V},$ and thus (ii) holds.

\label{230330-11:14 copy(5)}Continue the above discussion and assume $U$ is a
Jordan domain. Then $V$ is also a Jordan domain and by Riemann mapping theorem
there exists a conformal mapping $h_{1}$ from $U$ onto $V$ and by
Caratheodory's extension theorem $h_{1}$ can be extended to be homeomorphic
from $\overline{U}$ onto $\overline{V},$ and thus the extension of $h\circ
h_{1}$ is the desired mapping $\varphi$ in (iii).
\end{proof}

\label{samecurve}For two curves $(\alpha_{1},[a_{1},b_{1}])$ and $(\alpha
_{2},[a_{2},b_{2}])$ on $S,$ we call they equivalent and write
\[
(\alpha_{1},[a_{1},b_{1}])\sim(\alpha_{2},[a_{2},b_{2}])
\]
if there is an increasing homeomorphism $\tau:[a_{1},b_{1}]\rightarrow\lbrack
a_{2},b_{2}]$ such that $\alpha_{2}\circ\tau=\alpha_{1}.$ For two surfaces
$(f_{1},\overline{U_{1}})$ and $(f_{2},\overline{U_{2}})$, we call they
equivalent and write $(f_{1},\overline{U_{1}})\sim(f_{2},\overline{U_{2}})\ $
if there is an orientation-preserving homeomorphism (OPH) $\phi:\overline
{U_{1}}\rightarrow\overline{U_{2}}$ such that $f_{2}\circ\phi=f_{1}.$
\label{holomo}

By our convention , for any covering surface $\Sigma=(f,\overline{U})$ over
$S,$ $f$ is the restriction of an OPCOFOM $\widetilde{f}$ defined on a Jordan
neighborhood $V$ of $\overline{U}$. By Theorem \ref{st}, there is a
self-homeomorphism $h$ of $V$ such that $f\circ h$ is holomorphic on $V$.
Thus, $\Sigma$ is equivalent to the covering surface $(g,\overline{U_{1}}),$
where $\overline{U_{1}}=h^{-1}(\overline{U})$ and $g=f\circ h$ is holomorphic
on $\overline{U_{1}}$. For any two equivalent surfaces $\Sigma_{1}%
=(f_{1},\overline{U_{1}})$ and $\Sigma_{2}=(f_{2},\overline{U_{2}})$, we have
$A(f_{1},\overline{U_{1}})=A(f_{2},\overline{U_{2}})$, $L(f_{1},\partial
\overline{U_{1}})=L(f_{2},\partial\overline{U_{2}})$ and $\overline{n}%
(f_{1},E_{q})=\overline{n}(f_{2},E_{q})$ for a fixed set $E_{q}$. Thus we can
identify the equivalent surfaces and for any surface $\Sigma=(f,\overline{U}%
)$, we may assume $f$ is holomorphic in $U$.

Theorem \ref{st} is a powerful tool to explain the connection between OPCOFOM
and the holomorphic map. The following lemma is a consequence of Theorem
\ref{st}. We shall denote by $D(a,\delta)$ the disk on $S$ with center $a$ and
spherical radius $\delta$. Then $\Delta\subset S$ is the disk $D(0,\pi/2)$.

\begin{lemma}
\label{cov-1}Let $(f,\overline{U})\ $be a surface, $U$ be a domain on
$\mathbb{C}$ bounded by a finite number of Jordan curves and $\left(
f,\partial U\right)  $ is consisted of a finite number of simple circular arcs
and let $q\in f(\overline{U}).$ Then, for sufficiently small disk
$D(q,\delta)$ on $S\ $with $\delta<\pi/2,$ $f^{-1}(\overline{D(q,\delta)}%
)\cap\overline{U}$ is a finite union of disjoint sets $\{\overline{U_{j}%
}\}_{1}^{n}$ in $\overline{U},$ where each $U_{j}$ is a Jordan domain in $U,$
such that for each $j,$ $\overline{U_{j}}\cap f^{-1}(q)$ contains exactly one
point $x_{j}$ and (A) or (B) holds:

(A) $x_{j}\in U_{j}\subset\overline{U_{j}}\subset U$ and $f:\overline{U_{j}%
}\rightarrow\overline{D(q,\delta)}$ is a BCCM such that $x_{j}$ is the only
possible branch point.

(B) $x_{j}\in\partial U,$ $f$ is locally homeomorphic on $\overline{U_{j}%
}\backslash\{x_{j}\},$ and when $\left(  f,\overline{U}\right)  \in
\mathcal{F}, $ the following conclusions (B1)--(B3) hold:

(B1) The Jordan curve $\partial U_{j}$ has a partition $\alpha_{1}\left(
p_{1},x_{j}\right)  +\alpha_{2}\left(  x_{j},p_{2}\right)  +\alpha_{3}\left(
p_{2},p_{1}\right)  $ such that $\alpha_{1}+\alpha_{2}=\left(  \partial
U\right)  \cap\partial U_{j}$ is an arc of $\partial U,$ $\alpha_{3}^{\circ
}\subset U,$ $c_{j}=\left(  f,\alpha_{j}\right)  $ is an SCC arc for $j=1,2,$
and $c_{3}=\left(  f,\alpha_{3}\right)  $ is a locally SCC\footnote{The
condition $\delta<\pi/2$ makes $\partial D\left(  q,\delta\right)  $ strictly
convex, and it is possible that $\left(  f,\alpha_{3}^{\circ}\right)  $ may
describes $\partial D\left(  q,\delta\right)  $ more than one round, and in
this case $\left(  f,\alpha_{3}^{\circ}\right)  $ is just locally SCC.} arc in
$\partial D\left(  q,\delta\right)  $ from $q_{2}=f\left(  p_{2}\right)  $ to
$q_{1}=f\left(  p_{1}\right)  $. Moreover, $f$ is homeomorphic in a
neighborhood of $\alpha_{j}\backslash\{x_{j}\}$ in $\overline{U}$ for $j=1,2,$
and
\[
\partial\left(  f,\overline{U_{j}}\right)  =\left(  f,\partial U_{j}\right)
=c_{1}+c_{2}+c_{3}.
\]

(B2) The interior angle of $\left(  f,\overline{U_{j}}\right)  $ at $p_{1}$
and $p_{2}$ are both contained in $[\frac{7\pi}{16},\frac{9\pi}{16}].$

(B3) There exists a rotation $\psi$ of $S$ with $\psi(q)=0$ such that one of
the following holds:

(B3.1) $q_{1}=q_{2},\left(  f,\alpha_{1}\right)  =\overline{q_{1}q}%
=\overline{q_{2}q}=-\left(  f,\alpha_{2}\right)  ,$ say, $\left(  f,\alpha
_{1}+\alpha_{2}\right)  =\overline{q_{1}q}+\overline{qq_{1}},$ and $\left(
\psi\circ f,\overline{U_{j}}\right)  $ is equivalent to the
surface\footnote{Here $\delta z^{\omega_{j}}$ is regarded as the mapping
$z\mapsto\delta z^{\omega_{j}}\in S,z\in\overline{\Delta^{+}},$ via the
stereographic projection $P.$} $\left(  \delta z^{\omega_{j}}:\overline
{\Delta^{+}}\right)  $ on $S$ so that
\[
\left(  \delta z^{\omega_{j}},[-1,1]\right)  =\overline{a_{\delta}%
,0}+\overline{0,a_{\delta}},
\]
where $\omega_{j}$ is an even positive integer and $a_{\delta}\in\left(
0,1\right)  $ with $d\left(  0,a_{\delta}\right)  =\delta.$

(B3.2) $q_{1}\neq q_{2},$ as sets $c_{1}\cap c_{2}=\{q\},$ and $\left(
\psi\circ f,\overline{U_{j}}\right)  $ is equivalent to the the surface
$\left(  F,\overline{\Delta^{+}}\cup\overline{\mathfrak{D}_{1}^{\prime}}%
\cup\overline{\mathfrak{D}_{2}^{\prime}}\right)  $ so that the following holds.

(B3.2.1) $\mathfrak{D}_{1}^{\prime}=\mathfrak{D}^{\prime}\left(
\overline{-1,0},\theta_{1}\right)  \ $and $\mathfrak{D}_{2}^{\prime
}=\mathfrak{D}^{\prime}\left(  \overline{0,1},\theta_{2}\right)  $, such that
for each $j=1,2,\theta_{j}\in\lbrack0,\frac{\pi}{4}].$ Moreover $\theta_{1}=0$
(or $\theta_{2}=0$) when $c_{1}=\overline{q_{1}q}$ (or $c_{2}=\overline
{qq_{2}}$), and in this case $\mathfrak{D}_{1}^{\prime}=\emptyset$ (or
$\mathfrak{D}_{2}^{\prime}=\emptyset$). See Definition \ref{lune-lens} for the
notation $\mathfrak{D}^{\prime}\left(  \cdot,\cdot\right)  .$

(B3.2.2) $\left(  F,\overline{\Delta^{+}}\right)  $ is the surface $T=\left(
\delta z^{\omega_{j}},\overline{\Delta^{+}}\right)  $, where $\omega_{j}\ $is
a positive number which is not an even number and even may not be an integer,
$\left(  F,\overline{\mathfrak{D}_{1}^{\prime}}\right)  $ is the lune
$\psi\left(  \overline{\mathfrak{D}^{\prime}\left(  \overline{q_{1}q}%
,c_{1}\right)  }\right)  $ and $\left(  F,\overline{\mathfrak{D}_{2}^{\prime}%
}\right)  $ is the lune $\psi\left(  \overline{\mathfrak{D}^{\prime}\left(
\overline{qq_{2}},c_{2}\right)  }\right)  .$ That is to say, $\left(
f,\overline{U_{j}}\right)  $ is obtained by sewing\label{sew40} the sector
$\psi^{-1}\left(  T\right)  $ with center angle\footnote{This angle maybe
larger than $2\pi$ as the sector $\left(  z^{3},\overline{\Delta^{+}}\right)
$ at $0$.} $\omega_{j}\pi,$ and the closed lunes $\overline{\mathfrak{D}%
^{\prime}\left(  \overline{q_{1}q},c_{1}\right)  }$ and $\overline
{\mathfrak{D}^{\prime}\left(  \overline{qq_{2}},c_{2}\right)  }$ along
$\overline{q_{1}q}$ and $\overline{qq_{2}}$ respectively.
\end{lemma}

\begin{proof}
(A) follows from Stoilow's theorem directly when $x_{j}\in U$. (B) follows
from (A) and the assumption $\left(  f,\overline{U}\right)  \in\mathcal{F}$,
by considering the extension of $f$ which is an OPCOFOM in a neighborhood of
$x_{j}$ in $\mathbb{C}.$
\end{proof}

\begin{remark}
\label{notation}We list more elementary conclusions deduced from the previous
lemma directly and more notations. Let $\Sigma=(f,\overline{U})\in\mathcal{F}%
$, $q\in f(\overline{U}),$ $\delta,$ $x_{j},$ $U_{j}$ and $\alpha_{1}%
+\alpha_{2}$ be given as in Lemma \ref{cov-1}.

(A) If for some $j,$ $x_{j}\in\Delta,$ then by Lemma \ref{cov-1} (A), $f$ is a
BCCM in the neighborhood $U_{j}$ of $x_{j}$ in $\Delta,$ and the order
$v_{f}(x_{j})$ of $f$ at $x_{j}$ is well defined, which is a positive integer,
and $f$ is a $v_{f}(x_{j})$-to-$1$ CCM on $U_{j}\backslash\{x_{j}\}.$

(B) If for some $j,x_{j}\ $is contained in $\partial\Delta,$ then, using
notations in Lemma \ref{cov-1} (B), there are two possibilities:

(B1) $q_{1}=q_{2},$ the interior angle of $\Sigma$ at $x_{j}$ equals
$\omega_{j}\pi,$ and the order $v_{f}\left(  x_{j}\right)  $ is defined to be
$\omega_{j}/2,$ which is a positive integer.

(B2) $q_{1}\neq q_{2},c_{1}+c_{2}$ is a simple arc from $q_{1}$ to $q,$ and
then to $q_{2}$. In this case the interior angle of $\Sigma$ at $x_{j}$ equals
$\omega_{j}\pi+\varphi_{1}+\varphi_{2},$ where $\varphi_{1}$ and $\varphi_{2}$
are the interior angles of $\mathfrak{D}^{\prime}\left(  \overline{q_{1}%
q},c_{1}\right)  $ and $\mathfrak{D}^{\prime}\left(  \overline{qq_{2}}%
,c_{2}\right)  $ at the cusps, and we defined the order of $f$ at $x_{j}$ to
be the least integer $v_{f}\left(  x_{j}\right)  $ with $v_{f}\left(
x_{j}\right)  \geq\left(  \omega_{j}\pi+\varphi_{1}+\varphi_{2}\right)  /2\pi
$. Since $\omega_{j}\pi+\varphi_{1}+\varphi_{2}\geq\omega_{j}\pi>0,$ we have
$v_{f}\left(  x_{j}\right)  \geq1$ and $f$ is injective on $U_{j}%
\backslash\left\{  c_{1}+c_{2}\right\}  $ iff $v_{f}\left(  x_{j}\right)  =1.$
This is also easy to see by Corollary \ref{cov-2} (v).

(C) The number $v_{f}(x_{j})$ can be used to count path lifts with the same
initial point $x_{j}:$ when $x_{j}\in\Delta,$ any sufficiently short line
segment on $S$ starting from $q=f(x_{j})\ $has exactly $v_{f}\left(
x_{j}\right)  $ $f$-lifts starting from $x_{j}$ and disjoint in $\Delta
\backslash\{x_{j}\};$ and when $x_{j}\in\partial\Delta,$ for each arc $\beta$
of the two sufficiently short arcs of $\partial\Delta$ with initial point
$x_{j}$, $(f,\beta)$ is simple and has exactly $v_{f}(x_{j})-1$ $f$-lifts
$\left\{  \beta_{j}\right\}  _{j=1}^{v_{f}\left(  x_{j}\right)  -1}$ with the
same initial point $x_{j},$ $\beta_{j}\backslash\{x_{j}\}\subset\Delta$ for
each $j$ and they are disjoint in $\Delta.$ This is also easy to see by
Corollary \ref{cov-2} (v).

(D) A point $x\in\overline{U}$ is called a \emph{branch point} of $f$ (or
$\Sigma$) if $v_{f}(x)>1,$ or otherwise called a regular point if
$v_{f}\left(  x\right)  =1.$ We denote by $C_{f}$ the set of all branch points
of $f,$ and $CV_{f}$ the set of all branch values of $f.$ For a set
$A\subset\overline{U},$ we denote by $C_{f}\left(  A\right)  =C_{f}\cap A$ the
set of branch points of $f$ located in $A,$ and by $CV_{f}(K)=CV_{f}\cap
K\label{CVF}$ the set of branch values of $f$ located in $K\subset S.$ We will
write%
\[
C_{f}^{\ast}\left(  A\right)  =C_{f}\left(  A\right)  \backslash f^{-1}%
(E_{q})\mathrm{\ and\ }C_{f}^{\ast}=C_{f}\backslash f^{-1}(E_{q})=C_{f}\left(
\overline{U}\right)  \backslash f^{-1}(E_{q}).
\]

(E) For each $x\in\overline{U},$ $b_{f}\left(  x\right)  =v_{f}\left(
x\right)  -1\ $is called the branch number of $f$ at $x,$ and for a set
$A\subset\overline{U}$ we write $B_{f}\left(  A\right)  =\sum_{x\in A}%
b_{f}\left(  x\right)  .$ Then we have $b_{f}\left(  x\right)  \neq0$ iff
$C_{f}\left(  x\right)  =\{x\},$ and $B_{f}\left(  A\right)  =\sum_{x\in
C_{f}\left(  A\right)  }b_{f}(x).$ We also define%
\[
B_{f}^{\ast}\left(  A\right)  =B_{f}\left(  A\backslash f^{-1}(E_{q})\right)
.
\]
Then $B_{f}^{\ast}\left(  A\right)  \geq0,$ equality holding iff $C_{f}^{\ast
}\left(  A\right)  =\emptyset.$ When $A=\overline{U}$ is the domain of
definition of $f,$ we write%
\[
B_{f}=B_{f}\left(  \overline{U}\right)  \mathrm{\ and\ }B_{f}^{\ast}%
=B_{f}^{\ast}\left(  \overline{U}\right)  .\label{BBC}%
\]

\end{remark}

Now we can state a direct Corollary to Lemma \ref{cov-1}.

\begin{corollary}
\label{cov-2}\label{230330-22:41 copy(1)}Let $\Sigma=\left(  f,\overline
{U}\right)  \in\mathcal{F}$ and let $\left(  x_{1},U_{1}\right)  $ be a disk
of $\Sigma$ with radius $\delta_{1}.$ Then, the following hold.

(i) $f$ is locally homeomorphic on $\overline{U_{1}}\backslash\{x_{1}\}$; and
if $\left(  x_{1},U_{1}^{\prime}\right)  $ is another disk of $\Sigma$ with
radius $\delta_{1}^{\prime}>\delta_{1},$ then $\overline{U_{1}}\subset
U_{1}^{\prime},$ whether $x_{1}\ $is in $\partial U$ or $U$.

(ii) If $f$ is homeomorphic in some neighborhood of $x_{1}$ in $\overline{U}$
(which may be arbitrarily small), or if $f$ locally homeomorphic on
$\overline{U}$, then the disk $\left(  x_{1},\overline{U_{1}}\right)  $ is a
one sheeted closed domain of $\Sigma,$ say, $f$ restricted to $\overline
{U_{1}}$ is a homeomorphism onto $f\left(  \overline{U_{1}}\right)  .$

(iii) For each $x_{2}\in U_{1}\backslash\{x_{1}\},$ any closed disk $\left(
x_{2},\overline{U_{2}}\right)  $ of $\Sigma$ is a one sheeted closed domain of
$\Sigma,$ moreover, $\overline{U_{2}}\subset U_{1}$ when the radius of
$\left(  x_{2},U_{2}\right)  $ is smaller than $\delta-d\left(  f(x_{1}%
),f(x_{2})\right)  .$

(iv) If $x_{1}\in\partial U$, $f$ is regular at $x_{1}$ and $\left(
f,\partial U\right)  $ is circular near $x_{1},$ then $\left(  f,\overline
{U_{1}}\right)  $ is a \emph{convex} and one sheeted closed domain of
$\Sigma,$ which is in fact the closed lens $\overline{\mathfrak{D}\left(
I,c_{1},c_{1}^{\prime}\right)  }$, where $c_{1}$ and $c_{1}^{\prime}$ are
circular subarcs of $\partial\Sigma$ and the circle $\partial D\left(
f(x_{1}),\delta_{1}\right)  ,$ $I$ is the common chord, and the three paths
$c_{1},-c_{1}^{\prime},I$ have the same initial point. Moreover, if
$\partial\Sigma$ is straight at $x_{1},$ then $f(\overline{U_{1}}%
)=\overline{\mathfrak{D}^{\prime}\left(  -I,c_{1}^{\prime}\right)  }%
=\overline{\mathfrak{D}^{\prime}\left(  -c_{1},c_{1}^{\prime}\right)  }\ $is
half of the disk $\overline{D(f(x_{1}),\delta_{1})}$ on the left hand side of
diameter $c_{1}$ (see Definition \ref{lune-lens} for lenses and lunes).

(v) For any $x\in\overline{U_{1}},$ there exists a path $I\left(
x_{1},x\right)  $ in $\overline{U_{1}}$ from $x_{1}$ to $x$ such that
$I\left(  x_{1},x\right)  $ is the unique $f$-lift of $\overline{f(x_{1}%
)f(x)}.$ That is to say, $\left(  f,\overline{U_{1}}\right)  $ can be foliated
by the family of straight line segments $\{\left(  f,I\left(  x_{1},x\right)
\right)  :x\in\partial U_{1}\}$ which are disjoint in $\overline{U_{1}%
}\backslash\{x_{1}\}.$

(vi) For each $x\in\partial U,$ the interior angle of $\Sigma$ at $x$ is positive.
\end{corollary}

Lemma \ref{cov-1} also implies a criterion of regular point.

\begin{lemma}
\label{inj}Let $(f,\overline{U})\in\mathcal{F}.$ Then the following hold.

(A) For each $a\in\overline{U},f$ restricted to some neighborhood of $p$ in
$\overline{U}$ is a homeomorphism if one of the following alternatives holds.

(A1) $p\in U$ and $p$ is a regular point of $f.$

(A2) $p\in\partial U,$ $p$ is a regular point of $f$ and $(f,\partial U)$ is
simple in a neighborhood of $p$ on $\partial U$.

(B) For any SCC arc $\left(  f,\alpha\right)  $ of $\partial\Sigma=\left(
f,\partial U\right)  ,$ $f$ restricted to a neighborhood of $\alpha^{\circ}$
in $\overline{U}$ is a homeomorphic if and only if $h$ has no branch point on
$\alpha^{\circ}.$ Here $\alpha^{\circ}$ means the interior of the arc $\alpha$.
\end{lemma}

The hypothesis in condition (A2) that $(f,\partial U)$ is simple cannot be
ignored. See the following example.

\begin{example}
Take $f(z)=z^{2}$ for any $z\in\overline{\Delta^{+}}$. then $f$ is regular at
$z=0$ but not injective in any neighborhood of $0$ in $\overline{\Delta^{+}}$.
\end{example}

The following lemma shows that how to sew two surfaces together into one
surface along the equivalent curves. This is an important tool in Section 4.

\begin{lemma}
\label{patch}\label{last}
For $j=1,2,$ let $\Sigma_{j}=(f_{j},\overline{U_{j}})$ be a surface and let
$\alpha_{j}=\alpha_{j}\left(  x_{j1},x_{j2}\right)  $ be a proper arc of
$\partial U_{j}$ such that $\left(  f_{j},\alpha_{j}\right)  $ is a simple arc
with distinct endpoints. If
\begin{equation}
(f_{1},\alpha_{1})\sim-(f_{2},\alpha_{2}), \label{pp13}%
\end{equation}
then $(f_{1},\overline{U_{1}})$ and $(f_{2},\overline{U_{2}})\ $can be
\emph{sewn along}\label{sew38} $\left(  f_{1},\alpha_{1}\right)  $ to become a
surface $\Sigma_{3}=(f_{3},\overline{\Delta})$, such that the following hold:

(i) There exist orientation-preserving homeomorphisms (OPHs) \label{OPH}
$h_{1}:\overline{U_{1}}\rightarrow\overline{\Delta^{+}}$ and $h_{2}%
:\overline{U_{2}}\rightarrow\overline{\Delta^{-}},$ called
\emph{identification mappings} (\emph{IM}s), such that
\begin{equation}
(h_{1},\alpha_{1})\sim\lbrack-1,1]\sim-\left(  h_{2},\alpha_{2}\right)
=\left(  h_{2},-\alpha_{2}\right)  , \label{pp5}%
\end{equation}%
\begin{equation}
f_{1}\circ h_{1}^{-1}\left(  x\right)  =f_{2}\circ h_{2}^{-1}(x),\forall
x\in\lbrack-1,1], \label{pp6}%
\end{equation}
and%
\begin{equation}
f_{3}(z)=\left\{
\begin{array}
[c]{l}%
f_{1}\circ h_{1}^{-1}(z),z\in\overline{\Delta^{+}},\\
f_{2}\circ h_{2}^{-1}(z),z\in\overline{\Delta^{-}}\backslash\lbrack-1,1],
\end{array}
\right.  \label{pp7}%
\end{equation}
is a well defined OPCOFOM, and we have the equivalent relations
\[
(f_{3},\overline{\Delta^{+}})\sim(f_{1},\overline{U_{1}}),(f_{3}%
,\overline{\Delta^{-}})\sim(f_{2},\overline{U_{2}}),
\]%
\[
\partial\Sigma_{3}=\left(  f_{3},\left(  \partial\Delta\right)  ^{+}\right)
+\left(  f_{3},\left(  \partial\Delta\right)  ^{-}\right)  \sim\left(
f_{1},\left(  \partial U_{1}\right)  \backslash\alpha_{1}^{\circ}\right)
+\left(  f_{2},\left(  \partial U_{2}\right)  \backslash\alpha_{2}^{\circ
}\right)  ,
\]
and
\[
(f_{3},[-1,1])\sim(f_{1},\alpha_{1})\sim(f_{2},-\alpha_{2}).
\]

(ii)
\[
L(\partial\Sigma_{3})=L(\partial\Sigma_{1})+L(\partial\Sigma_{2}%
)-2L(f_{2},\alpha_{2}),
\]%
\[
A(\Sigma_{3})=A\left(  \Sigma_{1}\right)  +A(\Sigma_{2}),
\]%
\[
\overline{n}\left(  \Sigma_{3}\right)  =\overline{n}\left(  \Sigma_{1}\right)
+\overline{n}\left(  \Sigma_{2}\right)  +\#\left(  \gamma^{\circ}\cap
E_{q}\right)  ,
\]
and%
\begin{equation}
R(\Sigma_{3})=R(\Sigma_{1})+R(\Sigma_{2})-4\pi\#\left(  \gamma^{\circ}\cap
E_{q}\right)  . \label{RRR-1}%
\end{equation}

(iii) $z\in C_{f_{3}}\left(  \overline{\Delta}\backslash\{-1,1\}\right)  $ if
and only if $h_{1}^{-1}(z)\in C_{f_{1}}\left(  \overline{U_{1}}\backslash
\partial\alpha_{1}\right)  \mathrm{\ or\ }h_{2}^{-1}(z)\in C_{f_{2}}\left(
\overline{U_{2}}\backslash\partial\alpha_{2}\right)  $. In particular, if
$f_{1}(\partial\alpha_{1})\subset E_{q},$ then $f_{2}(\partial\alpha
_{2})\subset E_{q}$ and in addition
\[
CV_{f_{3}}(S\backslash E_{q})=CV_{f_{1}}(S\backslash E_{q})\cup CV_{f_{2}%
}(S\backslash E_{q}).
\]

\end{lemma}


\begin{proof}
The conclusion (i) in fact gives a routine how to sew \label{sew36}$\Sigma
_{1}$ and $\Sigma_{2}$. By (\ref{pp13}), there exists an OPH\footnote{Note
that $-\alpha_{2}$ is the same path with opposite direction, not the set
$\{-y:y\in\alpha_{2}\}.$} $\varphi:\alpha_{1}\rightarrow-\alpha_{2}$ such
that
\[
\left(  f_{1},\alpha_{1}\right)  =\left(  f_{2}\circ\varphi,\alpha_{1}\right)
,
\]
that is
\[
f_{2}\left(  \varphi(x)\right)  \equiv f_{1}(x),\forall x\in\alpha_{1}.
\]

Let $h_{1}:\overline{U_{1}}\rightarrow\overline{\Delta^{+}}$ be any OPH such
that $h_{1}(\alpha_{1})=[-1,1].$ Then let $h_{2}:\overline{U_{2}}%
\rightarrow\overline{\Delta^{-}}$ be an OPH such that
\begin{equation}
h_{2}\left(  y\right)  \equiv h_{1}\left(  \varphi^{-1}(y\right)  ),\forall
y\in\alpha_{2}. \label{pp14}%
\end{equation}
In fact, $h_{2}|_{\alpha_{2}}$ defined by (\ref{pp14}) is an OPH from
$\alpha_{2}$ onto $[1,-1]$ and can be extended to be an OPH $h_{2}$ from
$\overline{U_{2}}$ onto $\overline{\Delta^{-}}.$ The pair of $h_{1}$ and
$h_{2}$ are the desired mappings satisfying (i). Then (ii) is trivial to verify.

To prove (iii) we may assume that $\Sigma_{1}$ and $\Sigma_{2}$ are the
surfaces $\Sigma_{\pm}=(f_{\pm},\overline{\Delta^{\pm}})$ such that $f_{\pm}$
agree on $[-1,1],$ and then $f_{3}$ defined by $f_{\pm}$ on $\overline
{\Delta^{\pm}}\ $is an OPLM. When $x\in(-1,1)$ is a branch point of $f_{+}$ or
$f_{-},$ $x$ is obviously a branch point of $f_{3}.$ Since $f_{\pm}$ are the
restrictions of $f_{3}$ to $\overline{\Delta^{\pm}},$ and $\left(
f_{+},[-1,1]\right)  $ and $\left(  f_{-},[1,-1]\right)  $ are simple with
opposite direction, if $x\in\left(  -1,1\right)  $ is not a branch point of
$f_{\pm},$ then $f_{\pm}$ are homeomorphisms in neighborhoods $V^{\pm}$ of $x$
in $\overline{\Delta^{\pm}}$ and the simple arc $\left(  f_{3},[-1,1]\right)
$ separates $f_{+}(V^{+}\backslash\lbrack-1,1])$ and $f_{-}(V^{-}%
\backslash\lbrack-1,1])$, and thus $f_{3}$ is homeomorphic on a neighborhood
of $x$ and so $x$ cannot be a branch point of $f_{3}.$ Therefore $x\in
C_{f_{3}}$ iff $x\in C_{f_{1}}\cup C_{f_{2}}.$ In consequence we have
$C_{f_{3}}\left(  \overline{\Delta}\backslash\{-1,1\}\right)  =C_{f_{1}%
}\left(  \overline{\Delta^{+}}\backslash\{-1,1\}\right)  \cup C_{f_{2}}\left(
\overline{\Delta^{-}}\backslash\{-1,1\}\right)  ,$ and (iii) follows.

\end{proof}

\begin{remark}
The condition $(f_{1},\alpha_{1})\sim(f_{2},-\alpha_{2})$ is crucial. Two
copies of the hemisphere $\overline{S^{+}}$ cannot be sewn along their common
edge $\overline{0,1}\subset S$ to become a surface in $\mathcal{F},$ but
$\overline{S^{+}}$ and $\overline{S^{-}}$, with natural edges $\overline{0,1}$
and $-\overline{0,1}=\overline{1,0}$ respectively, can be sewn along
$\overline{0,1}$ to become a surface in $\mathcal{F}$.
\end{remark}

Lemma \ref{patch} will be used frequently when we patch the covering surfaces.
The condition in this lemma that $\alpha_{j}$ are proper arcs of $\partial
U_{j}$ can be replaced by that one of the curves $\alpha_{1}$ and $\alpha_{2}$
is proper. Indeed, if only $\alpha_{1}$ is proper, then we can find partitions
$\alpha_{1}=\alpha_{11}+\alpha_{12}$ and $\alpha_{2}=\alpha_{21}+\alpha_{22}$
so that $\alpha_{11}\sim\alpha_{21}$ and $\alpha_{12}\sim\alpha_{22}$, and we
can use Lemma \ref{patch} twice.

For a surface $\Sigma=(f,\overline{U})$ and an arc $\beta$ on $S$, we define
the lift of $\beta$ by $f$ as an arc $\alpha$ in $U$ satisfying that $(f,
\alpha) \sim\beta$. By Remark \ref{notation}, for any point $p\in U$, a
sufficiently short path $\beta$ from $f(p)$ has exactly $v_{f}(p)$ lifts from
$p$.


\begin{lemma}
\label{glue}\label{glue2}\label{cut-3}Let $\Sigma=(f,\overline{U}%
)\in\mathcal{F}$, $p_{0}\in\overline{U}$ and $\beta$ a polygonal simple path
on $S$ with distinct endpoints. Assume that $\beta$ has two $f$-lifts
$\alpha_{j}$, $j=1,2,$ with initial point $p_{0},$ such that $\alpha
_{1}^{\circ}\cap\alpha_{2}^{\circ}=\emptyset$. Then

(i) $f(\overline{U})=S$ if $\alpha_{1}$ and $\alpha_{2}$ terminate at the same
point; moreover, $\left(  f,\overline{U}\right)  $ can be sewn along $\left(
f,\alpha_{1}\right)  \sim\left(  f,\alpha_{2}\right)  $ becoming a closed
surface $\left(  f_{0},S\right)  .$

(ii) If $\alpha_{1}\cup\alpha_{2}$ is a proper arc of $\partial U$, then the
following (ii1) and (ii2) hold.

\quad(ii1) $(f,\overline{U})$ can be sewn along $\beta$ to become a covering
surface $\Sigma_{1}=(g,\overline{\Delta})\in\mathcal{F},$ such that
\begin{align*}
A(g,\Delta)  &  =A(f,U),\\
L(g,\partial\Delta)  &  =L(f,\partial U)-L(f,\alpha_{1}\cup\alpha_{2})\\
&  =L(f,\partial U)-2L(\beta),\\
\overline{n}\left(  \Sigma_{1},E_{q}\right)   &  =\overline{n}\left(
\Sigma,E_{q}\right)  +\#\{f\left(  \left(  \alpha_{1}\cup\alpha_{2}\right)
^{\circ}\right)  \cap E_{q}\}\\
&  =\overline{n}\left(  \Sigma,E_{q}\right)  +\#\{[\beta^{\circ}\cup
\{f(p_{0})\}]\cap E_{q}\}.
\end{align*}

\quad(ii2) $\left(  f,N\right)  $ and $\left(  g,N_{1}\right)  $ are
equivalent surfaces, where $N=\overline{U} \backslash\left(  \alpha_{1}%
\cup\alpha_{2}\right)  $ and $N_{1}=\overline{\Delta}\backslash[0,1]$, and
thus $(f, (\partial U) \backslash\left(  \alpha_{1}\cup\alpha_{2}\right)
^{\circ} )$, regarded as a closed curve, is equivalent to $(g,\partial
\Delta).$

(iii) If $\alpha_{1}\subset\partial U,\alpha_{2}\backslash\left\{
p_{0}\right\}  \subset U,$ the terminal point of $\beta$ is in $E_{q}$ but all
other points of $\beta$ are outside of $E_{q},$ then there exists a covering
surface $\Sigma_{1}$ such that
\begin{align*}
R(\Sigma_{1})  &  =R(\Sigma)+4\pi,\\
L(\partial\Sigma_{1})  &  =L(\partial\Sigma),
\end{align*}
and $\partial\Sigma_{1}$ is equivalent to the closed curve $\partial\Sigma.$
\end{lemma}

\begin{proof}
We first consider that $\alpha_{1}$ and $\alpha_{2}$ have the same terminal
point. Then they bound a Jordan domain $V$ in $U$, and thus $f(\overline
{U})\supset f(\overline{V})=S$ by the argument principle. One the other hand,
we can sew the closed domain $\overline{V}$ by identifying $\alpha_{1}$ and
$\alpha_{2}$ so that the points $x\in\alpha_{1}$ and $y\in\alpha_{2}$ are
identified if and only if $f(x)=f(y),$ to obtain the surface $S.$ Then
$(f,\overline{V})$ becoming a closed surface $\left(  f_{0},S\right)  $. So
(i) holds true.

To prove (ii), we may assume that $\alpha_{1}$ and $\partial U$ have the same
orientation. Then $\alpha_{2}$ and $\partial U$ have opposite orientations,
and there exists an orientation-preserving homeomorphism $\phi:\overline
{\Delta^{+}}\rightarrow\overline{U}$ with $\phi([0,1])=\alpha_{1},$
$\phi([-1,0])=-\alpha_{2}$ and $f\circ\phi(x)=f\circ\phi(-x)$ for any
$x\in\lbrack0,1].$ Let $g(z)=f\circ\phi(re^{i\theta/2})$ with $z=re^{i\theta
}\in\overline{\Delta}$, $\theta\in[0,2\pi]$. Then, $\Sigma_{1}=(g,\overline
{\Delta})\in\mathcal{F}$ is a covering surface which satisfies the conclusion
of (ii).

To prove (iii), let $h$ be an OPCOFOM map from $\overline{\Delta^{+}}$ onto
$\overline{U}$ such that $h$ restricted to $\overline{\Delta^{+}}%
\setminus\lbrack-1,1]$ is a homeomorphism onto $\overline{U}\backslash
\alpha_{2}$. Moreover, we assume that $h$ maps both $[-1,0]$ and $[0,1]$
homeomorphically onto $\alpha_{2}$ with opposite direction, and maps the arc
$\alpha_{1}^{\prime}=\left\{  e^{\sqrt{-1}\theta}:\theta\in\lbrack0,\frac{\pi
}{2}]\right\}  $ homeomorphically onto $\alpha_{1}$. Then we consider the
surface $\Sigma^{\prime}=\left(  f\circ h,\overline{\Delta^{+}}\right)  $.
After rescaling the parameter of $\partial\Sigma^{\prime}$, we may assume that
$\Sigma^{\prime}$ satisfies (ii), with $\alpha_{1}$ and $\alpha_{2}$ of (ii)
being replaced by $[0,1]$ and $\alpha_{1}^{\prime}.$ Then by identifying
$\alpha_{1}^{\prime}$ and $[0,1]$ as in (ii), we can sew $\Sigma_{1}^{\prime}$
to obtain a new surface $\Sigma_{1}.$ It is clear that $\overline{n}%
(\Sigma_{1},E_{q})=\overline{n}(\Sigma_{1}^{\prime},E_{q})=\overline{n}%
(\Sigma,E_{q})-1$, and thus $\Sigma_{1}$ satisfies (iii).
\end{proof}

\begin{lemma}
\label{Ri}(\cite{Ri} p. 32--35) Let $\Sigma=(f,\overline{\Delta}%
)\in\mathcal{F}$ and $\beta$ be a path on $S$ with initial point $q_{1}$.
Assume that $\alpha\subset\overline{\Delta}$ is an $f$-lift of some subarc of
$\beta$ from $q_{1}$, and $\alpha^{\circ}\subset\Delta.$ Then $\alpha$ can be
extended to an $f$-lift $\alpha^{\prime}$ of a longer subarc of $\beta$ with
$\alpha^{\prime\circ}\subset\Delta,$ such that either $\alpha^{\prime}$
terminates at a point on $\partial\Delta,$ or $\alpha^{\prime}$ is the
$f$-lift of the whole path $\beta$.
\end{lemma}

The following lemma is obvious, which states that two different interior
branch points can be exchanged.

\begin{lemma}
\label{move}Let $\Sigma=(f,\Sigma)\in\mathcal{F}$, $b\in\Delta$ be a branch
point of $f$ with $v_{f}(b)=d$, and $\delta>0$ be a sufficiently small number.
Then there exists a Jordan neighborhood $V$ of $b$ in $\Delta$ such that
$f:\overline{V}\to\overline{D(f(b),\delta)} $ is a $d$-to-1 BCCM so that $b$
is the unique branch point, and for any $y_{1}$ with $d(f(b),y_{1})<\delta
\ $and any $b_{1}\in V,$ there exists a surface $\Sigma=(f_{1},\overline
{\Delta})\in\mathcal{F}$ such that $f_{1}$ restricted to $\overline{\Delta
}\backslash V$ equals $f$ and $f_{1}:\overline{V}\to\overline{D(f(b),\delta)}
$ is a $d$-to-1 branched covering map such that $b_{1}$ becomes the unique
branch point of $f_{1}$ in $V$, $y_{1}=f_{1}(b_{1})$ and $v_{f_{1}}%
(b_{1})=v_{f}(b).$
\end{lemma}

The following results are essentially consequences of argument principle.

\begin{lemma}
\label{continue0}Let $(f,\overline{\Delta})\in\mathcal{F}$ and let $D$ be a
Jordan domain on $S$ such that $f^{-1}$ has a univalent branch $g$ defined on
$D.$ Then $g$ can be extended to a univalent branch of $f^{-1}$ defined on
$\overline{D}$.
\end{lemma}

The proof of this lemma is almost the same as that of Lemma 5.2 in \cite{Zh1}.

\begin{lemma}
\label{b-in}Let $D_{1}$ and $D_{2}$ be Jordan domains on $\mathbb{C}$ or $S$
and let $f:\overline{D_{1}}\rightarrow\overline{D_{2}}$ be a map such that
$f:\overline{D_{1}}\rightarrow f(\overline{D_{1}})$ is a homeomorphism. If
$f(\partial D_{1})\subset\partial D_{2},$ then $f(\overline{D_{1}}%
)=\overline{D_{2}}$.
\end{lemma}

\section{Removing branch points outside $f^{-1}(E_{q})$}

In this section, we will introduce the surgeries to remove branch points
outside $f^{-1}(E_{q})$. Before the key techniques, we remark some properties
of the partitions of covering surface $\Sigma=\left(  f,\overline{\Delta
}\right)  \in\mathcal{C}\left(  L,m\right)  .$

\begin{definition}
\label{round}Let $\Gamma=\left(  f,\partial\Delta\right)  $ be a closed curve
in $S$ which consists of a finite number of SCC arcs. We define $\mathfrak{L}%
\left(  \Gamma\right)  $ to be the minimal integer $m$ with the following
property: there exists closed arcs $\gamma_{j1},\gamma_{j2},j=1,\dots,m,$ such
that
\begin{align*}
\Gamma &  =\gamma_{02}=\gamma_{11}+\gamma_{12};\\
\gamma_{12}  &  =\gamma_{21}+\gamma_{22};\\
&  \ldots\\
\gamma_{m-1,2}  &  =\gamma_{m1},
\end{align*}
in which for each $j=1,2,\dots,m,$ $\gamma_{j1}$ is either a simple closed arc
of $\gamma_{j-1,2},$ or a folded path $I+\left(  -I\right)  $ where $I$ is a
maximal simple arc such that $I+\left(  -I\right)  $ is a folded arc of
$\gamma_{j-1,2}$. Note that the same closed curve $\gamma_{jk}$ may have
different initial point in different places.
\end{definition}

Note that $\mathfrak{L}\left(  \partial\Sigma\right)  <+\infty$ if $\Sigma
\in\mathcal{C}\left(  L,m\right)  $. The following examples give an intuitive
explanation of $\mathfrak{L}\left(  \Gamma\right)  $.

\begin{example}
(1) If $\Gamma$ is simple or $\Gamma=\overline{ab}+\overline{ba}$, then
$\mathfrak{L}\left(  \Gamma\right)  =1$. When $\Gamma$ is a point, we write
$\mathfrak{L}\left(  \Gamma\right)  =0.$

(2) For the closed curve $\Gamma$ in Figure \ref{f4.1} (1) we have
$\mathfrak{L}\left(  \Gamma\right)  =2.$

(3) The closed curve $\Gamma=ABCDEFGHIJKLMNOPQA$ in Figure \ref{f4.1} (2), in
which $CD,GH,KLM,LM,NO$ are five straight line segments on $S$ ($KLM$ is
straight), contains no simple closed arcs, but it contains four maximal folded
closed arcs $CDE$, $GHI,$ $LMN,$ and $NOP,$ and thus $\mathfrak{L}\left(
\Gamma\right)  =5.$
\end{example}

%

\begin{figure}
[ptb]
\begin{center}
\includegraphics[
height=2.0263in,
width=4.1632in
]%
{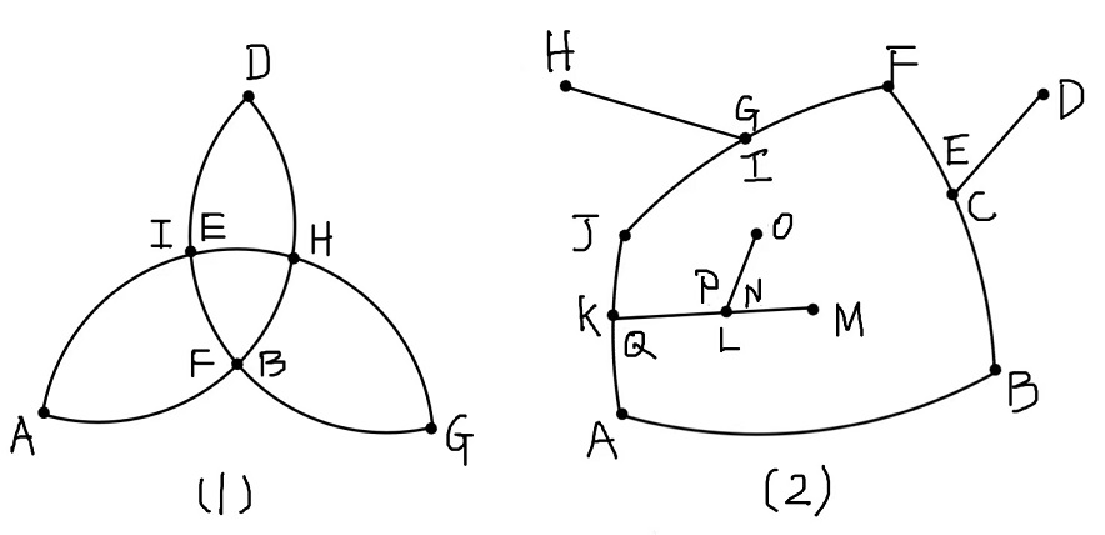}%
\caption{ }%
\label{f4.1}%
\end{center}
\end{figure}

The following lemma is trivial from Definition \ref{round}.

\begin{lemma}
\label{L-1}Let $\Gamma$ be a closed curve on $S$ which consists of a finite
number of simple circular arcs. If $\Gamma$ has a partition $\Gamma=\gamma
_{1}+\gamma_{2}$ such that $\gamma_{1}$ is a simple closed arc, or a maximal
folded closed arc, then
\[
\mathfrak{L}\left(  \gamma_{1}\right)  =\mathfrak{L}(\Gamma)-1.
\]

\end{lemma}

Now we start to introduce some lemmas to deal with the non-special branch
points, i.e. the branch points over $E_{q}$ (Correspondingly, the special
branch points mean the branch points over $E_{q}$). It is essentialy similar
to previous results in \cite{S-Z}. We first establish a lemma to remove the
non-special branch points in the interior, that is, the branch points in
$C_{f}^{\ast}(\Delta)$ (Recall Remark \ref{notation} for the notations).

\begin{lemma}
\label{move-in-1}Let $\Sigma=\left(  f,\overline{\Delta}\right)
\in\mathcal{C}^{\ast}\left(  L,m\right)  $ and assume that (\ref{210615})
holds. If $C_{f}^{\ast}(\Delta)\neq\emptyset$, then there exists a surface
$\Sigma_{1}=\left(  f_{1},\overline{\Delta}\right)  \in\mathcal{C}^{\ast
}\left(  L,m\right)  $ such that%
\begin{equation}
H(\Sigma_{1})\geq H\left(  \Sigma\right)  ,L(\partial\Sigma_{1})\leq
L(\partial\Sigma), \label{as2-4}%
\end{equation}
and%
\begin{equation}
B_{f_{1}}^{\ast}\left(  \Delta\right)  \leq B_{f}^{\ast}\left(  \Delta\right)
-1. \label{as2-5}%
\end{equation}
Moreover, $L(\partial\Sigma_{1})=L(\partial\Sigma)$ if and only if
$\partial\Sigma_{1}=\partial\Sigma,H(\Sigma_{1})\geq H\left(  \Sigma\right)  $
and $B_{f_{1}}^{\ast}\left(  \partial\Delta\right)  >B_{f}^{\ast}\left(
\partial\Delta\right)  .$
\end{lemma}

\begin{proof}
Corresponding to Definition \ref{sp}, we assume $\partial\Delta$ and
$\partial\Sigma$ have $\mathcal{C}^{\ast}(L,m)$-partitions
\begin{equation}
\partial\Delta=\alpha_{1}\left(  a_{1},a_{2}\right)  +\alpha_{2}(a_{2}%
,a_{3})+\dots+a_{m}\left(  a_{m},a_{1}\right)  \label{a1}%
\end{equation}
and
\begin{equation}
\partial\Sigma=c_{1}\left(  q_{1},q_{2}\right)  +c_{2}(q_{2},q_{3}%
)+\dots+c_{m}\left(  q_{m},q_{1}\right)  , \label{a2}%
\end{equation}
where $q_{j}=f\left(  a_{j}\right)  $ and $c_{j}\left(  q_{j},q_{j+1}\right)
=\left(  f,\alpha_{j}\left(  a_{j},a_{j+1}\right)  \right)  ,j=1,\dots,m$. By
definition of $\mathcal{C}^{\ast}(L,m)$, $f$ has no branch points in
$\alpha_{j}^{\circ}\cap f^{-1}(E_{q})$ for each $j=1,2,\dots,m.$

Let $p_{0}\in{C}_{f}^{\ast}\left(  \Delta\right)  $, say, $p_{0}$ is a
non-special branch point of $f$ with order $v$ and let $b_{0}=f(p_{0})$. Let
$b$ be a point in $E_{q}$ such that $d\left(  b_{0},b\right)  <\pi.$ Then
there is a polygonal simple path $\eta=\eta\left(  b_{0},b\right)  $ on $S$
from $b_{0}$ to $b$ such that

\begin{condition}
\label{conas1-1}$\eta^{\circ}\cap E_{q}=\emptyset$, $\eta^{\circ}\cap
\{q_{j}\}_{j=1}^{m}=\emptyset$, and $\eta^{\circ}$ contains no branch value of
$f.$ Moreover, $\eta^{\circ}$ intersects $\partial\Sigma$ perpendicularly and
$\beta\cap\partial\Sigma$ contains only finitely many points.
\end{condition}

We can extract a maximal subarc $\eta_{1}=\eta\left(  b_{0},b_{1}\right)
\ $of $\eta$ with $b_{1}\in\eta\backslash\{b_{0}\}$ such that $\eta_{1}$ has
$v$ distinct $f$-lifts $\beta_{l}=\beta_{l}\left(  p_{0},p_{l}\right)
,l=1,2,\dots,v,$ starting from $p_{0}$ with,%
\[
\beta_{l}^{\circ}\subset\Delta,\quad l=1,\dots,v,
\]
and that%
\[
\beta_{l_{1}}^{\circ}\cap\beta_{l_{2}}^{\circ}=\emptyset,\quad1\leq
l_{1}<l_{2}\leq v.
\]
The maximum of $\eta_{1}$ means that either $b_{1}=b\in E_{q},$ or some of
$\{p_{l}\}_{l=1}^{v}$ are contained in $\partial\Delta.$ We write
$A=\cup_{l=1}^{v}\beta_{l},$ and assume that $\beta_{l}$ are arranged
anticlockwise around the common initial point $p_{0}$. Thus, by Condition
\ref{conas1-1}, the following claim holds.

\begin{claim}
\label{conas3}\label{Cl1} (i) $\left\{  p_{l}\right\}  _{l=1}^{v}\subset
\Delta$ only if $b_{1}=b;$

(ii) $\left\{  p_{l}\right\}  _{l=1}^{v}\subset f^{-1}(E_{q})$ if and only if
$b_{1}=b\in E_{q};$

(iii) $p_{l_{1}}=p_{l_{2}}$ for some $l_{1}\neq l_{2}$ if and only if
$p_{l_{1}}$ is also a branch point and $b_{1}=b$.
\end{claim}

Then we have only five possibilities:

\noindent\textbf{Case (1).} $p_{l_{1}}=p_{l_{2}}$ for some $l_{1}\neq l_{2}$
and $p_{l_{1}}\in\partial\Delta.$

\noindent\textbf{Case (2).} $p_{l_{1}}=p_{l_{2}}$ for some $l_{1}\neq l_{2}$
and $p_{l_{1}}\in\Delta.$

\noindent\textbf{Case (3).} $p_{l},l=1,\dots,v,$ are distinct from each other
and $\{p_{l}\}_{l=2}^{v}\subset\Delta\ $but $p_{1}\in\partial\Delta.$

\noindent\textbf{Case (4).} $p_{l},l=1,\dots,v,$ are distinct from each other
and $\{p_{l}\}_{l=1}^{v}\subset\Delta.$

\noindent\textbf{Case (5).} $p_{l},l=1,\dots,v,$ are distinct from each other
and there exist some distinct $l_{1}\ $and $l_{2}$ such that both $p_{l_{1}}$
and $p_{l_{2}}$ are contained in $\partial\Delta.$

Now we will discuss the above cases one by one.

\noindent\textbf{Cases (1) and (2) cannot occur.}

Assume Case (1) occurs. By Claim \ref{conas3} (iii), $p_{l_{1}}$($=p_{l_{2}}$)
is a branch point in $f^{-1}(E_{q})$ and $b_{1}=b$. Since $\{\beta_{j}%
\}_{j=1}^{l}$ are arranged anticlockwise, we can derive that $p_{l_{1}%
}=p_{l_{2}}=p_{l_{1}+1}$, which means that there exist two adjacent $f$-lifts
$\beta_{l_{1}}$ and $\beta_{l_{1}+1}$ whose terminal points coincide. The
$f$-lift $\beta_{l_{1}}-\beta_{l_{1}+1}$ encloses a domain $D\subset\Delta$.
Thus we can cut $D$ off $\Delta$ along its boundary and sew the remained part
to obtain a new surface $\Sigma_{1}=\left(  f_{1},\overline{\Delta}\right)  $
such that $f_{1}=f$ in a neighborhood of $\partial\Delta\backslash\{p_{l_{1}%
}\}$ in $\overline{\Delta}.$ Then $\partial\Sigma_{1}=\partial\Sigma.$ We also
have $p_{l_{1}}\in\{a_{j}\}_{j=1}^{m}$ since $p_{l_{1}}$ is a branch point in
$f^{-1}(E_{q})$ and $f$ has no branch points in $\alpha_{j}^{\circ}\cap
f^{-1}(E_{q})$ for $j=1,2,\dots,m$. Thus (\ref{a1}) and (\ref{a2}) are
$\mathcal{C}^{\ast}\left(  L,m\right)  $ partitions of $\partial\Sigma_{1}$
which implies that $\Sigma_{1}\in\mathcal{C}^{\ast}\left(  L,m\right)  .$ By
Lemma \ref{cut-3} (i) we have:

\begin{claim}
\label{closed} $(f,\overline{D})$ can be sewn along its boundary
$(f,\beta_{l_{1}})\sim(f,\beta_{l_{2}})=\eta$, resulting a closed surface
$\Sigma_{0}=(f_{0},S)$.
\end{claim}

Assume the degree of $\Sigma_{0}$ is $d$, then by Riemann-Hurwitz formula we
have
\begin{align*}
\overline{n}\left(  \Sigma_{0},E_{q}\right)   &  =qd-\sum\limits_{x\in
f_{0}^{-1}(E_{q})}(v_{f_{0}}(x)-1)\\
&  \geq qd-\sum\limits_{x\in S}(v_{f_{0}}(x)-1)\\
&  \geq(q-2)d+2.
\end{align*}
On the other hand, $(\partial D)\cap f^{-1}(E_{q})=\{p_{l_{1}}\}$. Thus we
have $\overline{n}\left(  \Sigma_{1}\right)  =\overline{n}\left(
\Sigma\right)  -\overline{n}\left(  \Sigma_{0}\right)  +1\leq\overline
{n}\left(  \Sigma\right)  -(q-2)d-1$. It is clear that $A\left(  \Sigma
_{1}\right)  =A(\Sigma)-4d\pi$. Then we have
\begin{align*}
R(\Sigma_{1})  &  =\left(  q-2\right)  A\left(  \Sigma_{1}\right)
-4\pi\overline{n}\left(  \Sigma_{1},E_{q}\right) \\
&  \geq\left(  q-2\right)  A\left(  \Sigma\right)  -4\pi\left(  q-2\right)
d-4\pi\overline{n}\left(  \Sigma,E_{q}\right)  +4\pi(q-2)d+4\pi\\
&  =R(\Sigma)+4\pi,
\end{align*}
and thus $H(\Sigma_{1})=H(\Sigma)+\frac{4\pi}{L(\partial\Sigma_{1})},$ which
with $\partial\Sigma_{1}=\partial\Sigma$ and (\ref{210615}) implies a
contradiction:
\[
H_{L}\geq H(\Sigma_{1})\geq H_{L}-\frac{4\pi}{2L(\partial\Sigma)}+\frac{4\pi
}{L(\partial\Sigma_{1})}=H_{L}+\frac{7\pi}{2L(\partial\Sigma_{1})}.
\]
Thus Case (1) cannot occur.

Following the same arguments, one can show that Case (2) also cannot occur.

\noindent\textbf{Discussion of Case (5).}

Assume Case (5) occurs. Then the $f$-lift $-\beta_{l_{1}}+\beta_{l_{2}}$
divides $\Delta$ into two Jordan domains $\Delta_{1}$ and $\Delta_{2}$ with%
\[
\partial\Delta_{1}=-\beta_{l_{2}}+\beta_{l_{1}}+\tau_{1},\quad\partial
\Delta_{2}=-\beta_{l_{1}}+\beta_{l_{2}}+\tau_{2},
\]
where $\tau_{1}$ is the arc of $\partial\Delta$ from $p_{l_{1}}$ to $p_{l_{2}%
}$, and $\tau_{2}=\left(  \partial\Delta\right)  \backslash\tau_{1}^{\circ}.$
Then by Lemma \ref{cut-3}, we can sew $\left(  f,\overline{\Delta_{1}}\right)
$ and $\left(  f,\overline{\Delta_{2}}\right)  $ along $-\beta_{l_{1}}%
+\beta_{l_{2}}$ respectively to obtain two new surfaces $\Sigma_{1}=\left(
f_{1},\overline{\Delta}\right)  $ and $\Sigma_{2}=\left(  f_{2},\overline
{\Delta}\right)  $ such that
\begin{equation}
\label{s12}\partial\Sigma=\partial\Sigma_{1}+\partial\Sigma_{2}, \quad
R\left(  \Sigma_{1}\right)  +R\left(  \Sigma_{2}\right)  =R(\Sigma),
\end{equation}
and that $\Sigma_{1}$ and $\Sigma_{2}$ satisfy the following condition.

\begin{condition}
\label{cond1} $\tau_{1}^{\circ}$ (resp. $\tau_{2}^{\circ}$) has a neighborhood
$N_{1}$ (resp. $N_{2}$) in $\overline{\Delta_{1}}$ (resp. $\overline
{\Delta_{2}}$). And $\left(  \partial\Delta\right)  \backslash\{1\}$ has a
neighborhood $N_{1}^{\prime}$ (resp. $N_{2}^{\prime}$) in $\overline{\Delta},$
such that $\left(  f_{1},N_{1}^{\prime}\right)  $ (resp. $\left(  f_{2}%
,N_{2}^{\prime}\right)  $) is equivalent to $\left(  f_{1},N_{1}\right)  $
(resp. $\left(  f_{2},N_{2}\right)  $).
\end{condition}

Since each arc in partition (\ref{a1}) is SCC and $\left(  f,\tau_{1}\right)
$ (resp.$\left(  f,\tau_{2}\right)  $) is closed, we may assume $p_{l_{1}}%
\in\alpha_{i_{1}}\left(  a_{i_{1}},a_{i_{1}+1}\right)  \backslash\{a_{i_{1}%
+1}\}$ and $p_{l_{2}}\in\alpha_{i_{1}+k}\left(  a_{i_{1}+k},a_{i_{1}%
+k+1}\right)  \backslash\{a_{i_{1}+k+1}\}$ for some $0\leq k\leq m.$

We should show that $0<k<m$. Otherwise, $p_{l_{1}}$ and $p_{l_{2}}$ are both
contained in $\alpha_{i_{1}}\backslash\{a_{i_{1}+1}\}$ when $k=0$ or $m$. But
$f$ is injective on $\alpha_{j}\backslash\{a_{j+1}\}$ for each $j$, and thus
$p_{l_{1}}=p_{l_{2}},$ contradicting to the assumption. Then
\begin{equation}
\tau_{1}=\alpha_{i_{1}}\left(  p_{l_{1}},a_{i_{1}+1}\right)  +\alpha_{i_{1}%
+1}\left(  a_{i_{1}+1},a_{i_{1}+2}\right)  +\dots+\alpha_{i_{1}+k}\left(
a_{i_{1}+k},p_{l_{2}}\right)  , \label{a3}%
\end{equation}
and
\begin{equation}
\tau_{2}=\alpha_{i_{1}+k}\left(  p_{l_{2}},a_{i_{1}+k+1}\right)
+\alpha_{i_{1}+k+1}\left(  a_{i_{1}+k+1},a_{i_{1}+k+2}\right)  +\dots
+\alpha_{i_{1}+m}\left(  a_{i_{1}+m},p_{l_{1}}\right)  , \label{a4}%
\end{equation}
where $a_{i_{1}+j}=a_{i_{1}+j-m}$ and $\alpha_{i_{1}+j}=\alpha_{i_{1}+j-m}$ if
$i_{1}+j>m$, and either of the two partitions (\ref{a3}) and (\ref{a4})
contains at most $m$ terms.

We first show that $\Sigma_{1}\in\mathcal{C}^{\ast}\left(  L,m\right)  .$ We
may assume $\partial\Delta$ has a partition
\begin{align}
\partial\Delta &  =\alpha_{1}^{\prime}+\alpha_{2}^{\prime}+\dots+\alpha
_{k+1}^{\prime}\label{a5}\\
&  =\alpha_{1}^{\prime}\left(  a_{1}^{\prime},a_{2}^{\prime}\right)
+\alpha_{2}^{\prime}\left(  a_{2}^{\prime},a_{3}^{\prime}\right)
+\dots+\alpha_{k+1}^{\prime}\left(  a_{k+1}^{\prime},a_{1}^{\prime}\right)
,\nonumber
\end{align}
such that
\begin{align*}
\left(  f,\alpha_{i_{1}}\left(  p_{l_{1}},a_{i_{1}+1}\right)  \right)   &
=\left(  f_{1},\alpha_{1}^{\prime}\right)  ,\\
(f,\alpha_{i_{1}+1}\left(  a_{i_{1}+1},a_{i_{1}+2}\right)  )  &  =\left(
f_{1},\alpha_{2}^{\prime}\right)  ,\\
&  \ldots\\
\left(  f,\alpha_{i_{1}+k}\left(  a_{i_{1}+k},p_{l_{2}}\right)  \right)   &
=\left(  f_{1},\alpha_{k+1}^{\prime}\left(  a_{k+1}^{\prime},a_{1}^{\prime
}\right)  \right)  .
\end{align*}
Note that $\partial\Delta=\alpha_{1}^{\prime}$ if and only if $p_{l_{1}%
}=a_{i_{1}}$, $p_{l_{2}}=a_{i_{1}+1}$, $c_{i_{1}}=\left(  f,\alpha_{i_{1}%
}\right)  $ is a whole circle, and $(f_{1},\partial\Delta)=\left(  f,\tau
_{1}\right)  $. In this way, $L(\partial\Sigma_{1})<L(\partial\Sigma)$. It
follows from (\ref{a1}), (\ref{a3}) and Condition \ref{cond1} that, the
partition (\ref{a5}) is a $\mathcal{C}^{\ast}\left(  L,m\right)  $-partition.
Similarly, $\Sigma_{2}$ also has a $\mathcal{C}^{\ast}\left(  L,m\right)  $-partition.

It is clear that
\[
\max\{B_{f_{1}}^{\ast}(\Delta),B_{f_{2}}^{\ast}(\Delta)\}\leq B_{f_{1}}^{\ast
}(\Delta)+B_{f_{2}}^{\ast}(\Delta)=B_{f}^{\ast}\left(  \Delta\right)  -1.
\]
Recalling the condition (\ref{s12}), we deduce that $\max\{H(\Sigma
_{1}),H(\Sigma_{2})\}\geq H(\Sigma).$ We may assume $H(\Sigma_{1})\geq
H(\Sigma_{2}),$ otherwise we replace $\Sigma_{1}$ with $\Sigma_{2}.$ Then
$\Sigma_{1}$ is the desired surface in Case (5) and in this case,
$L(\partial\Sigma_{1})<L(\partial\Sigma).$

\noindent\textbf{Discussion of Cases (3). }Let $A=\cup_{l=1}^{v}\beta_{l}$ and
$\Delta_{1}=\Delta\backslash A.$ Then we obtain a surface $F$ whose interior
is $\left(  f,\Delta_{1}\right)  $ and whose boundary is
\begin{align*}
\gamma_{1}  &  =\left(  \partial\Delta\right)  -\beta_{1}\left(  p_{0}%
,p_{1}\right)  +\beta_{2}\left(  p_{0},p_{2}\right)  -\beta_{2}\left(
p_{0},p_{2}\right) \\
&  +\dots+\beta_{v}\left(  p_{0},p_{v}\right)  -\beta_{v}\left(  p_{0}%
,p_{v}\right)  +\beta_{1}\left(  p_{0},p_{1}\right)  ,
\end{align*}
where $\partial\Delta$ is regarded as a closed path from $p_{1}$ to $p_{1}.$
See (1) of Figure \ref{fig4.2} for the case $v=3$. Now we split $A$ into a
simple path
\begin{align*}
\gamma &  =-\beta_{1}^{\prime\prime}\left(  p_{0}^{2},p_{1}\right)  +\beta
_{2}^{\prime}\left(  p_{0}^{2},p_{2}\right)  -\beta_{2}^{\prime\prime}\left(
p_{0}^{3},p_{2}\right) \\
&  +\dots+\beta_{v}^{\prime}\left(  p_{0}^{v},p_{v}\right)  -\beta_{v}%
^{\prime\prime}\left(  p_{0}^{1},p_{v}\right)  +\beta_{1}^{\prime}\left(
p_{0}^{1},p_{1}\right)  ,
\end{align*}
as in Figure \ref{fig4.2} (2). Via a homeomorphism from $\Delta_{1}^{\prime}$
onto $\Delta_{1}$, we obtain the surface $F=\left(  g,\overline{\Delta
_{1}^{\prime}}\right)  $ whose interior is equivalent to $\left(  f,\Delta
_{1}\right)  $ and whose boundary $\partial F=\left(  g,\partial\Delta
_{1}^{\prime}\right)  $ is equivalent to $\left(  f,\gamma_{1}\right)  .$ Then
it is easy to see that $\Sigma$ can be recovered by sewing $F$ along
$\beta_{l}^{\prime}$ and $\beta_{l}^{\prime\prime},$ which means by
identifying $\beta_{l}^{\prime}$ and $\beta_{l}^{\prime\prime},$
$l=1,2,\dots,v.$

It is interesting that, by Lemma \ref{cut-3} (ii), we can sew $F$ by
identifying $\beta_{l}^{\prime\prime}$ with $\beta_{l+1}^{\prime},$ for
$l=1,2,\dots,v-1,$ and $\beta_{v}^{\prime\prime}$ with $\beta_{1}^{\prime}$,
to obtain a new surface $\Sigma_{1}=\left(  f_{1},\overline{\Delta}\right)  .$
Indeed, we can deform $\overline{\Delta_{1}^{\prime}}$ as in Figure
\ref{fig4.2} (2) into $\overline{\Delta_{1}^{\prime\prime}}$ as in Figure
\ref{fig4.2} (3) with $p_{1}$ fixed, and then deform $\Delta_{1}^{\prime
\prime}$ homeomorphically onto the disk $\Delta$ omitting the union $B$ of the
$v$ line segments $\overline{p_{0}^{l}p_{1}}$ for $l=1,2,\dots,v,$ as in
Figure \ref{fig4.2} (4).%

\begin{figure}
[ptb]
\begin{center}
\includegraphics[
height=4.2263in,
width=4.1727in
]%
{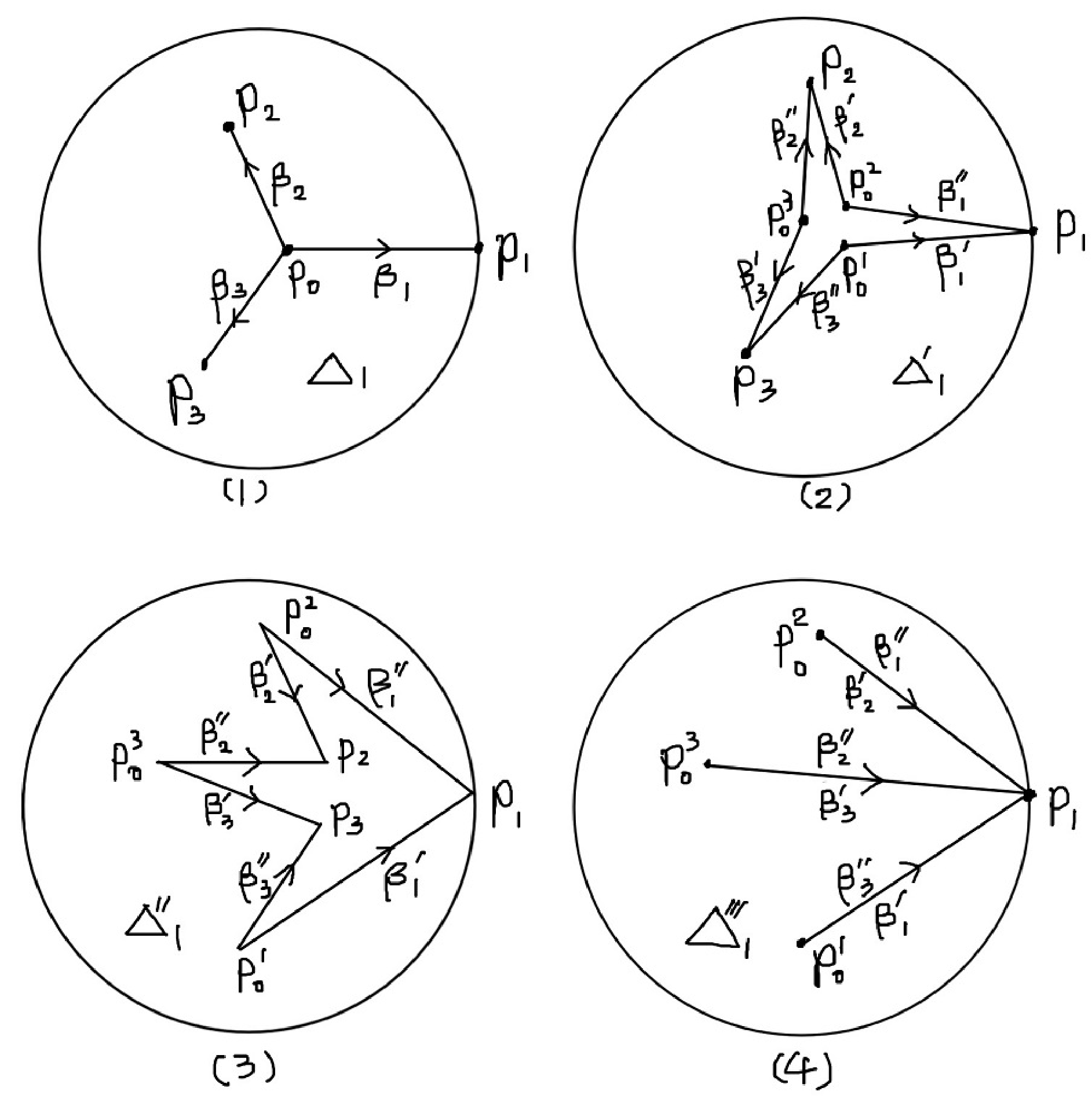}%
\caption{ }%
\label{fig4.2}%
\end{center}
\end{figure}

It is clear that $A\left(  \Sigma\right)  =A\left(  \Sigma_{1}\right)  $ and
$L(\partial\Sigma)=L(\partial\Sigma_{1}).$ When $b_{1}=b,$ we see by $b\in
E_{q}$ that $\{p_{j}\}_{j=1}^{v}\subset f^{-1}(E_{q})$ and when $b_{1}\neq b$
we have $A\cap E_{q}=\emptyset.$ Thus
\[
\overline{n}\left(  F,E_{q}\right)  =\overline{n}\left(  f_{1},E_{q}\right)
=\#\{f^{-1}(E_{q})\cap\left(  \Delta\backslash A\right)  \}=\overline
{n}\left(  \Sigma\right)  -\left(  v-1\right)  \chi_{E_{q}}\left(
b_{1}\right)  ,
\]
where $\chi_{E_{q}}\left(  b_{1}\right)  =1$ when $b_{1}\in E_{q}$ and
$\chi_{E_{q}}\left(  b_{1}\right)  =0$ when $b_{1}\notin E_{q}.$ Clearly, we
have $\partial\Sigma\sim\partial\Sigma_{1}.$ Thus $\Sigma_{1}\in
\mathcal{C}\left(  L,m\right)  $ and
\begin{equation}
H\left(  \Sigma_{1}\right)  =H\left(  \Sigma\right)  +\frac{4\pi\left(
v-1\right)  \chi_{E_{q}}\left(  b_{1}\right)  }{L(\partial\Sigma)}.
\label{H>H}%
\end{equation}
If $b_{1}=b,$ then by (\ref{H>H}) and (\ref{210615}), we obtain a
contradiction that
\[
H_{L}\geq H(\Sigma_{1})>H_{L}-\frac{\pi}{2L(\partial\Sigma)}+\frac{4\pi\left(
v-1\right)  }{L(\partial\Sigma)}>H_{L}.
\]
Thus we have
\begin{equation}
b_{1}\neq b\ \mathrm{and\ }A\cap f^{-1}(E_{q})=\emptyset, \label{bb=0}%
\end{equation}
which induces that $\Sigma_{1}\in\mathcal{C}^{\ast}\left(  L,m\right)  ,$ and
that $H\left(  \Sigma_{1}\right)  =H\left(  \Sigma\right)  .$

After above deformations, all $p_{0}^{l}$, $l=1,2,\dots,v,$ are regular points
of $f_{1}.$ Thus
\[
\sum_{l=1}^{v}\left(  v_{f_{1}}\left(  p_{0}^{l}\right)  -1\right)
=v_{f}\left(  p_{0}\right)  -v=0.
\]
On the other hand, $p_{0}$ and $\left\{  p_{l}\right\}  _{l=2}^{v}$ are the
only possible branch points of $f$ on $A\cap\Delta$, and the cut $B$ inside
$\Delta$ contains no branch point of $f_{1}$. Thus we have
\[
B_{f}\left(  \{p_{0},p_{2},\dots,p_{v}\}\right)  \geq B_{f}\left(
p_{0}\right)  =v_{f}\left(  p_{0}\right)  -1=v-1,
\]
and
\begin{align*}
B_{f}^{\ast}\left(  \Delta\right)   &  =B_{f}^{\ast}\left(  \left(
\Delta\backslash A\right)  \right)  +B_{f}^{\ast}\left(  \{p_{0},p_{2}%
,\dots,p_{v}\}\right) \\
&  \geq B_{f}^{\ast}\left(  \left(  \Delta\backslash A\right)  \right)  +v-1\\
&  =B_{f_{1}}^{\ast}\left(  \Delta\backslash\cup_{l=1}^{v}\overline{p_{1}%
p_{0}^{l}}\right)  +v-1\\
&  =B_{f_{1}}^{\ast}\left(  \Delta\right)  +v-1\\
&  \geq B_{f_{1}}^{\ast}\left(  \Delta\right)  +1.
\end{align*}

It is clear that
\[
v_{f_{1}}\left(  p_{1}\right)  =v_{f}\left(  p_{1}\right)  +v_{f}\left(
p_{2}\right)  +\dots+v_{f}\left(  p_{v}\right)  \geq v_{f}\left(
p_{1}\right)  +v-1>v_{f}\left(  p_{1}\right)  +1.
\]
and thus we have by (\ref{bb=0}) $B_{f_{1}}^{\ast}\left(  p_{1}\right)
>B_{f}^{\ast}\left(  p_{1}\right)  .$ On the other hand, we have $b_{f}\left(
z\right)  \equiv b_{f_{1}}\left(  z\right)  $ for all $z\in\left(
\partial\Delta\right)  \backslash\{p_{1}\}.$ Thus we have%
\[
B_{f_{1}}^{\ast}\left(  \partial\Delta\right)  >B_{f}^{\ast}\left(
\partial\Delta\right)  .
\]
This completes the proof of Case (3).%

\begin{figure}
[ptb]
\begin{center}
\includegraphics[
height=4.5965in,
width=4.1935in
]%
{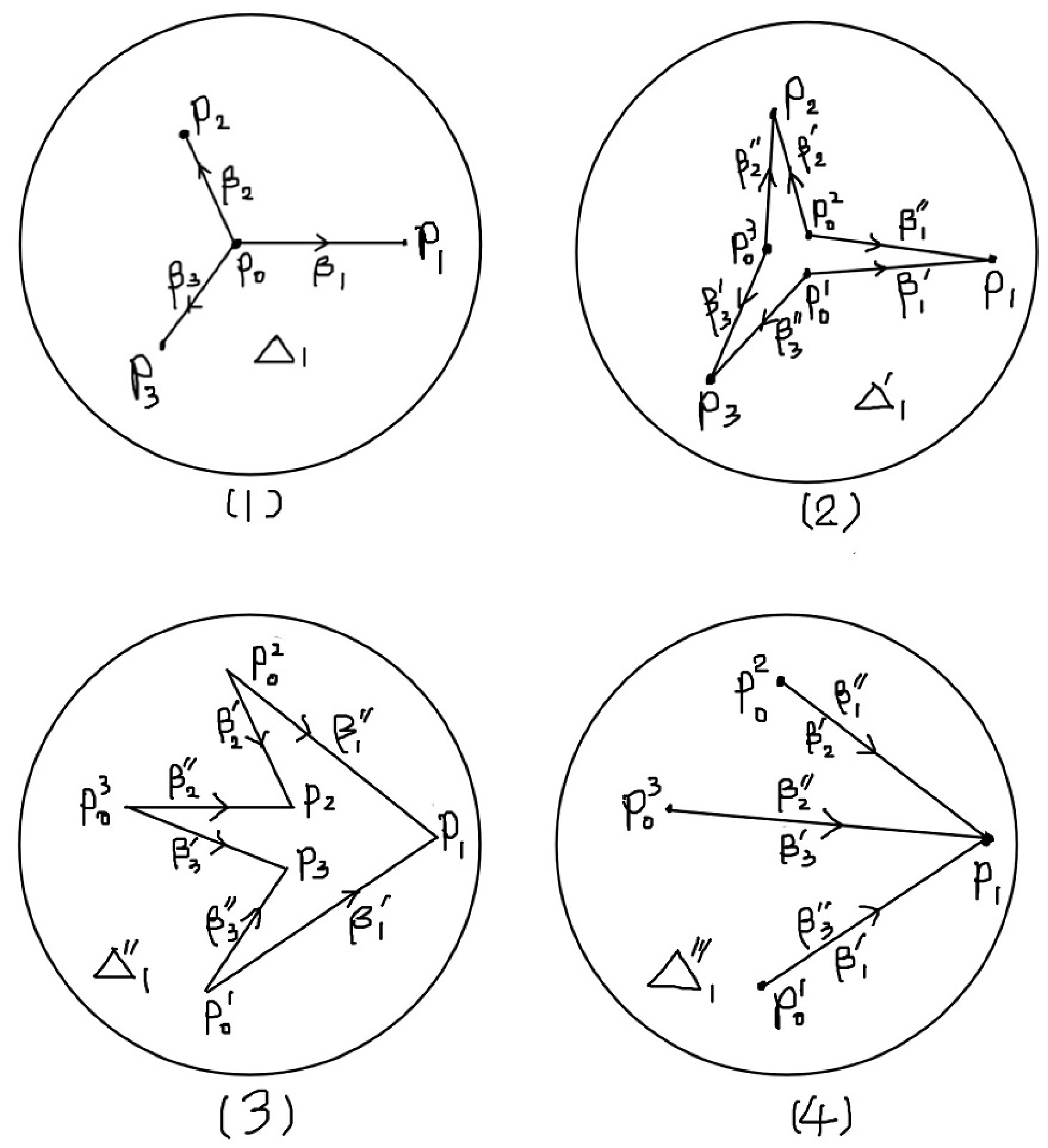}%
\caption{ }%
\label{fig4.3}%
\end{center}
\end{figure}

\noindent\textbf{Case (4) cannot occur. }In this case, $b_{1}=b$, $\left\{
p_{l}\right\}  _{l=1}^{v}\subset f^{-1}(E_{q})\ $and $A\subset\Delta.$ The
discussion is similar to of Case (3) with $b_{1}=b$, and we can deduce a
contradiction. Then, as in Figure \ref{fig4.3}, we can cut and split $\Delta$
along $A$ to obtain an annulus $\Delta_{1}=\Delta\backslash\overline{D}$ with
$\partial\Delta_{1}=\partial\Delta-\partial D$, where $\partial D=\beta
_{1}^{\prime}-\beta_{1}^{\prime\prime}+\beta_{2}^{\prime}-\beta_{2}%
^{\prime\prime}+\dots+\beta_{v}^{\prime}-\beta_{v}^{\prime\prime}.$ Repeating
the same strategies in Case (3), we can obtain a new surface $\Sigma
_{1}=\left(  f_{1},\overline{\Delta}\right)  $ so that $f_{1}$ and $f$
coincide on a neighborhood of $\partial\Delta$ in $\overline{\Delta},$ which
implies that $\Sigma_{1}\in\mathcal{C}^{\ast}\left(  L,m\right)  .$ In Figure
\ref{fig4.3} (4), $B=\overline{p_{1}p_{0}^{1}}\cup\overline{p_{1}p_{0}^{2}%
}\cup\overline{p_{1}p_{0}^{3}}\cup\dots\cup\overline{p_{1}p_{0}^{v}}$ contains
only one point $p_{1}$ of $f_{1}^{-1}(E_{q}),$ and thus
\begin{align*}
\#[f^{-1}(E_{q})\cap\Delta]  &  =\#\left[  f^{-1}(E_{q})\cap\Delta
\backslash\{p_{l}\}_{l=1}^{v}\right]  +v\\
&  =\#\left[  f_{1}^{-1}(E_{q})\cap\Delta\backslash B\right]  +\#[f_{1}%
^{-1}(E_{q})\cap B]+v-1\\
&  =\#[f_{1}^{-1}(E_{q})\cap\Delta]+v-1,
\end{align*}
which implies
\[
\overline{n}\left(  \Sigma_{1}\right)  =\overline{n}\left(  \Sigma\right)
-v+1\leq\overline{n}\left(  \Sigma\right)  -1.
\]
From above arguments, we derive $H(\Sigma_{1})\geq H\left(  \Sigma\right)
+\frac{4\pi}{L(\partial\Sigma)}.$ This again implies a contradiction.

Now our proof has been completed.
\end{proof}

\begin{remark}
For a branch point $a$ of $f,$ we call $\left(  a,f(a)\right)  $ a branch pair
of $f.$ In Case (3) of previous proof, $f_{1}$ can be understood as a movement
of the branch pair $\left(  p_{0},f(p_{0})\right)  $ of $f$ to the branch pair
$\left(  p_{1},f_{1}\left(  p_{1}\right)  \right)  $ of $f_{1}$ along the
curve $\beta_{1}\left(  p_{0},p_{1}\right)  $. Then $\left(  p_{0}%
,f(p_{0})\right)  $ is split into $v$ regular pairs $\left(  p_{0}^{l}%
,f_{1}\left(  p_{0}^{l}\right)  \right)  =\left(  p_{0}^{l},f\left(
p_{0}\right)  \right)  $, $l=1,\ldots,v$, and $\left(  p_{1},f(p_{1}\right)
)$ becomes a branch pair of $f_{1}$ at the boundary point $p_{1},$ whose order
$v_{f_{1}}\left(  p_{1}\right)  =\sum\limits_{l=1}^{v}v_{f}\left(
p_{l}\right)  $. Meanwhile, all other branch pairs $\left(  x,f(x)\right)  $
remain unchanged, saying that there exists a homeomorphism $h$ from
$\overline{\Delta}\backslash A$ onto $\overline{\Delta}\backslash B$ such that
$\left(  f,\overline{\Delta}\backslash A\right)  $ is equivalent to $\left(
f_{1}\circ h,{\Delta}\backslash A\right)  .$
\end{remark}

\begin{corollary}
\label{move-out} Let $\Sigma=\left(  f,\overline{\Delta}\right)
\in\mathcal{C}^{\ast}\left(  L,m\right)  $, and assume that (\ref{210615})
holds. Then there exists a surface $\Sigma_{1}=\left(  f_{1},\Delta\right)
\in\mathcal{C}^{\ast}\left(  L,m\right)  $ satisfying (\ref{210615}) such that
$C_{f_{1}}^{\ast}\left(  \Delta\right)  =\emptyset$,
\[
H(\Sigma_{1})\geq H\left(  \Sigma\right)  ,L(\partial\Sigma_{1})\leq L\left(
\partial\Sigma\right)  ,
\]
and (i) or (ii) holds:

(i) $C_{f}^{\ast}\left(  \Delta\right)  \neq\emptyset\ $and $L(\partial
\Sigma_{1})<L\left(  \partial\Sigma\right)  $.

(ii) $H(\Sigma_{1})=H\left(  \Sigma\right)  ,L(\partial\Sigma_{1})=L\left(
\partial\Sigma\right)  ,\partial\Sigma_{1}=\partial\Sigma;$ and moreover
$B_{f_{1}}^{\ast}\left(  \partial\Delta\right)  >B_{f}^{\ast}\left(
\partial\Delta\right)  \ $if and only if $C_{f}^{\ast}\left(  \Delta\right)
\neq\emptyset.$
\end{corollary}

\begin{proof}
When $C_{f}^{\ast}\left(  \Delta\right)  =\emptyset,$ then $\Sigma_{1}=\Sigma$
is the desired surface and (ii) holds. So we assume $C_{f}^{\ast}\left(
\Delta\right)  \neq\emptyset.$ Then by Lemma \ref{move-in-1}, there exists a
surface $\Sigma_{1}^{\prime}=\left(  f_{1}^{\prime},\overline{\Delta}\right)
\in\mathcal{C}^{\ast}\left(  L,m\right)  $ such that
\[
H(\Sigma_{1}^{\prime})\geq H\left(  \Sigma\right)  ,L(\partial\Sigma
_{1}^{\prime})\leq L(\partial\Sigma),
\]
and%
\begin{equation}
B_{f_{1}^{\prime}}^{\ast}\left(  \Delta\right)  \leq B_{f}^{\ast}\left(
\Delta\right)  -1. \label{as2-13}%
\end{equation}
Moreover, $L(\partial\Sigma_{1}^{\prime})=L(\partial\Sigma)$ if and only if
$\partial\Sigma_{1}^{\prime}=\partial\Sigma,H(\Sigma_{1}^{\prime})=H\left(
\Sigma\right)  $ and $C_{f_{1}^{\prime}}^{\ast}\left(  \partial\Delta\right)
>C_{f}^{\ast}\left(  \partial\Delta\right)  .$ It is clear that $\Sigma
_{1}^{\prime}$ again satisfies the inequality (\ref{210615}). Repeating this
procedure at most $B_{f_{1}^{\prime}}^{\ast}\left(  \Delta\right)  $ times, we
can obtain the desired surface $\Sigma_{1}$.
\end{proof}

Next, we will establish some lemmas to remove the branch point on the boundary.

\begin{lemma}
\label{move-b-1}Let $\Sigma=\left(  f,\overline{\Delta}\right)  \in
\mathcal{C}^{\ast}\left(  L,m\right)  $ be a surface satisfying the inequality
(\ref{210615}) with the $\mathcal{C}^{\ast}\left(  L,m\right)  $-partitions
(\ref{a1}) and (\ref{a2}). Suppose that

(A) $f$ has no branch points in $\Delta\backslash f^{-1}(E_{q})$;

(B) For the first term $\alpha_{1}\left(  a_{1},a_{2}\right)  $ of (\ref{a1}),
$\alpha_{1}\left(  a_{1},a_{2}\right)  \backslash\{a_{2}\}$ contains a branch
point $p_{0}$ of $f$ with $p_{0}\notin f^{-1}(E_{q}).$ $p_{1}\ $is a point in
$\alpha_{1}\left(  p_{0},a_{2}\right)  $ such that $f\left(  p_{0}\right)
\neq f(p_{1}),$ $\left[  \alpha_{1}\left(  p_{0},p_{1}\right)  \backslash
\{p_{1}\}\right]  \cap f^{-1}(E_{q})=\emptyset$ and that $\alpha_{1}^{\circ
}\left(  p_{0},p_{1}\right)  $ contains no branch point of $f$;

(C) For $b_{0}=f\left(  p_{0}\right)  $ and $b_{1}=f\left(  p_{1}\right)  ,$
the subarc $c_{1}^{\prime}=c_{1}\left(  b_{0},b_{1}\right)  $ of $c_{1}$ has
$v=v_{f}\left(  p_{0}\right)  $ distinct $f$-lifts $\beta_{1}\left(
p_{0},p_{1}\right)  ,\beta_{2}\left(  p_{0},p_{2}\right)  ,\dots,\beta
_{v}\left(  p_{0},p_{v}\right)  ,$ arranged anticlockwise around $p_{0}$, such
that $\beta_{l}\backslash\{p_{0},p_{l}\}\subset\Delta$ for $l=2,\dots,v.$

Then there exists a surface $\Sigma_{1}=\left(  f_{1},\overline{\Delta
}\right)  \in\mathcal{C}^{\ast}\left(  L,m\right)  $ such that there is no
branch points of $f_{1}$ in $\Delta\backslash f_{1}^{-1}(E_{q})$, and one of
the following alternatives (i) and (ii) holds:

(i) The partition number $m\geq2,$%
\[
H(\Sigma_{1})\geq H\left(  \Sigma\right)  ,L(\partial\Sigma_{1})<L(\partial
\Sigma),
\]
and
\begin{equation}
\#\left(  \partial\Delta\right)  \cap f_{1}^{-1}(E_{q})\leq\#\left(
\partial\Delta\right)  \cap f^{-1}(E_{q}). \label{eq<eq}%
\end{equation}
Moreover
\begin{equation}
\#C_{f_{1}}^{\ast}\left(  \partial\Delta\right)  \leq\#C_{f}^{\ast}\left(
\partial\Delta\right)  , \label{as2-10}%
\end{equation}
with equality only if one of the following relations (\ref{as2-11}%
)--(\ref{23-3}) holds:
\begin{equation}
B_{f_{1}}^{\ast}\left(  \partial\Delta\right)  \leq B_{f}^{\ast}\left(
\partial\Delta\right)  -1, \label{as2-11}%
\end{equation}%
\begin{equation}
\mathfrak{L}\left(  \partial\Sigma_{1}\right)  \leq\mathfrak{L}\left(
\partial\Sigma\right)  -1, \label{23-1}%
\end{equation}%
\begin{equation}
\Sigma_{1}=\left(  f_{1},\overline{\Delta}\right)  \in\mathcal{C}^{\ast
}\left(  L,m-1\right)  , \label{23-2}%
\end{equation}%
\begin{equation}
A(\Sigma_{1})\leq A\left(  \Sigma\right)  -4\pi. \label{23-3}%
\end{equation}

(ii) $p_{l},l=1,2,\dots,v,$ are distinct, $\left\{  p_{l}\right\}  _{l=2}%
^{v}\subset\Delta$, $p_{1}\notin f^{-1}(E_{q})$, $\partial\Sigma_{1}%
=\partial\Sigma,H(\Sigma_{1})=H(\Sigma),$ $v_{f}(x)=v_{f_{1}}(x)$ for all
$x\in\left(  \partial\Delta\right)  \backslash\{p_{0},p_{1}\}$, $v_{f_{1}%
}(p_{0})=1$ and
\begin{equation}
v_{f_{1}}\left(  p_{1}\right)  =v_{f}\left(  p_{1}\right)  +v-1, \label{=vv-1}%
\end{equation}
and moreover, (\ref{a1}) and (\ref{a2}) are still $\mathcal{C}^{\ast}\left(
L,m\right)  $-partitions of $\partial\Sigma_{1},$%
\begin{equation}
B_{f_{1}}^{\ast}\left(  \partial\Delta\right)  =B_{f}^{\ast}\left(
\partial\Delta\right)  , \label{b=b}%
\end{equation}
and
\begin{equation}
\#C_{f_{1}}^{\ast}\left(  \partial\Delta\right)  \leq\#C_{f}^{\ast}\left(
\partial\Delta\right)  , \label{c=c}%
\end{equation}
equality holding if and only if $p_{1}\notin C_{f}^{\ast}\left(
\partial\Delta\right)  \cup f^{-1}(E_{q}).\label{chenges to < or = copy(1)}$
\end{lemma}

\begin{proof}
By (C) we have that $\beta_{1}=\alpha_{1}(p_{0},p_{1})$. Write $A=\cup
_{l=1}^{v}\beta_{l}$. We will imitate the arguments in the proof of Lemma
\ref{move-in-1}. Under the partitions (\ref{a1}) and (\ref{a2}), we first
consider the case that $p_{l_{1}}=p_{l_{2}}$ for some pair $1\leq l_{1}%
<l_{2}\leq v.$ Then $\beta_{l_{1}}-\beta_{l_{2}}$ bounds a Jordan domain $D$
contained in $\Delta$, and we may face the following three Cases.

\noindent\textbf{Case (1). }$l_{1}=1$ and $l_{2}=2\ $(see Figure \ref{fig4.4} (1)).

\noindent\textbf{Case (2). }$1<l_{1}\ $and $p_{l_{1}}\in\partial\Delta$ (see
Figure \ref{fig4.4} (3)).

\noindent\textbf{Case (3). }$1<l_{1}$ $\ $and $p_{l_{1}}\in\Delta\ $(see
Figure \ref{fig4.4} (5)).

We show that none of the above three cases can occur, by deduce a
contradiction that $H_{L}>H_{L}.$

When Case (1) occurs, we put $h_{1}$ to be a homeomorphism from $\overline
{\Delta}$ onto $\overline{\Delta_{1}}=\overline{\Delta}\backslash D$ so that
$h_{1}$ is an identity on $\left(  \partial\Delta_{1}\right)  \cap
\partial\Delta$. Then put $\Sigma_{1}=\left(  f_{1},\overline{\Delta}\right)
$ with $f_{1}=f\circ h_{1}$ (See Figure \ref{fig4.4} (1) and (2)).

When Case (2) occurs, $\partial D$ divides $\Delta$ into three Jordan domains
$\Delta_{1}$, $D$ and $\Delta_{2}$ as in Figure \ref{fig4.4} (3). We can glue
the surfaces $\left(  f|_{\overline{\Delta_{1}}},\overline{\Delta_{1}}\right)
$ and $\left(  f|_{\overline{\Delta_{2}}},\overline{\Delta_{2}}\right)  $
together along the boundary $(f,\beta_{l_{1}})\sim(f,\beta_{l_{2}})$ to obtain
a new surface $\Sigma_{1}=\left(  f_{1},\overline{\Delta}\right)  $. Indeed,
we can take a continuous mapping $h_{2}:\overline{\Delta}\backslash
D\to\overline{\Delta}$ so that $h_{2}|_{\overline{\Delta_{1}}}:\overline
{\Delta_{1}}\rightarrow\overline{\Delta_{1}^{\prime}}$ (resp. $h_{2}%
|_{\overline{\Delta_{2}}}:\overline{\Delta_{2}}\rightarrow\overline{\Delta
_{2}^{\prime}}$) is an orientation-preserving homeomorphism, $f(h_{2}%
^{-1}(y))$ is a singleton for all $y\in\beta$, and $h_{2}$ is an identity on a
neighborhood of $\left(  \partial\Delta\right)  \backslash\{p_{0},p_{l_{1}}\}$
in $\overline{\Delta}$. Then we define $\Sigma_{1}=\left(  f_{1}%
,\overline{\Delta}\right)  $ with $f_{1}=f\circ h_{2}^{-1}$ (See Figure
\ref{fig4.4} (3) and (4)).

When Case (3) occurs, $\left(  \beta_{l_{1}}\cup\beta_{l_{2}}\right)
\backslash\{p_{0}\}\subset\Delta,$ and $\Delta_{1}=\Delta\backslash
\overline{D}$ is a domain as in Figure \ref{fig4.4} (5) when $l_{1}=1$ and
$l_{2}=2$. We can sew $\left(  f,\overline{\Delta\backslash D}\right)  $ along
$(f,\beta_{l_{1}})\sim(f,\beta_{l_{2}})$ to obtain a surface $\left(
f_{1},\overline{\Delta}\right)  $ so that $\beta_{l_{1}}-\beta_{l_{2}}$
becomes a simple path $\beta$, the line segment from $p_{0}$ to $p_{l_{1}}$ as
in Figure \ref{fig4.4} (5) and (6). In fact we can define $f_{1}:=f\circ
h_{3}^{-1}$, where $h_{3}:\overline{\Delta_{1}}\rightarrow\overline{\Delta}$
is an OPCOFOM so that $\beta_{l_{1}}$ and $\beta_{l_{2}}$ are mapped
homeomorphically onto $\beta$, $h_{3}(p_{l_{1}})=p_{l_{1}}$, $h_{3}%
(p_{0})=p_{0}$, $h_{3}$ is an identity on $\partial\Delta$ and on a
neighborhood of $\partial\Delta\backslash\{p_{0}\}$ in $\overline{\Delta}$,
and $h_{3}:\Delta_{1}\rightarrow\Delta_{1}^{\prime}$ is a homeomorphism.%

\begin{figure}
[ptb]
\begin{center}
\includegraphics[
height=2.872in,
width=4.1935in
]%
{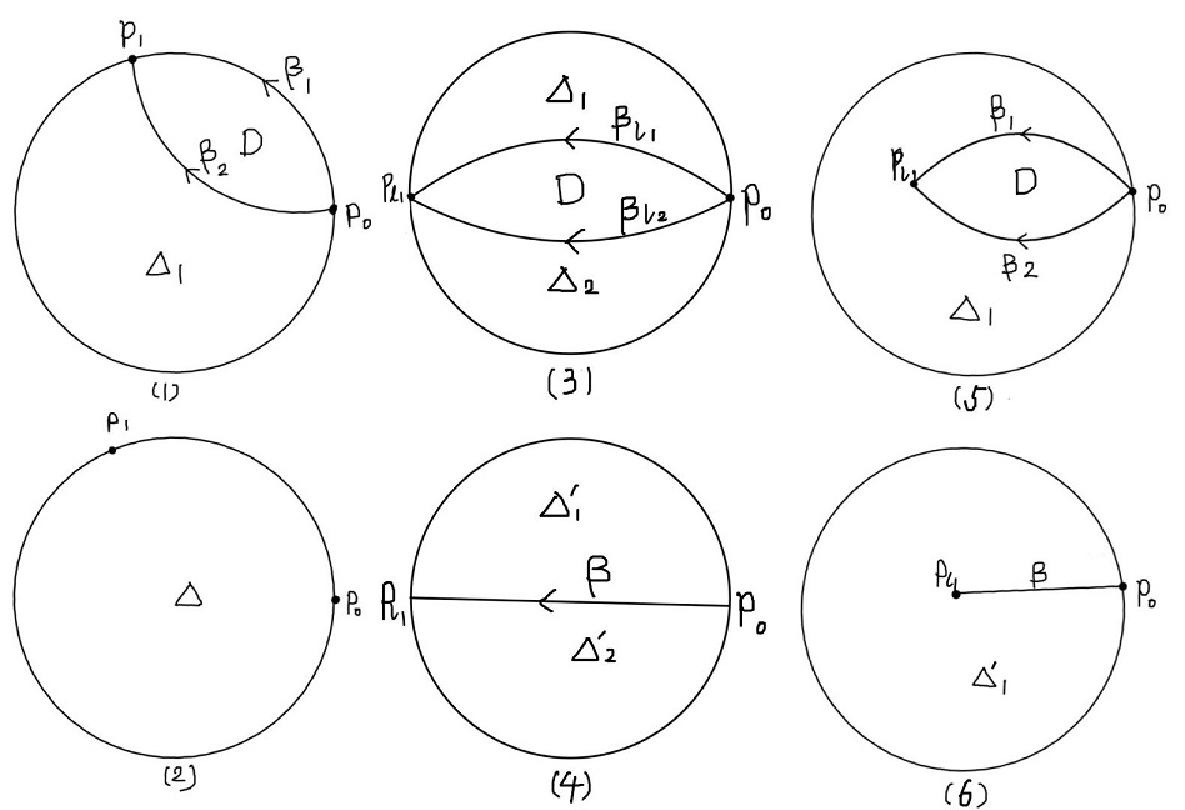}%
\caption{ }%
\label{fig4.4}%
\end{center}
\end{figure}

In the above Cases (1)--(3), it is clear that $\Sigma_{1}$ also has
$\mathcal{C}^{\ast}\left(  L,m\right)  $-partitions as (\ref{a1}) and
(\ref{a2}), and the interior angle of $(f,\overline{D})$ at $p_{l_{1}}$ is a
positive multiple of $2\pi$. Then we have $v_{f_{1}}\left(  p_{l_{1}}\right)
\leq v_{f}\left(  p_{l_{1}}\right)  -1$. Then $\partial D\cap f^{-1}%
(E_{q})=\{p_{1}\}$ or $\emptyset$.

As in the proof of Claim \ref{closed}, $(f,\overline{D})$ can be sewn to be a
closed surface $\Sigma_{0}=(f_{0},S)$ along the equivalent paths
$(f,\beta_{l_{1}})$ and $(f,\beta_{l_{2}})$. Assume that the degree of $f_{0}$
is $d_{0}$. Then we have in any case of Cases (1), (2) and (3),
\[
\overline{n}\left(  \Sigma_{1},E_{q}\right)  \leq\overline{n}\left(
\Sigma,E_{q}\right)  -\overline{n}\left(  \Sigma_{0},E_{q}\right)  +1.
\]
On the other hand, as in the proof of Claim \ref{closed}, by Riemann-Hurwitz
formula, we have $\overline{n}\left(  \Sigma_{0},E_{q}\right)  \geq
(q-2)d_{0}+2$ with the equality holding if and only if $C_{f_{0}}(V)\subset
f_{0}^{-1}(E_{q}).$ Then we have
\[
\overline{n}\left(  \Sigma_{1},E_{q}\right)  \leq\overline{n}\left(
\Sigma,E_{q}\right)  -\overline{n}\left(  \Sigma_{0},E_{q}\right)
+1\leq\overline{n}\left(  \Sigma,E_{q}\right)  -(q-2)d_{0}-1.
\]
Now we have $A(\Sigma)=A(\Sigma_{0})+A(\Sigma_{1})$ and $A(\Sigma_{0})=4\pi
d_{0}$. Then
\begin{align*}
R\left(  \Sigma_{1}\right)   &  =\left(  q-2\right)  A(\Sigma_{1}%
)-4\pi\overline{n}\left(  \Sigma_{1}\right) \\
&  \geq\left(  q-2\right)  (A(\Sigma)-4\pi d_{0})-4\pi\left[  \overline
{n}\left(  \Sigma\right)  -(q-2)d_{0}-1\right] \\
&  =R\left(  \Sigma\right)  +4\pi.
\end{align*}
On the other hand, we have $L(\partial\Sigma)=L(\partial\Sigma_{1})$. Then we
derive
\[
H(\Sigma_{1})\geq\frac{R(\Sigma)+4\pi}{L(\partial\Sigma)}=H(\Sigma)+\frac
{4\pi}{L(\partial\Sigma)},
\]
which with (\ref{210615}) implies the contradiction that $H_{L}\geq
H(\Sigma_{1})>H_{L}$. Hence Cases (1)--(3) can not occur.

There are still two cases left.

\noindent\textbf{Case (4).} $p_{l},l=1,\dots,v,$ are distinct from each other
and $\{p_{l}\}_{l=2}^{v}\subset\Delta$.

\noindent\textbf{Case (5).} $p_{l},l=1,\dots,v,$ are distinct from each other
and $p_{l_{1}}\in\partial\Delta$ for some $2\leq l_{1}\leq v$. In particular,
$\{p_{2},\dots,p_{l_{1}-1}\}\subset\Delta$ when $l_{1}>2,$ and it is possible
that $p_{l_{2}}\in\partial\Delta$ for some $l_{1}<l_{2}\leq v.$

Assume Case (4) occurs. Except for a few differences, the following discussion
is similar to the Cases (3) and (4) in the proof of Lemma \ref{move-in-1}.
Here, we just present the arguments for $v=3$, as in Figure \ref{fig4.5}. Cut
$\Delta$ along the lifts $\beta_{2}$ and $\beta_{3}$ and split $\beta_{2}$ and
$\beta_{3}$ via an OPCOFOM $h$ from a closed Jordan domain $\overline
{\Delta_{1}^{\prime}}$ as in Figure \ref{fig4.5} (2) onto $\overline{\Delta}$
such that $h:\Delta_{1}^{\prime}\rightarrow\Delta_{1}=\Delta\backslash
(\beta_{2}\cup\beta_{3})$ is a homeomorphism.%

\begin{figure}
[ptb]
\begin{center}
\includegraphics[
height=4.1822in,
width=4.1935in
]%
{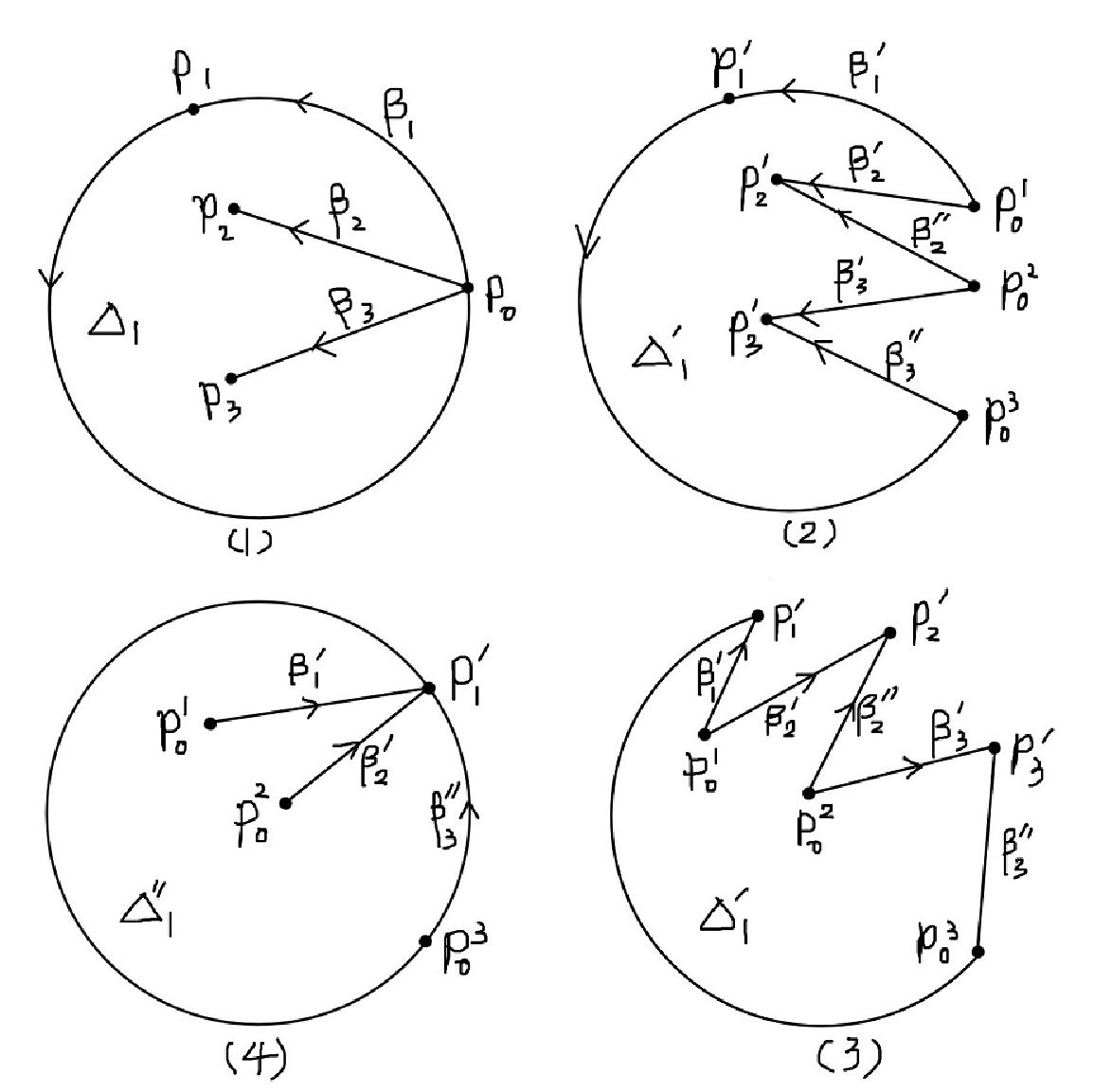}%
\caption{ }%
\label{fig4.5}%
\end{center}
\end{figure}

Then we obtain a surface $\Sigma_{1}^{\prime}=\left(  f_{1}^{\prime}%
,\Delta_{1}^{\prime}\right)  $ such that
\[
\left(  f_{1},\beta_{1}^{\prime}\right)  \sim\left(  f_{1},\beta_{2}^{\prime
}\right)  \sim\left(  f_{1},\beta_{2}^{\prime\prime}\right)  \sim\left(
f_{1},\beta_{3}^{\prime}\right)  \sim\left(  f_{1},\beta_{3}^{\prime\prime
}\right)  .
\]
It is clear that we can recover the surface $\Sigma$ when we identify
$\beta_{2}^{\prime}$ with $\beta_{2}^{\prime\prime}$ and $\beta_{3}^{\prime}$
with $\beta_{3}^{\prime\prime}$. However, by Lemma \ref{cut-3} (ii), we can
also identify $\beta_{1}^{\prime}$ with $\beta_{2}^{\prime},$ and $\beta
_{2}^{\prime\prime}$ with $\beta_{3}^{\prime},$ by deformations in Figure
\ref{fig4.5} (2)-(4), resulting a new surface $\Sigma_{1}=\left(
f_{1},\overline{\Delta}\right)  .$ On the other hand, since $\beta_{l}^{\circ
},l=1,\dots,v,$ contains no point of $f^{-1}(E_{q})$ and $C_{f}^{\ast}\left(
\Delta\right)  =\emptyset,$ $f$ is homeomorphic in neighborhoods of $\beta
_{j}^{\circ},j=2,\dots,v.$ Thus we can conclude the following.

\begin{summary}
\label{Sum1}$\partial\Sigma_{1}\sim\partial\Sigma$. There exists a
neighborhood $N_{1}$ of $\left(  \partial\Delta\right)  \backslash
\{p_{0},p_{1}\}$ in $\overline{\Delta}$ and a neighborhood $N_{1}^{\prime}$ of
$\left(  \partial\Delta\right)  \backslash\{p_{0}^{3},p_{1}^{\prime}\}$ in
$\overline{\Delta}$ such that $\left(  f,N_{1}\right)  \sim\left(  f_{1}%
,N_{1}^{\prime}\right)  .$ In fact as in Figure \ref{fig4.5} (2) and (3),
$\beta_{1}^{\prime\circ},\beta_{2}^{\prime\prime\circ},\beta_{3}^{\prime
\prime\circ}$ have neighborhoods in $\overline{\Delta_{1}^{\prime}}$ so that
the restrictions of $f_{1}^{\prime}$ to them, respectively, are equivalent to
the restriction of $f$ to a neighborhood of $\beta_{1}$ in $\overline{\Delta
}.$ Thus we may replace $p_{1}^{\prime}$ and $p_{0}^{3}$ by $p_{1}$ and
$p_{0}$, and make $\partial\Sigma_{1}=\partial\Sigma$ via a homeomorphism of
$\overline{\Delta}.$ Then partitions (\ref{a1}) and (\ref{a2}) are both
$\mathcal{C}^{\ast}\left(  L,m\right)  $-partitions of $\partial\Sigma_{1}%
\ $if and only if $p_{1}\notin f^{-1}(E_{q})\cap\alpha_{1}^{\circ},$ and in
general $\Sigma_{1}\in\mathcal{C}^{\ast}\left(  L,m+1\right)  \subset
\mathcal{F}\left(  L\right)  .$
\end{summary}

It is clear that $A(\Sigma_{1})=A(\Sigma)\ $and $L(\partial\Sigma
_{1})=L(\partial\Sigma).$ We can also see that $\{p_{0}^{l}\}_{l=1}^{v}$
become regular points of $f_{1}$ and
\[
v_{f_{1}}(p_{1})=v_{f}\left(  p_{1}\right)  +v_{f}\left(  p_{2}\right)
+\dots+v_{f}\left(  p_{v}\right)  .
\]
It implies that
\[
\overline{n}\left(  \Sigma_{1}\right)  =\left\{
\begin{array}
[c]{ll}%
\overline{n}\left(  \Sigma\right)  , & \mathrm{if\ }p_{1}\notin f^{-1}%
(E_{q}),\\
\overline{n}\left(  \Sigma\right)  -v+1, & \mathrm{if\ }p_{1}\in f^{-1}%
(E_{q}).
\end{array}
\right.
\]
Thus in the case $p_{1}\in f^{-1}(E_{q})$, we have
\[
R(\Sigma_{1})=R(\Sigma)+\left(  v-1\right)  4\pi\geq R(\Sigma)+4\pi,
\]
which with (\ref{210615}) implies a contradiction that $H_{L}\geq H(\Sigma
_{1})\geq H(\Sigma)+\frac{4\pi}{L(\partial\Sigma_{1})}>H_{L}$. So we have to
assume $p_{1}\notin f^{-1}(E_{q}),$ which implies $\Sigma_{1}\in
\mathcal{C}^{\ast}\left(  L,m\right)  $ and $H(\Sigma_{1})=H(\Sigma),$ and
moreover
\begin{equation}
\{p_{l}\}_{l=1}^{v}\cap f^{-1}(E_{q})=\emptyset. \label{pf0}%
\end{equation}
Then each $p_{l}\notin C_{f}\left(  \Delta\right)  $ and $v_{f_{1}}%
(p_{1})=v_{f}\left(  p_{1}\right)  +v-1,$ say
\[
b_{f_{1}}(p_{1})-b_{f}\left(  p_{1}\right)  =v-1.
\]
On the other hand we have $v_{f_{1}}(p_{0})=1$ and $v_{f}(p_{0})=v,$ which
implies
\[
b_{f_{1}}\left(  p_{0}\right)  -b_{f}(p_{0})=-v+1,
\]
and thus by (\ref{pf0}) we have
\begin{equation}
B_{f_{1}}^{\ast}\left(  \left\{  p_{0},p_{1}\right\}  \right)  =B_{f}^{\ast
}\left(  \left\{  p_{0},p_{1}\right\}  \right)  . \label{p0=p1}%
\end{equation}

It is clear that, by Summary \ref{Sum1}, $B_{f_{1}}^{\ast}\left(  \left(
\partial\Delta\right)  \backslash\{p_{0},p_{1}\}\right)  =B_{f}^{\ast}\left(
\left(  \partial\Delta\right)  \backslash\{p_{0},p_{1}\}\right)  .$ Then we
have by (\ref{p0=p1})
\[
B_{f_{1}}^{\ast}\left(  \partial\Delta\right)  =B_{f}^{\ast}\left(
\partial\Delta\right)  .
\]
On the other hand, we have $\#C_{f_{1}}^{\ast}\left(  \left\{  p_{0}%
,p_{1}\right\}  \right)  =\#C_{f_{1}}^{\ast}\left(  p_{1}\right)  =1$ and
$\#C_{f}^{\ast}\left(  \left\{  p_{0},p_{1}\right\}  \right)  =1$ if and only
if $p_{1}$ is not a branch point (note that we are in the environment of
$p_{1}\notin f^{-1}(E_{q}),$ which implies $p_{1}\notin f_{1}^{-1}(E_{q})$).
Thus we have $\#C_{f_{1}}^{\ast}\left(  \partial\Delta\right)  \leq
\#C_{f}^{\ast}\left(  \partial\Delta\right)  $ equality holding if and only if
$p_{1}\notin C_{f}^{\ast}\left(  \partial\Delta\right)  \cup f^{-1}(E_{q}).$
Hence, all conclusions in (ii) hold in Case (4).

{Assume Case (5) occurs. }When $m=1,\partial\Sigma=c_{1}\left(  q_{1}%
,q_{1}\right)  $ is a simple circle and thus $f^{-1}(b_{1})\cap\partial
\Delta=\{p_{1}\}$. This case can not occur, since in Case (5), $\{p_{1}%
,p_{l_{1}}\}\subset f^{-1}(b_{1})\cap\partial\Delta$ and $p_{1}\neq p_{l_{1}}%
$. So we have $m\geq2.$

It is clear that $f$ restricted to a neighborhood of $\beta_{l_{1}}^{\circ}$
is homeomorphic and $\beta_{l_{1}}$ divides $\Delta$ into two Jordan domains
$\Delta_{1}$ and $\Delta_{2}$. Denote by $\Delta_{1}$ the domain on the right
hand side of $\beta_{l_{1}}.$ Let $\gamma_{1}$ be the arc of $\partial\Delta$
from $p_{1}$ to $p_{l_{1}}$ and $\gamma_{2}$ be the complement arc of
$\gamma_{1}$ in $\partial\Delta$, both oriented anticlockwise. Recall that
$\beta_{1},\beta_{2},\dots,\beta_{v}$ are arranged anticlockwise around
$p_{0}$. Then we have $\cup_{l=2}^{l_{1}-1}\beta_{l}\backslash\{p_{0}%
\}\subset\Delta_{1},$ while $\left\{  \beta_{l}\right\}  _{l=l_{1}+1}^{v}$ is
contained in $\overline{\Delta_{2}}.$ Based on (\ref{a1}) and (\ref{a2}), we
also have the partitions
\begin{equation}
\gamma_{1}=\alpha_{1}\left(  p_{1},a_{2}\right)  +\alpha_{2}+\dots
+\alpha_{k-1}+\alpha_{k}\left(  a_{k},p_{l_{1}}\right)  , \label{a6}%
\end{equation}
and
\begin{equation}
\gamma_{2}=\alpha_{k}\left(  p_{l_{1}},a_{k+1}\right)  +\alpha_{k+1}%
+\dots+\alpha_{m}+\alpha_{1}\left(  a_{1},p_{1}\right)  , \label{a7}%
\end{equation}
where%
\[
p_{l_{1}}\in\alpha_{k}\left(  a_{k},a_{k+1}\right)  \backslash\{a_{k+1}\}.
\]
We can see that
\[
v_{f|_{\overline{\Delta_{1}}}}\left(  p_{0}\right)  =l_{1}-1,\quad
v_{f|_{\overline{\Delta_{2}}}}\left(  p_{0}\right)  =v-l_{1}+1.
\]
Considering $p_{0}\notin f^{-1}(E_{q})$, we have
\begin{equation}
B_{f|_{\overline{\Delta_{1}}}}^{\ast}\left(  p_{0}\right)  =l_{1}-2,\quad
B_{f|_{\overline{\Delta_{2}}}}^{\ast}\left(  p_{0}\right)  =v-l_{1},
\label{BBB0}%
\end{equation}
and
\begin{equation}
B_{f}^{\ast}\left(  p_{0}\right)  -B_{f|_{\overline{\Delta_{1}}}}^{\ast
}\left(  p_{0}\right)  -B_{f|_{\overline{\Delta_{2}}}}^{\ast}\left(
p_{0}\right)  =1. \label{BBB}%
\end{equation}

Now we shall consider $\Delta_{1}$ and $\Delta_{2}$ separately.

Firstly, let $h_{2}$ be a homeomorphism from $\overline{\Delta_{2}}$ onto
$\overline{\Delta}$ such that $h_{2}|_{\gamma_{2}\setminus\beta_{1}}=id$ and
$h_{2}|_{\beta_{l_{1}}}=\beta_{1}^{\circ}+\gamma_{1}.$ Recall that $\beta
_{1}=\alpha_{1}\left(  p_{0},p_{1}\right)  $. Then we can construct a new
surface as
\[
\Sigma_{2}^{\prime}=\left(  f_{2}^{\prime},\overline{\Delta}\right)  =\left(
f\circ h_{2}^{-1},\overline{\Delta}\right)  ,
\]
with
\begin{align*}
L(f_{2}^{\prime},\partial\Delta)  &  =L(f,(\partial\Delta_{2})\backslash
\beta_{l_{1}})+L(f,\beta_{l_{1}})\\
&  =L(f,(\partial\Delta_{2})\backslash\beta_{l_{1}})+L(f,\beta_{1})\\
&  =L(\gamma_{2})<L.
\end{align*}
Since $p_{1}\neq p_{l_{1}}$, $f(p_{1})=f(p_{l_{1}})=b_{1}$ and $f$ is
injective on each $\alpha_{k}\left(  a_{k},a_{k+1}\right)  \backslash
\{a_{k+1}\}$, we conclude that either of the two partitions (\ref{a6}) and
(\ref{a7}) contains at least two terms. Since the sum of terms of (\ref{a6})
and (\ref{a7}) is at most $m+2$, we conclude that either of (\ref{a6}) and
(\ref{a7}) contains at most $m$ terms. Thus we have $\Sigma_{2}^{\prime}%
\in\mathcal{C}^{\ast}\left(  L,m\right)  .$ Hence, summarizing the above
discussion, we have

\begin{claim}
\label{CL2}$C_{f_{2}^{\prime}}^{\ast}\left(  \Delta\right)  =\emptyset$,
$\Sigma_{2}^{\prime}\in\mathcal{C}^{\ast}\left(  L,m\right)  $, and moreover,
by definition of $f_{2}^{\prime}$,
\[
\#C_{f_{2}^{\prime}}^{\ast}\left(  \partial\Delta\right)  =\#C_{f|_{\overline
{\Delta_{2}}}}^{\ast}\left(  \partial\Delta_{2}\right)  =\#C_{f|_{\overline
{\Delta_{2}}}}^{\ast}\left(  \gamma_{2}\right)  \leq\#C_{f}^{\ast}\left(
\partial\Delta\right)  ,
\]%
\[
\#\left(  \partial\Delta\right)  \cap f_{2}^{\prime-1}\left(  E_{q}\right)
=\#\gamma_{2}\cap f^{-1}\left(  E_{q}\right)  \leq\#\left(  \partial
\Delta\right)  \cap f^{-1}\left(  E_{q}\right)  ,
\]%
\[
B_{f_{2}^{\prime}}^{\ast}\left(  \partial\Delta\right)  =\#B_{f|_{\overline
{\Delta_{2}}}}^{\ast}\left(  \partial\Delta_{2}\right)  \leq\#B_{f}^{\ast
}\left(  \partial\Delta\right)  -1.
\]

\end{claim}%

\begin{figure}
[ptb]
\begin{center}
\includegraphics[
height=3.4688in,
width=4.1511in
]%
{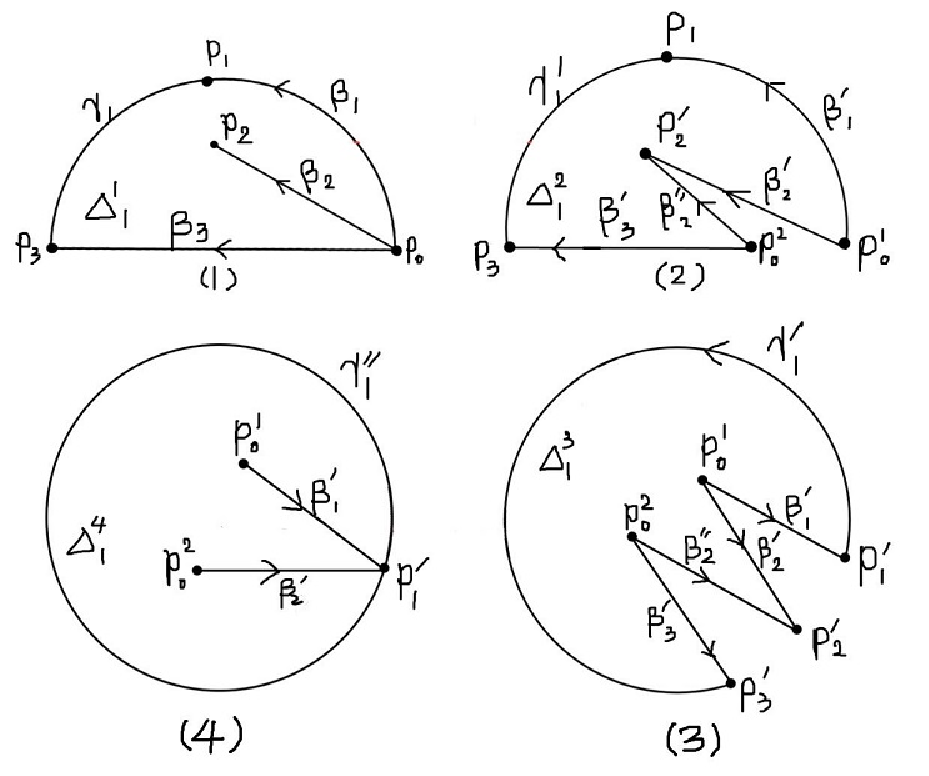}%
\caption{ }%
\label{fig4.7}%
\end{center}
\end{figure}

Next, we construct a new surface $\Sigma_{1}^{\prime}=\left(  f_{1}^{\prime
},\overline{\Delta}\right)  $ as follows. Denote by $\Delta_{1}^{1}=\Delta
_{1}\backslash\cup_{l=2}^{l_{1}-1}\beta_{l}$, which is a simply connected
domain. Cutting $\Delta_{1}^{1}$ along the paths $\cup_{l=1}^{l_{1}-1}%
\beta_{j}$, we can obtain a Jordan domain $\Delta_{1}^{2}$ as in Figure
\ref{fig4.7} (2) where $l_{1}=3$. Indeed, there exists an OPCOFOM
$h_{1}:\overline{\Delta_{1}^{2}}\rightarrow\overline{\Delta_{1}^{1}}$ such
that the restrictions
\[
h_{1}:\Delta_{1}^{2}\rightarrow\Delta_{1}^{1},\quad h_{1}:\beta_{l}^{\prime
}\rightarrow\beta_{l},\quad h_{1}:\beta_{l}^{\prime\prime}\rightarrow\beta_{l}%
\]
are homeomorphisms for $l=2,\dots,l_{1}-1$. Then the surface $F_{1}:=\left(
g_{1},\overline{\Delta_{1}^{2}}\right)  =\left(  f\circ h_{1},\overline
{\Delta_{1}^{2}}\right)  $ is simply connected and we can recover the surface
$\left(  f|_{\overline{\Delta_{1}}},\overline{\Delta_{1}}\right)  $ when we
glue $F_{1}$ along the pairs $\left(  g_{1},\beta_{l}^{\prime}\right)  $ and
$\left(  g_{1},\beta_{l}^{\prime\prime}\right)  $ for $l=2,\dots,l_{1}-1.$
Since
\[
\left(  g_{1},\beta_{1}^{\prime}\right)  \sim\left(  g_{1},\beta_{2}^{\prime
}\right)  ,\left(  g_{1},\beta_{2}^{\prime\prime}\right)  \sim\left(
g_{1},\beta_{3}^{\prime}\right)  ,\cdots,\left(  g_{1},\beta_{l_{1}-1}%
^{\prime\prime}\right)  \sim\left(  g_{1},\beta_{l_{1}}^{\prime}\right)  ,
\]
we can also glue $F_{1}$ along the above equivalent pairs and obtain a new
surface $\Sigma_{1}^{\prime}=\left(  f_{1}^{\prime},\overline{\Delta}\right)
,$ as the deformations described in Figure \ref{fig4.7} (2)--(4). In this way,
$p_{1},\dots,p_{l_{1}}$ are glued into a single point $p_{1}^{\prime}%
\in\partial\Delta$. It is clear that we have
\[
v_{f_{1}^{\prime}}\left(  p_{1}^{\prime}\right)  \leq v_{f}\left(
p_{1}\right)  +v_{f}\left(  p_{2}\right)  +\cdots+v_{f}\left(  p_{l_{1}%
-1}\right)  +v_{f|_{\overline{\Delta_{1}}}}\left(  p_{l_{1}}\right)  .
\]
When $b_{1}\notin E_{q}$, by condition (A) of Lemma \ref{move-b-1} we have
$v_{f}\left(  p_{2}\right)  =\cdots=v_{f}\left(  p_{l_{1}-1}\right)  =1$.
Thus
\begin{equation}
v_{f_{1}^{\prime}}\left(  p_{1}^{\prime}\right)  \leq v_{f}\left(
p_{1}\right)  +v_{f|_{\overline{\Delta_{1}}}}\left(  p_{l_{1}}\right)
+l_{1}-2. \label{vf}%
\end{equation}
As in Figure \ref{fig4.7} (2) or (3), $p_{0}^{1},\ldots,p_{0}^{l_{1}-1}$ are
regular points of $g_{1}$, and $g_{1}$ is homeomorphic on some neighborhoods
of $(\beta_{j}^{\prime})^{\circ}$ and $(\beta_{j}^{\prime\prime})^{\circ}$ in
$\overline{\Delta_{1}^{2}}$ for $j=1,\ldots,l_{1}-1$. Thus $f_{1}^{\prime}$ is
homeomorphic on some neighborhood of $\beta_{j}^{\prime}\backslash
\{p_{1}^{\prime}\}$ for $j=1,\ldots,l_{1}-1$. Therefore by (\ref{vf}) we have that

\begin{claim}
\label{CL1} $C_{f_{1}^{\prime}}^{\ast}\left(  \Delta\right)  =\emptyset$,
$\left(  f_{1}^{\prime},\partial\Delta\right)  \sim\left(  f,\gamma
_{1}\right)  ,$ (\ref{a6}) is an $\mathcal{F}(L,k)$-partition of
$\partial\Sigma_{1}^{\prime}$ and moreover
\begin{align*}
B_{f_{1}^{\prime}}^{\ast}\left(  p_{1}^{\prime}\right)   &  =0,\mathrm{\ if\ }%
p_{1}\in f^{-1}(E_{q});\\
B_{f_{1}^{\prime}}^{\ast}\left(  p_{1}^{\prime}\right)   &  =v_{f_{1}^{\prime
}}\left(  p_{1}^{\prime}\right)  -1\leq B_{f}^{\ast}\left(  p_{1}\right)
+B_{f|_{\overline{\Delta_{1}}}}^{\ast}\left(  p_{l_{1}}\right)  +l_{1}%
-1\mathrm{\ if\ }p_{1}\notin f^{-1}(E_{q}).
\end{align*}

\end{claim}

Now we will apply Claims \ref{CL2} and \ref{CL1} to verify the conclusion (i).
There is no doubt that $A(\Sigma_{1}^{\prime})+A(\Sigma_{2}^{\prime}%
)=A(\Sigma)$ and $L(\Sigma_{1}^{\prime})+L(\Sigma_{2}^{\prime})=L\left(
\Sigma\right)  .$ We can deduce from the previous constructions that
\[
\overline{n}\left(  \Sigma\right)  =\overline{n}\left(  \Sigma_{1}^{\prime
}\right)  +\overline{n}\left(  \Sigma_{2}^{\prime}\right)  +\left(
l_{1}-2\right)  \chi_{E_{q}}\left(  f\left(  p_{1}\right)  \right)  ,
\]
where $\chi_{E_{q}}\left(  f\left(  p_{1}\right)  \right)  =1$ if $p_{1}\in
f^{-1}(E_{q})$ and $\chi_{E_{q}}\left(  f\left(  p_{1}\right)  \right)  =0$
otherwise. Then we have%
\[
R(\Sigma_{1}^{\prime})+R\left(  \Sigma_{2}^{\prime}\right)  =R(\Sigma
)+4\pi\left(  l_{1}-2\right)  \chi_{E_{q}}\left(  f\left(  p_{1}\right)
\right)  .
\]
Take $\Sigma_{1}=\Sigma_{1}^{\prime}$ or $\Sigma_{2}^{\prime}$ such that
$H(\Sigma_{1})=\max\left\{  H\left(  \Sigma_{1}^{\prime}\right)  ,H\left(
\Sigma_{2}^{\prime}\right)  \right\}  .$ Then we have%
\[
H(\Sigma_{1})\geq H(\Sigma)+\frac{4\pi\left(  l_{1}-2\right)  \chi_{E_{q}%
}\left(  f\left(  p_{1}\right)  \right)  }{L\left(  \partial\Sigma\right)  }.
\]
By the restriction of inequality (\ref{210615}), however, we can obtain the
contradiction $H\left(  \Sigma_{1}\right)  >H_{L}$ when $l_{1}>2$ and
$p_{1}\in f^{-1}(E_{q}).$ Then in the sequel we assume that

\begin{condition}
\label{Condition}$l_{1}=2$ or $f\left(  p_{1}\right)  \notin E_{q}.$
\end{condition}

If $\Sigma_{1}=\Sigma_{2}^{\prime},$ then by Claim \ref{CL2}, $\Sigma_{1}$
satisfies (i). Thus in the sequel, we assume that
\[
\Sigma_{1}=\left(  f_{1},\Delta\right)  =\Sigma_{1}^{\prime}=\left(
f_{1}^{\prime},\overline{\Delta}\right)  ,
\]
say, $f_{1}=f_{1}^{\prime}.$ Then by condition $f\left(  p_{1}\right)  \notin
E_{q}$ it is trivial that%
\[
\#\left(  \partial\Delta\right)  \cap f_{1}^{-1}\left(  E_{q}\right)
=\#\gamma_{1}\cap f^{-1}\left(  E_{q}\right)  \leq\#\left(  \partial
\Delta\right)  \cap f^{-1}\left(  E_{q}\right)  ,
\]
and
\begin{equation}
\#C_{f_{1}}^{\ast}\left(  \partial\Delta\right)  =\#C_{f_{1}}^{\ast}\left(
\left(  \partial\Delta\right)  \backslash\{p_{1}^{\prime}\}\right)
+\#C_{f_{1}}^{\ast}\left(  p_{1}^{\prime}\right)  =\#C_{f}^{\ast}\left(
\gamma_{1}^{\circ}\right)  +\#C_{f_{1}}^{\ast}\left(  p_{1}^{\prime}\right)  .
\label{===}%
\end{equation}
Thus, by the relations $\gamma_{2}=\left(  \partial\Delta\right)
\backslash\gamma_{1}^{\circ}$ and $\gamma_{2}\supset\{p_{0},p_{1},p_{l_{1}%
}\},$ we have
\begin{align*}
\#C_{f}^{\ast}\left(  \partial\Delta\right)  -\#C_{f_{1}}^{\ast}\left(
\partial\Delta\right)   &  =\#C_{f}^{\ast}\left(  \partial\Delta\right)
-\#C_{f}^{\ast}\left(  \gamma_{1}^{\circ}\right)  -\#C_{f_{1}}^{\ast}\left(
p_{1}^{\prime}\right) \\
&  =\#C_{f}^{\ast}\left(  \gamma_{2}\right)  -\#C_{f_{1}}^{\ast}\left(
p_{1}^{\prime}\right) \\
&  \geq\#C_{f}^{\ast}\left(  p_{0}\right)  +\#C_{f}^{\ast}\left(
p_{1}\right)  +\#C_{f}^{\ast}\left(  p_{l_{1}}\right)  -\#C_{f_{1}}^{\ast
}\left(  p_{1}^{\prime}\right) \\
&  =1+\#C_{f}^{\ast}\left(  p_{1}\right)  +\#C_{f}^{\ast}\left(  p_{l_{1}%
}\right)  -\#C_{f_{1}}^{\ast}\left(  p_{1}^{\prime}\right) \\
&  \geq1+0+0-\#C_{f_{1}}^{\ast}\left(  p_{1}^{\prime}\right)  \geq0.
\end{align*}
Therefore, (\ref{as2-10}) holds, equality holding only if
\begin{equation}
\#C_{f}^{\ast}\left(  p_{1}\right)  =\#C_{f}^{\ast}\left(  p_{l_{1}}\right)
=0, \label{=0}%
\end{equation}
and%
\begin{equation}
\#C_{f}^{\ast}\left(  \gamma_{2}\right)  =\#C_{f}^{\ast}\left(  p_{0}\right)
=\#C_{f_{1}}^{\ast}\left(  p_{1}^{\prime}\right)  =1, \label{==1}%
\end{equation}
which implies
\begin{equation}
f_{1}(p_{1}^{\prime})=f\left(  \left\{  p_{l}\right\}  _{l=1}^{v}\right)
=b_{1}\notin E_{q}. \label{outEq}%
\end{equation}

Assume that the equality in (\ref{as2-10}) holds. Then (\ref{=0}%
)--(\ref{outEq}) hold and imply%
\begin{equation}
B_{f}^{\ast}\left(  p_{1}\right)  =B_{f}^{\ast}\left(  p_{l_{1}}\right)
=0\mathrm{\ but\ }B_{f_{1}}^{\ast}\left(  p_{1}^{\prime}\right)  \geq1.
\label{=0>0}%
\end{equation}
By (\ref{===}), (\ref{=0}) and (\ref{==1}), considering that $\partial
\Delta=\gamma_{1}+\gamma_{2}$ we have
\[
B_{f}^{\ast}\left(  \partial\Delta\right)  =B_{f}^{\ast}\left(  \gamma
_{1}^{\circ}\right)  +B_{f}^{\ast}\left(  \gamma_{2}\right)  =B_{f}^{\ast
}\left(  \gamma_{1}^{\circ}\right)  +B_{f}^{\ast}\left(  p_{0}\right)  ,
\]%
\[
B_{f_{1}}^{\ast}\left(  \partial\Delta\right)  =B_{f}^{\ast}\left(  \gamma
_{1}^{\circ}\right)  +B_{f_{1}}^{\ast}\left(  p_{1}^{\prime}\right)  ,
\]
and%
\[
B_{f_{2}^{\prime}}^{\ast}\left(  \partial\Delta\right)  =B_{f|_{\overline
{\Delta_{2}}}}^{\ast}\left(  p_{0}\right)  ;
\]
and then
\[
B_{f}^{\ast}\left(  \partial\Delta\right)  -B_{f_{1}}^{\ast}\left(
\partial\Delta\right)  =B_{f}^{\ast}\left(  p_{0}\right)  -B_{f_{1}}^{\ast
}\left(  p_{1}^{\prime}\right)  .
\]
Thus, by Claim \ref{CL1}, (\ref{=0>0}), and the assumption $v_{f}^{\ast
}\left(  p_{0}\right)  =v,$ we have%
\begin{align*}
&  B_{f}^{\ast}\left(  \partial\Delta\right)  -B_{f_{1}}^{\ast}\left(
\partial\Delta\right) \\
&  \geq B_{f}^{\ast}\left(  p_{0}\right)  -B_{f}^{\ast}\left(  p_{1}\right)
-B_{f|_{\overline{\Delta_{1}}}}^{\ast}\left(  p_{l_{1}}\right)  -l_{1}+1\\
&  =v-1-0-0-l_{1}+1=v-l_{1},
\end{align*}
and then by (\ref{=0>0}) and the assumption $v_{f}^{\ast}\left(  p_{0}\right)
=v$ we have
\begin{equation}
B_{f_{1}}^{\ast}\left(  \partial\Delta\right)  \leq B_{f}^{\ast}\left(
\partial\Delta\right)  -v+l_{1}\leq B_{f}^{\ast}\left(  \partial\Delta\right)
, \label{noeq}%
\end{equation}
with equality only if $l_{1}=v.$

Now we assume the equality in (\ref{as2-10}) holds while (\ref{as2-11}) does
not hold, which implies $l_{1}=v$ by (\ref{noeq}).

\label{zzzzzzzzzzzzzzzzzzzzzzzzzzzzzzzzzzzzzzzzz}Since $f$ is injective on
$\alpha_{1}\left(  a_{1},a_{2}\right)  \backslash\{a_{1}\},p_{1}\in\alpha
_{1}\left(  a_{1},a_{2}\right)  \backslash\{a_{1}\},$ $p_{1}\neq p_{l_{1}}$ in
Case (5) and $f(p_{1})=f\left(  p_{l_{1}}\right)  ,$ we have $p_{l_{1}}%
\notin\alpha_{1}\left(  a_{1},a_{2}\right)  \backslash\{a_{1}\},$ which
implies $a_{2}\in\gamma_{1}\ $and $p_{l_{1}}\notin\beta_{1}.$ Thus
$a_{1}\notin\gamma_{1}^{\circ}$ and $a_{1}\in\gamma_{2}=\left(  \partial
\Delta\right)  \backslash\gamma_{1}^{\circ},$ say, $\#\left[  \gamma_{2}%
\cap\{a_{j}\}_{j=1}^{m}\right]  \geq1,$ with equality if and only if
$\gamma_{2}\cap\{a_{j}\}_{j=1}^{m}=\left\{  a_{1}\right\}  .$

\begin{discussion}
\label{dis1}(a) If $\gamma_{2}$ contains two points of $\left\{
a_{j}\right\}  _{j=1}^{m},$ then $\Sigma_{1}=\Sigma_{1}^{\prime}\in
\mathcal{C}^{\ast}\left(  L,m-1\right)  .$ In fact in this case, (\ref{a6})
contains at most $k\leq m-1$ terms, and thus by Claim \ref{CL1} (\ref{a6})
$\Sigma_{1}\in\mathcal{F}(L,k)\subset\mathcal{F}(L,m-1).$

(b) If $\gamma_{2}$ contains only one point of $\left\{  a_{j}\right\}
_{j=1}^{m},$ say, $\gamma_{2}\cap\{a_{j}\}_{j=1}^{m}=\left\{  a_{1}\right\}
,$ then $p_{l_{1}}\in\alpha_{m}\left(  a_{m},a_{1}\right)  \backslash
\{a_{m}\}$ and then either $\left(  f,\gamma_{2}\right)  $ is a simple closed
arc of $\partial\Sigma_{1}$, or it is a folded arc, say,
\[
\left(  f,\gamma_{2}\right)  =c_{m}\left(  f(p_{l_{1}}),f(a_{1})\right)
+c_{1}\left(  f(a_{1}),f\left(  p_{l_{1}}\right)  \right)  =c_{m}\left(
f(p_{l_{1}}),f(a_{1})\right)  -c_{m}\left(  f\left(  p_{l_{1}}\right)
,f(a_{1})\right)  .
\]
Hence either
\[
\mathfrak{L}\left(  \partial\Sigma_{1}\right)  \leq\mathfrak{L}\left(
\partial\Sigma\right)  -1
\]
by Lemma \ref{L-1}; or $\left(  f,\overline{\Delta_{2}}\right)  =S$ by Lemma
\ref{cut-3} (i), and thus
\[
A(\Sigma_{1})\leq A(\Sigma)-4\pi.
\]

\end{discussion}

Summarizing (\ref{noeq}) and Discussion \ref{dis1}, we can derive that the
equality in (\ref{as2-10}) holds only if at least one of (\ref{as2-11}%
)-(\ref{23-3}) holds. Then (i) holds in Case (5), and we have finished the proof.
\end{proof}

\begin{remark}
\label{movement}When (ii) holds, say, in Case (4), $f_{1}$ plays the role that
moves the branch property of $p_{0}$ to $p_{1},$ so that $H\left(
\Sigma\right)  ,R(\Sigma),\partial\Sigma,\overline{n}\left(  \Sigma\right)  $
and the branch property of all other points, say, points in $\left(
\partial\Delta\right)  \backslash\{p_{0},p_{1}\},$ remain unchanged, while
$p_{0}$ becomes a regular point and $p_{1}$ becomes a branch point with
$v_{f_{1}}\left(  p_{1}\right)  =v_{f}\left(  p_{1}\right)  +v_{f}\left(
p_{0}\right)  -1\ $(note that the interior $\alpha_{1}\left(  p_{0}%
,p_{1}\right)  ^{\circ}$ of $\alpha_{1}\left(  p_{0},p_{1}\right)  $ contains
no branch point of $f$ and contains no point of $f^{-1}(E_{q})$). Such
movement fails in Case (5), and in this case, (i) holds.
\end{remark}

\begin{corollary}
\label{move-together}Let $\Sigma=\left(  f,\overline{\Delta}\right)  \ $be a
surface in $\mathcal{C}^{\ast}(L,m)\ $with the $\mathcal{C}^{\ast}\left(
L,m\right)  $-partitions (\ref{a1}) and (\ref{a2}). Assume that condition (A)
of Lemma \ref{move-b-1} holds, say $C_{f}^{\ast}\left(  \Delta\right)
=\emptyset,$ and assume (\ref{210615}) holds. Write
\[
\mathcal{E}_{f}:=C_{f}^{\ast}\left(  \partial\Delta\right)  \cup\left(
\partial\Delta\cap f^{-1}(E_{q})\right)  =\left\{  p_{0}^{\prime}%
,p_{1}^{\prime},\dots,p_{s-1}^{\prime}\right\}  ,
\]
and assume $p_{0}^{\prime}\in C_{f}^{\ast}\left(  \partial\Delta\right)  ,$
$s\geq2$ and $p_{0}^{\prime},\dots,p_{s-1}^{\prime}$ are arranged on
$\partial\Delta$ anticlockwise. Then there exists a surface $\Sigma
_{1}=\left(  f_{1},\overline{\Delta}\right)  \in\mathcal{C}^{\ast}\left(
L,m\right)  $ such that $C_{f_{1}}^{\ast}\left(  \Delta\right)  =\emptyset$
and one of the followings holds.

(a) The conclusion (i) of Lemma \ref{move-b-1} holds. Thus $L(\partial
\Sigma_{1})<L\left(  \partial\Sigma\right)  ,$ $\#\mathcal{E}_{f_{1}}%
\leq\#\mathcal{E}_{f}$ and either $\#C_{f_{1}}^{\ast}\left(  \partial
\Delta\right)  \leq\#C_{f}^{\ast}\left(  \partial\Delta\right)  -1,$ or
$\#C_{f_{1}}^{\ast}\left(  \partial\Delta\right)  =\#C_{f}^{\ast}\left(
\partial\Delta\right)  \ $and one of (\ref{as2-11})--(\ref{23-3}) holds.

(b) $p_{1}^{\prime}\in C_{f}^{\ast}\left(  \partial\Delta\right)  ,$ $H\left(
\Sigma_{1}\right)  =H(\Sigma),$ $\partial\Sigma_{1}=\partial\Sigma$,
$\#\mathcal{E}_{f_{1}}=\left\{  p_{1}^{\prime},\dots,p_{s-1}^{\prime}\right\}
=\#\mathcal{E}_{f}-1,$ and $B_{f_{1}}^{\ast}\left(  \partial\Delta\right)
=B_{f}^{\ast}\left(  \partial\Delta\right)  .$
\end{corollary}

\begin{proof}
Let $p_{0}^{\prime}\in C_{f}^{\ast}\left(  \partial\Delta\right)  $ and
$p_{0}^{\prime},p_{1}^{\prime},\dots,p_{s-1}^{\prime},s\geq2,$ be all points
of $\mathcal{E}_{f}$ arranged anticlockwise on $\partial\Delta$. Then
$\mathcal{E}_{f}$ gives a partition of $\partial\Delta$ as
\[
\partial\Delta=\beta_{1}^{\prime}\left(  p_{0}^{\prime},p_{1}^{\prime}\right)
+\beta_{2}^{\prime}\left(  p_{1}^{\prime},p_{2}^{\prime}\right)  +\dots
+\beta_{s}^{\prime}\left(  p_{s-1}^{\prime},p_{0}^{\prime}\right)  .
\]
Without loss of generality, we assume that $p_{0}^{\prime}\in\alpha_{1}\left(
a_{1},a_{2}\right)  \backslash\{a_{2}\}$ is the first point of $C_{f}^{\ast
}\left(  \partial\Delta\right)  $ in $\alpha_{1}\left(  a_{1},a_{2}\right)  ,$
say,
\[
C_{f}^{\ast}\left(  \partial\Delta\right)  \cap\alpha_{1}\left(  a_{1}%
,p_{0}^{\prime}\right)  =\{p_{0}^{\prime}\}.
\]

Firstly, we consider the simple case that
\begin{equation}
\beta_{1}^{\prime}\left(  p_{0}^{\prime},p_{1}^{\prime}\right)  \subset
\alpha_{1}\left(  p_{0}^{\prime},a_{2}\right)  . \label{as2-3}%
\end{equation}
We may further assume $f\left(  p_{0}^{\prime}\right)  \neq f(p_{1}^{\prime
}).$ Otherwise, we must have that $p_{0}^{\prime}=a_{1},p_{1}^{\prime}=a_{2}$
and that $c_{1}=c_{1}\left(  q_{1},q_{2}\right)  =\left(  f,\beta_{1}^{\prime
}\left(  p_{0}^{\prime},p_{1}^{\prime}\right)  \right)  =\left(  f,\alpha
_{1}\right)  $ is a circle with $\alpha_{1}\cap f^{-1}(E_{q})=\emptyset$, and
then we can discuss based on the following argument.

\begin{Argument}
\label{agr1}Consider a proper subarc $\beta_{1}=\beta_{1}\left(  p_{0}%
^{\prime},p_{01}^{\prime}\right)  \ $of $\beta_{1}^{\prime}=\beta_{1}^{\prime
}\left(  p_{0}^{\prime},p_{1}^{\prime}\right)  $, with $f(p_{01}^{\prime})\neq
f(p_{0}^{\prime})$ and $p_{01}^{\prime}\notin f^{-1}(E_{q}),$ so that
$f\left(  \beta_{1}\right)  $ has other $v-1$ $f$-lifts $\beta_{2}\left(
p_{0}^{\prime},p_{02}^{\prime}\right)  ,\dots,\beta_{v}\left(  p_{0}^{\prime
},p_{0v}^{\prime}\right)  $ so that $p_{0}^{\prime},\left\{  p_{0l}^{\prime
}\right\}  _{l=1}^{v}\ $and $\{\beta_{l}\}_{l=1}^{v}$ satisfy all conditions
of Lemma \ref{move-b-1} and the condition of Case 4 in the proof of Lemma
\ref{move-b-1}. Then by Lemma \ref{move-b-1} there exists a surface
$\Sigma_{1}^{\prime}=\left(  f_{1}^{\prime},\overline{\Delta}\right)
\in\mathcal{C}^{\ast}\left(  L,m\right)  $ so that $C_{f_{1}}^{\ast}\left(
\Delta\right)  =\emptyset\ $and Lemma \ref{move-b-1} (ii) holds, say,
$\partial\Sigma_{1}^{\prime}=\partial\Sigma,H(\Sigma_{1}^{\prime})=H(\Sigma)$,
$p_{01}^{\prime}\in C_{f_{1}^{\prime}}^{\ast},\mathcal{E}_{f_{1}^{\prime}%
}=\left\{  p_{01}^{\prime},p_{1}^{\prime},\dots,p_{s-1}^{\prime}\right\}
\ $and $B_{f_{1}^{\prime}}^{\ast}\left(  \partial\Delta\right)  =B_{f}^{\ast
}\left(  \partial\Delta\right)  ,$ and moreover, (\ref{a1}) and (\ref{a2}) are
still $\mathcal{C}^{\ast}\left(  L,m\right)  $-partitions of $\partial
\Sigma_{1}^{\prime}.$ Then we can replace $\Sigma$ with $\Sigma_{1}^{\prime}$
to continue our proof under (\ref{as2-3}).
\end{Argument}

Now we may assume $f\left(  p_{0}^{\prime}\right)  \neq f(p_{1}^{\prime})$ and
forget Argument \ref{agr1}. Let $\beta_{1}=\beta_{01}^{\prime}\left(
p_{0}^{\prime},p_{01}^{\prime}\right)  $ be the longest subarc of $\beta
_{1}^{\prime}\left(  p_{0}^{\prime},p_{1}^{\prime}\right)  $ such that (B) and
(C) of Lemma \ref{move-b-1} are satisfied by $\beta_{1}.$ There is nothing to
show when conclusion (i) of Lemma \ref{move-b-1} holds.

Assume conclusion (ii) of Lemma \ref{move-b-1} holds for $\beta_{1}$. Then
only Case (4) ocurs and $p_{01}^{\prime}\notin f^{-1}(E_{q})$. If
$p_{01}^{\prime}\neq p_{1}^{\prime},$ then we can extend $\beta_{1}$ longer so
that it still is a subarc of $\beta_{1}^{\prime}\left(  p_{0}^{\prime}%
,p_{1}^{\prime}\right)  \subset\alpha_{1}\left(  p_{0}^{\prime},a_{2}\right)
$ and satisfies (B) and (C), which contradicts definition of $\beta_{1}$.
Then, $p_{01}^{\prime}=p_{1}^{\prime}\in C_{f}^{\ast}\left(  \partial
\Delta\right)  ,$ and by Lemma \ref{move-b-1} (ii) we have

(c) There exists a surface $\Sigma_{1}=\left(  f_{1},\overline{\Delta}\right)
\in\mathcal{C}^{\ast}\left(  L,m\right)  $ such that $C_{f_{1}}^{\ast}\left(
\Delta\right)  =\emptyset$ and (b) holds, and moreover (\ref{a1}) and
(\ref{a2}) are still $\mathcal{C}^{\ast}\left(  L,m\right)  $-partitions of
$\partial\Sigma_{1}.$

\begin{remark}
\label{EfNoEq}The corollary is proved under the consition (\ref{as2-3}). When
$p_{1}^{\prime}\notin C_{f}^{\ast}\left(  \partial\Delta\right)  ,$ we have
$p_{1}^{\prime}\in f^{-1}(E_{q}),$ and then only (a) holds.
\end{remark}

Next, we show what will happen if (\ref{as2-3}) fails. Then $a_{2}\in\beta
_{1}^{\prime\circ}$ and so $a_{2}\notin\mathcal{E}_{f}.$ Assume that
\[
p_{1}^{\prime}\in\alpha_{j_{0}}\left(  a_{j_{0}},a_{j_{0}+1}\right)
\backslash\{a_{j_{0}}\},\label{as2-3+1}%
\]
for some $j_{0}>1$. Then we can find a point $p_{1}\in\alpha_{1}\left(
p_{0}^{\prime},a_{2}\right)  $ so that $\beta_{1}=\alpha_{1}\left(
p_{0}^{\prime},p_{1}\right)  $ is a maximal subarc of $\alpha_{1}\left(
p_{0},a_{2}\right)  $ satisfying conditions (B) and (C) of Lemma
\ref{move-b-1}. Then $p_{1}\notin\mathcal{E}_{f}$ and according to the above
proof only Case 4 or Case 5 occurs. If Case 5 occurs, then the proof for Case
(5) deduces the conclusion (i) of Lemma \ref{move-b-1}, and so does (a). If
Case 4 occurs, then by the condition $C_{f}^{\ast}\left(  \Delta\right)
=\emptyset,$ the maximal property of $\beta_{1}$ and Lemma \ref{Ri}, we have
$p_{1}=a_{2}$. Then the proof of Case 4 again deduces that (ii) in Lemma
\ref{move-b-1} holds, and we obtain a surface $\Sigma_{1}=\left(
f_{1},\overline{\Delta}\right)  \in\mathcal{C}^{\ast}\left(  L,m\right)  $
such that $H(\Sigma_{1})=H(\Sigma),$ $\partial\Sigma_{1}=\partial\Sigma,$
$C_{f_{1}}^{\ast}\left(  \Delta\right)  =\emptyset,$ $p_{1}\notin f_{1}%
^{-1}(E_{q}),$ $B_{f_{1}}^{\ast}\left(  \partial\Delta\right)  =B_{f}^{\ast
}\left(  \partial\Delta\right)  $ and
\[
\mathcal{E}_{f_{1}}=\{p_{1},p_{1}^{\prime},\dots,p_{s-1}^{\prime}%
\}=\{a_{2},p_{1}^{\prime},\dots,p_{s-1}^{\prime}\}.
\]
Thus using Lemma \ref{move-b-1} repeatedly, we can either prove (a) holds, or
obtain a surface $\Sigma_{j_{0}}=\left(  f_{j_{0}},\overline{\Delta}\right)  $
such that $H(\Sigma_{j_{0}})=H(\Sigma),$ $\partial\Sigma_{j_{0}}%
=\partial\Sigma,$ $C_{f_{j_{0}}}^{\ast}\left(  \Delta\right)  =\emptyset,$
$a_{j_{0}}\notin f_{j_{0}}^{-1}(E_{q}),$ $B_{f_{j_{0}}}^{\ast}\left(
\partial\Delta\right)  =B_{f}^{\ast}\left(  \partial\Delta\right)  $ and
\[
\mathcal{E}_{f_{j_{0}}}=\{a_{j_{0}},p_{1}^{\prime},\dots,p_{s-1}^{\prime
}\}\mathrm{\ and\ }B_{f_{j_{0}}}^{\ast}\left(  \partial\Delta\right)
=B_{f}^{\ast}\left(  \partial\Delta\right)  .
\]
Note that $a_{j_{0}}\ $and $p_{1}^{\prime}$ are both contained in the same arc
$\alpha_{j_{0}}.$ Then we can go back to condition (\ref{as2-3}) to show that
either (a) or (b) holds, and moreover, by Remark \ref{EfNoEq}, (b) holds only
if $p_{1}^{\prime}\in C_{f_{j_{0}}}^{\ast}\left(  \partial\Delta\right)  ,$
which implies $p_{1}^{\prime}\not \in f_{j_{0}}^{-1}\left(  E_{q}\right)  .$
\end{proof}

\begin{corollary}
\label{move-together-1}Let $\Sigma_{0}=\left(  f_{0},\overline{\Delta}\right)
\ $be a surface in $\mathcal{C}^{\ast}(L,m)\ $with the $\mathcal{C}^{\ast
}\left(  L,m\right)  $-partitions (\ref{a1}) and (\ref{a2}). Assume that
condition (A) of Lemma \ref{move-b-1} holds, say $C_{f_{0}}^{\ast}\left(
\Delta\right)  =\emptyset,$ and that (\ref{210615}) holds. Then there exists a
surface $\Sigma_{1}=\left(  f_{1},\overline{\Delta}\right)  \in\mathcal{C}%
^{\ast}\left(  L,m\right)  ,$ such that $C_{f_{1}}^{\ast}\left(
\Delta\right)  =\emptyset,$%
\[
H\left(  \Sigma_{1}\right)  \geq H(\Sigma_{0}),L(\partial\Sigma_{1})\leq
L(\partial\Sigma_{0}),
\]
that $H\left(  \Sigma_{1}\right)  >H(\Sigma_{0})\ $implies $L(\partial
\Sigma_{1})<L(\partial\Sigma_{0}).$ Moreover one of the following conclusions
(I)--(II) holds.

(I) $\Sigma_{1}\in\mathcal{F}_{r}\left(  L,m\right)  $ and, in this case,
$L(\partial\Sigma_{1})<L(\partial\Sigma_{0})$ if and only if $C_{f_{0}}^{\ast
}\left(  \partial\Delta\right)  \neq\emptyset.$

(II) $\Sigma_{1}\in\mathcal{C}^{\ast}\left(  L,m\right)  ,$ and both
$\mathcal{E}_{f_{1}}$ and $C_{f_{1}}^{\ast}\left(  \partial\Delta\right)  $
are the same singleton$,$ say, $\mathcal{E}_{f_{1}}=C_{f_{1}}^{\ast}\left(
\partial\Delta\right)  $ is a singleton outside $f_{1}^{-1}(E_{q}).$ Moreover,
if $\#C_{f_{0}}^{\ast}\left(  \partial\Delta\right)  \neq\emptyset$ and
$f_{0}^{-1}(E_{q})\cap\partial\Delta\neq\emptyset,$ then $L(\partial\Sigma
_{1})<L(\partial\Sigma_{0})$ and either $\#C_{f_{1}}^{\ast}\left(
\partial\Delta\right)  \leq\#C_{f_{0}}^{\ast}\left(  \partial\Delta\right)
-1$ holds, or $\#C_{f_{1}}^{\ast}\left(  \partial\Delta\right)  =\#C_{f_{0}%
}^{\ast}\left(  \partial\Delta\right)  $ and one of (\ref{as2-11}%
)--(\ref{23-3}) hold with $f=f_{0}.$
\end{corollary}

\begin{proof}
We will prove this by induction on $\#C_{f_{0}}^{\ast}\left(  \partial
\Delta\right)  .$ If $C_{f_{0}}^{\ast}\left(  \partial\Delta\right)
=\emptyset,$ then (I) holds for $\Sigma_{1}=\Sigma_{0}$.

If $C_{f_{0}}^{\ast}\left(  \partial\Delta\right)  $ is a singleton and
$\mathcal{E}_{f_{0}}=C_{f_{0}}^{\ast}\left(  \partial\Delta\right)  ,$ then
$f_{0}^{-1}(E_{q})\cap\partial\Delta=\emptyset$ and so (II) holds with
$\Sigma_{1}=\Sigma_{0}$.

Now, assume that $C_{f_{0}}^{\ast}\left(  \partial\Delta\right)  \neq
\emptyset$ and $\#\mathcal{E}_{f_{0}}\geq2$. Then we can write%
\[
\mathcal{E}_{f_{0}}=\left\{  p_{0}^{0},p_{1}^{0},\dots,p_{s_{0}-1}%
^{0}\right\}
\]
so that $p_{j}^{0},j=0,1,\dots,s_{0}-1$, are arranged anticlockwise on
$\partial\Delta\ $and $p_{0}^{0}\in C_{f_{0}}^{\ast}\left(  \partial
\Delta\right)  .$ Then by Corollary \ref{move-together}, there exists a
surface $\Sigma_{1}\in\mathcal{C}^{\ast}(L,m)$ such that the following
conclusion (a)-$(n-1,n)$ or (b)-$(n-1,n)$ holds for $n=1.$

(a)-$(n-1,n)$: $C_{f_{n}}^{\ast}\left(  \Delta\right)  =\emptyset$ and the
conclusion (i) of Lemma \ref{move-b-1} holds, and thus $H(\Sigma_{1})\geq
H\left(  \Sigma_{0}\right)  ,L(\partial\Sigma_{1})<L(\partial\Sigma_{0}),$
$\#\mathcal{E}_{f_{1}}\leq\#\mathcal{E}_{f_{0}}$ and either $\#C_{f_{1}}%
^{\ast}\left(  \partial\Delta\right)  \leq\#C_{f_{0}}^{\ast}\left(
\partial\Delta\right)  -1$ holds or $\#C_{f_{1}}^{\ast}\left(  \partial
\Delta\right)  =\#C_{f_{0}}^{\ast}\left(  \partial\Delta\right)  $ and one of
(\ref{as2-11})--(\ref{23-3}) holds with $f=f_{0}$.

(b)-$(n-1,n)$: $C_{f_{n}}^{\ast}\left(  \Delta\right)  =\emptyset,$
$\mathcal{E}_{f_{n}}=\left\{  p_{1}^{n-1},\dots,p_{s_{n-1}-1}^{n-1}\right\}  $
with $p_{1}^{n-1}\in C_{f_{n-1}}^{\ast}\left(  \partial\Delta\right)  ,$
$H\left(  \Sigma_{n}\right)  =H(\Sigma_{n-1}),$ $\partial\Sigma_{n}%
=\partial\Sigma_{n-1},$ and $B_{f_{n}}^{\ast}\left(  \partial\Delta\right)
=B_{f_{n-1}}^{\ast}\left(  \partial\Delta\right)  .$

If $\Sigma_{1}\in\mathcal{F}_{r}(L,m)$, then (b)-$(0,1)$ does not hold, and
then (a)-$(0,1)$ holds, which imlpies (I). Note that when $\Sigma_{1}%
\in\mathcal{F}_{r}(L,m)$ and $H\left(  \Sigma_{1}\right)  >H(\Sigma_{0})$
hold, we must have $C_{f_{0}}^{\ast}\left(  \partial\Delta\right)
\neq\emptyset$.

Assume $\Sigma_{1}$ is not in $\mathcal{F}_{r}(L,m)$ and $\mathcal{E}_{f_{1}%
}=C_{f_{1}}^{\ast}\left(  \partial\Delta\right)  $ is a singleton. Then
(b)-$\left(  0,1\right)  $ holds and so $C_{f_{1}}^{\ast}\left(
\partial\Delta\right)  =\mathcal{E}_{f_{1}}=\left\{  p_{1}^{0},\dots
,p_{s_{0}-1}^{0}\right\}  =\left\{  p_{1}^{0}\right\}  \in C_{f_{n-1}}^{\ast
}\left(  \partial\Delta\right)  .$ In this case, we must have $f_{0}%
^{-1}(E_{q})\cap\partial\Delta=\emptyset$. Thus (II) holds.

Now, assume that $\Sigma_{1}$ is not in $\mathcal{F}_{r}(L,m)$ but
$\mathcal{E}_{f_{1}}$ contains at least two points. Then $C_{f_{1}}^{\ast
}\left(  \partial\Delta\right)  \neq\emptyset,$ and we can iterate the above
discussion to obtain surfaces $\Sigma_{j}=\left(  f_{j},\overline{\Delta
}\right)  ,j=1,2,\dots,n_{0},$ so that $\Sigma_{n_{0}}$ no longer can be
iterated. Then for each $\Sigma_{n},$ $n=1,\dots,n_{0},$ (a)-$\left(
n-1,n\right)  $ or (b)-$\left(  n-1,n\right)  $ holds, and thus one of the
following holds.

(c) $C_{f_{n_{0}}}^{\ast}\left(  \partial\Delta\right)  $ is empty and
(a)-$\left(  n_{0}-1,n_{0}\right)  $ holds.

(d) $C_{f_{n_{0}}}^{\ast}\left(  \partial\Delta\right)  $ is a singleton and
(a)-$\left(  n_{0}-1,n_{0}\right)  $ holds.

(e) $C_{f_{n_{0}}}^{\ast}\left(  \partial\Delta\right)  $ is a singleton and
(b)-$\left(  n_{0}-1,n_{0}\right)  $ holds.

We show that $\Sigma_{n_{0}}$ is a desired surface. First of all, we have
$\#C_{f_{n_{0}}}^{\ast}\left(  \Delta\right)  =\emptyset.$

If (c) holds, then $\Sigma_{n_{0}}\in\mathcal{F}_{r}(L,m)$ and we obtain (I).

Assume (d) holds. Then we have $L(\partial\Sigma_{n_{0}})<L(\partial
\Sigma_{n_{0}-1})\leq L(\partial\Sigma_{0})$ and $C_{f_{0}}^{\ast}\left(
\partial\Delta\right)  \neq\emptyset.$ Then (II) holds in this case, no matter
$f_{0}^{-1}(E_{q})\cap\partial\Delta$ is empty or not.

Assume (e) holds. If all conditions (b)-$\left(  n-1,n\right)  $ hold for
$n=1,\dots,n_{0},$ then $s_{0}=n_{0}+1$, $\partial\Sigma_{n_{0}}%
=\partial\Sigma_{n_{0}-1}=\dots=\partial\Sigma_{0}$ and $\mathcal{E}_{f_{n}%
}=C_{f_{1}}^{\ast}\left(  \partial\Delta\right)  =\{p_{n}^{0},p_{n+1}%
^{0},\dots,p_{n_{0}}^{0}\}$ for $n=0,1,2,\dots,n_{0},$ say, $f_{0}^{-1}%
(E_{q})\cap\partial\Delta$ is empty. Thus, when $f_{0}^{-1}(E_{q})\cap
\partial\Delta\neq\emptyset,$ (a)-$\left(  n_{0}^{\prime}-1,n_{0}^{\prime
}\right)  $ has to be satisfied for some $n_{0}^{\prime}<n_{0}$. Thus all
conclusion in (II) hold.
\end{proof}

\section{Proof of the main theorem}

Now, we can complete the proof of the main theorem, Theorem \ref{re}.

\begin{definition}
\label{df}Let $\Sigma=\left(  f,\overline{\Delta}\right)  \in\mathcal{F}.$ For
any two points $a$ and $b$ in $\overline{\Delta},$ define their $d_{f}%
$-distance $d_{f}\left(  a,b\right)  $ by%
\[
d_{f}(a,b)=\inf\left\{  L(f,I):I\mathrm{\ is\ a\ path\ in\ }\overline{\Delta
}\mathrm{\ from\ }a\mathrm{\ to\ }b\right\}  .
\]
For any two sets $A$ and $B$ in $\overline{\Delta}$, define their $d_{f}%
$-distance by%
\[
d_{f}\left(  A,B\right)  =\inf\left\{  d_{f}\left(  a,b\right)  :a\in A,b\in
B\right\}  .
\]

\end{definition}

\begin{lemma}
\label{d-f}\label{nobo}Let $\mathcal{L}=\{L^{\prime}>0:H_{L}$ is continuous at
$L\},\ L\in\mathcal{L}$ and let $L_{0}\in(0,L].$ Then there exists a positive
number $\delta_{L_{0}}$ such that
\begin{equation}
d_{f}(\Delta\cap f^{-1}(E_{q}),\partial\Delta)>\delta_{L_{0}} \label{god}%
\end{equation}
holds for all surfaces $\Sigma=\left(  f,\overline{\Delta}\right)  $ in
$\mathcal{F}(L)$ with $L(\partial\Sigma)\geq L_{0}$ and with (\ref{210615}).
\end{lemma}

This is proved in \cite{Zh2}. In fact, if this fails, then for any
$\varepsilon>0,$ there exists a surface $\Sigma\in\mathcal{F}(L)$ with
$L(\partial\Sigma)\geq L_{0}$ such that $d_{f}(\Delta\cap f^{-1}%
(E_{q}),\partial\Delta)<\varepsilon/3$ and $\Delta\cap f^{-1}(E_{q}%
)\neq\emptyset.$ Then one can cut $\Sigma$ from a boundary point on
$\partial\Delta$ to a point in $f^{-1}(E_{q})\cap\Delta,$ along a path
$I_{\varepsilon}\subset\overline{\Delta}$ so that $\left(  f,I_{\varepsilon
}\right)  $ is polygonal and that $2L(f,I_{\varepsilon})<\varepsilon,$
obtaining a surface $\Sigma_{\varepsilon}\in\mathcal{F}\left(  L+\varepsilon
\right)  $ with $L(\partial\Sigma_{\varepsilon})=L(\partial\Sigma
)+2L(f,I_{\varepsilon}),$ $A(\Sigma_{\varepsilon})=A(\Sigma)$ and
$\overline{n}\left(  \Sigma_{\varepsilon}\right)  \leq\overline{n}\left(
\Sigma\right)  -1.$ Then we have $R\left(  \Sigma_{\varepsilon}\right)  \geq
R(\Sigma)+4\pi$ and thus
\[
H(\Sigma_{\varepsilon})=\frac{R\left(  \Sigma_{\varepsilon}\right)
}{L(\partial\Sigma_{\varepsilon})}\geq\frac{R(\Sigma)+4\pi}{L(\partial
\Sigma)+\varepsilon}=\frac{R(\Sigma)/L(\partial\Sigma)+4\pi/L(\partial\Sigma
)}{1+\varepsilon/L(\partial\Sigma)}.
\]
This and (\ref{210615}) deduce that $H_{L+\varepsilon}\geq H(\Sigma
_{\varepsilon})>H_{L}+\frac{\pi}{2L(\partial\Sigma)}$ when $\varepsilon$ is
small enough. But this contradicts the assumption $L\in\mathcal{L},$ which
implies that $H_{L+\varepsilon}\rightarrow H_{L}$ as $\varepsilon
\rightarrow0.$

\begin{proof}
[\textbf{Proof of Theorem \ref{re}}]Let $\Sigma=\left(  f,\overline{\Delta
}\right)  \in\mathcal{C}^{\ast}(L,m)$ be a covering surface such that
(\ref{210615}) holds. If $C_{f}^{\ast}=C_{f}^{\ast}\left(  \overline{\Delta
}\right)  =\emptyset,$ then $\Sigma^{\prime}=\Sigma$ itself is the desired
surface in Theorem \ref{re}.

If $C_{f}^{\ast}\left(  \Delta\right)  =\emptyset,$ but $C_{f}^{\ast}\left(
\partial\Delta\right)  \neq\emptyset,$ then by Corollary \ref{move-together-1}%
, either the conclusion of Theorem \ref{re} holds with $L(\partial\Sigma
_{1})<L(\partial\Sigma_{0}),$ or

(III) there exists a surface $\Sigma_{1}=\left(  f_{1},\overline{\Delta
}\right)  \in\mathcal{C}^{\ast}(L,m)$ such that $C_{f_{1}}^{\ast}\left(
\Delta\right)  =\emptyset,$ $H\left(  \Sigma_{1}\right)  \geq H\left(
\Sigma\right)  $, $L\left(  \partial\Sigma_{1}\right)  \leq L(\partial
\Sigma),$ $H\left(  \Sigma_{1}\right)  >H\left(  \Sigma\right)  $ only if
$L\left(  \partial\Sigma_{1}\right)  <L(\partial\Sigma);$ and both
$\mathcal{E}_{f_{1}}$ and $C_{f_{1}}^{\ast}\left(  \partial\Delta\right)  $
are the same singleton. Moreover, if $\#C_{f}^{\ast}\left(  \partial
\Delta\right)  \neq\emptyset$ and $f^{-1}(E_{q})\cap\partial\Delta
\neq\emptyset,$ then $L(\partial\Sigma_{1})<L(\partial\Sigma_{0})$ and either
$\#C_{f_{1}}^{\ast}\left(  \partial\Delta\right)  \leq\#C_{f}^{\ast}\left(
\partial\Delta\right)  -1$ holds, or $\#C_{f_{1}}^{\ast}\left(  \partial
\Delta\right)  =\#C_{f}^{\ast}\left(  \partial\Delta\right)  $ and one of
(\ref{as2-11})--(\ref{23-3}) holds.

Assume $C_{f}^{\ast}\left(  \Delta\right)  \neq\emptyset.$ Then by Corollary
\ref{move-out}, we have

\begin{claim}
There exists a surface $\Sigma_{0}=(f_{0},\overline{\Delta})\in\mathcal{C}%
^{\ast}\left(  L,m\right)  $ such that $C_{f_{0}}^{\ast}\left(  \Delta\right)
=\emptyset$, $H\left(  \Sigma_{0}\right)  \geq H\left(  \Sigma\right)  $ and
$L\left(  \partial\Sigma_{0}\right)  \leq L(\partial\Sigma)$. Moreover,
$L(\partial\Sigma_{0})=L(\partial\Sigma)$ holds if and only if $H(\Sigma
_{0})=H(\Sigma),$ $\partial\Sigma_{0}=\partial\Sigma$ and $B_{f_{0}}^{\ast
}\left(  \partial\Delta\right)  >B_{f}^{\ast}\left(  \partial\Delta\right)  $ hold.
\end{claim}

If $C_{f_{0}}^{\ast}\left(  \partial\Delta\right)  =\emptyset,$ then
$B_{f_{0}}^{\ast}\left(  \partial\Delta\right)  =0$, $\Sigma_{0}\in
\mathcal{F}_{r}(L,m)$ and by the claim, $L(\partial\Sigma_{0})<L\left(
\partial\Sigma\right)  $, and thus $\Sigma^{\prime}=\Sigma_{0}$ satisfies the
conclusion of Theorem \ref{re}.

If $\#C_{f_{0}}^{\ast}\left(  \partial\Delta\right)  =1$ and $\mathcal{E}%
_{f_{0}}=C_{f_{0}}^{\ast}\left(  \partial\Delta\right)  =\{p_{0}\}$ is a
singleton, then (III) holds with $\Sigma_{1}=\Sigma_{0}$ by the claim.

Now assume that $\#C_{f_{0}}^{\ast}\left(  \partial\Delta\right)  \geq1$ and
$\#\mathcal{E}_{f_{0}}\geq2$. Then there exists a surface $\Sigma_{1}=\left(
f_{1},\overline{\Delta}\right)  \in\mathcal{C}^{\ast}\left(  L,m\right)  $
satisfying all conclusions of Corollary \ref{move-together-1} with (I), or
(II). When (I) holds, $\Sigma^{\prime}=\Sigma_{1}$ again satisfies Theorem
\ref{re}, and the proof finishes. If (II) holds, and (I) fails, (III) holds
again. So we may complete the proof based on $\Sigma_{1}$ under the assumption (III).

We will show that there exists a surface $\Sigma^{\prime}$ satisfying the
conclusion of \ref{re} with $L(\partial\Sigma^{\prime})<L(\partial\Sigma).$

Let $P$ and $P^{\ast}$ be two antipodal points of $S$ and let $\varphi
_{\theta}$ be a continuous rotation on $S$ with the axis passing through $P$
and $P^{\ast}\ $and rotation angle $\theta,$ which rotates anticlockwise
around $P$ when $\theta$ increases and we view $S$ from sinside. $P$ and
$P^{\ast}$ are chosen so that we can define $\theta_{0}=\theta_{0}\left(
\Sigma_{1}\right)  \in(0,\pi),$ such that
\[
\varphi_{\theta_{0}}\left(  \partial\Sigma_{1}\right)  \cap E_{q}\neq
\emptyset,
\]
while%
\[
\varphi_{\theta}\left(  \partial\Sigma_{1}\right)  \cap E_{q}=\emptyset
\mathrm{\ for\ all\ }\theta\in(0,\theta_{0}),
\]
and%
\[
\varphi_{\theta_{0}}(f_{1}(p_{0}))\notin E_{q}.
\]
We may assume $P^{\ast}$ and $P$ are outside $\partial\Sigma_{1}.$ We first
show that

\begin{claim}
\label{con1}There exists a surface $\Sigma_{2}=\left(  f_{2},\overline{\Delta
}\right)  \in\mathcal{C}^{\ast}(L,m)$ such that (\ref{210615}) holds,%
\begin{equation}
H(\Sigma_{2})=H(\Sigma_{1}),L\left(  \partial\Sigma_{2}\right)  =L(\partial
\Sigma_{1}), \label{hhll}%
\end{equation}%
\begin{equation}
\partial\Sigma_{2}=\left(  f_{2},\partial\Delta\right)  =\left(
\varphi_{\theta_{1}}\circ f_{1},\partial\Delta\right)  , \label{fg}%
\end{equation}%
\begin{equation}
\overline{n}\left(  \Sigma_{2},E_{q}\right)  =\overline{n}\left(  \Sigma
_{1},E_{q}\right)  ,A(\Sigma_{2})=A(\Sigma_{1}), \label{nnaa}%
\end{equation}
$\partial\Sigma_{2}$ contains at least one point of $E_{q},$ and $p_{0}%
\in\partial\Delta$ is the unique branch point of $f_{2}$ in $\overline{\Delta
}\backslash f_{2}^{-1}(E_{q}),$ where $\theta_{1}\in(0,2\pi].$
\end{claim}

Let $\delta_{L_{0}}$ with $L_{0}=L(\partial\Sigma_{1})$ be determined by Lemma
\ref{nobo} and let $\delta_{E_{q}}$ be the smallest positive distance between
points of $E_{q}.$ Then $d_{f_{1}}\left(  f_{1}^{-1}(E_{q}),\partial
\Delta\right)  >\delta_{L_{0}}$. Let $\theta_{1}$ be the maximal number in
$(0,\theta_{0})$ such that for each $\theta\in(0,\theta_{1})$
\[
\max_{\mathfrak{a}\in E_{q}}d\left(  \mathfrak{a},\varphi_{\theta
}(\mathfrak{a}\right)  )<\delta_{L_{0}}^{\prime}=\min(\delta_{E_{q}}%
,\delta_{L_{0}})/3.
\]
Let $b_{1},b_{2},\dots,b_{\overline{n}},\overline{n}=\overline{n}\left(
\Sigma_{1},E_{q}\right)  ,$ be all distinct points in $f_{1}(\Delta)\cap
E_{q}.$ Then for each $j\leq\overline{n}$ there exists a Jordan domain $U_{j}$
conaining $b_{j}$ with $j=1,\dots,\overline{n}$ and $\overline{U_{j}}%
\subset\Delta,$ such that $f_{1}$ restricted to $\overline{U_{j}}$ is a BCCM
onto the closed disk $\overline{V_{j}}=D\left(  f_{1}\left(  b_{j}\right)
,\delta_{L_{0}}^{\prime}\right)  ,$ with $\overline{U_{i}}\cap\overline{U_{j}%
}=\emptyset$ if $i\neq j$ and $b_{j}$ is the unique possible branch point of
$f_{1}$ in $\overline{U_{j}}.$

Let $g_{1}=\varphi_{\theta_{1}}\circ f_{1}:\overline{\Delta}\rightarrow S,$
and let $\phi_{j}$ be the homeomorphism from $\varphi_{\theta_{1}}%
(\overline{V_{j}})$ onto itself, which is an identity on $\partial
\varphi_{\theta_{1}}(\overline{V_{j}})$ and maps $\varphi_{\theta_{1}}\left(
f_{1}(b_{j})\right)  $ to $f(b_{j}).$ Note that both $f_{1}\left(
b_{j}\right)  $ and $\phi_{j}\left(  f_{1}\left(  b_{j}\right)  \right)  $ ar
both contained in $\varphi_{\theta_{1}}\left(  V_{j}\right)  .$ Let
$g_{1}^{\prime}$ be the mapping given by $g_{1}$ on $\overline{\Delta
}\backslash\left(  \cup_{j=1}^{\overline{n}}U_{j}\right)  $ and by $\phi
_{j}\circ g_{1}$ on $\overline{U_{j}}.\ $Then $g_{1}^{\prime}$ is an OPLM so
that $G_{1}=\left(  g_{1}^{\prime},\overline{\Delta}\right)  $ is contained in
$\mathcal{C}^{\ast}\left(  L,m\right)  $ with $C_{g_{1}^{\prime}}^{\ast
}(\overline{\Delta})=\{p_{0}\},$ and that, for each $j=1,\dots,\overline{n},$
$b_{j}$ is the only possible branch point of $g_{1}^{\prime}$ in
$\overline{U_{j}}$ with $g_{1}^{\prime}(b_{j})=f(b_{j})$ and $v_{g_{1}%
^{\prime}}(b_{j})=v_{f_{1}}(b_{j}).$ Thus it is the clear that (\ref{hhll}%
)--(\ref{nnaa}) hold for $\Sigma_{2}=G_{1}$. Therefore, in the case
$\theta_{1}=\theta_{0}$ we have $\left(  \partial\Delta\right)  \cap
g_{1}^{\prime-1}(E_{q})\neq\emptyset,$ and we proved Claim \ref{con1} when
$\theta_{1}=\theta_{0}.$

Assume that $\theta_{1}<\theta_{\theta_{0}}.$ Then $G_{1}$ satisfies all
assumptions of the (III), and additionally satisfies (\ref{210615}), and then
we still have $d_{g_{1}^{\prime}}\left(  g_{1}^{\prime-1}(E_{q}),\partial
\Delta\right)  >\delta_{L_{0}}.$ Moreover, we have $\theta_{0}\left(
G_{1}\right)  =\theta_{0}\left(  \Sigma_{1}\right)  -\theta_{1}.$ Then we can
repeat the above arguments at most $k-1=\left[  \frac{\theta_{0}\left(
\Sigma_{1}\right)  -\theta_{1}}{\theta_{1}}\right]  +1$ times to obtain a
surface $\Sigma_{2}=\left(  f_{2},\overline{\Delta}\right)  =G_{k}=\left(
g_{k}^{\prime},\overline{\Delta}\right)  $ satisfying Claim \ref{con1}. The
existence of $\Sigma_{2}$ is proved.

Now we can write $\mathcal{E}_{f_{2}}=\{p_{0},p_{1},\dots,p_{s-1}\},s\geq2,$
so that $p_{0}\in C_{f_{2}}^{\ast}\left(  \partial\Delta\right)  $ and
\begin{equation}
\{p_{1},\dots,p_{s-1}\}\subset f_{2}^{-1}(E_{q}). \label{ppp}%
\end{equation}
Then by Corollary \ref{move-together-1} there exists a surface $\Sigma
_{3}=\left(  f_{3},\overline{\Delta}\right)  \in\mathcal{C}^{\ast}\left(
L,m\right)  $ such that $C_{f_{3}}^{\ast}\left(  \Delta\right)  =\emptyset,$%
\[
H\left(  \Sigma_{3}\right)  \geq H(\Sigma_{2}),L(\Sigma_{3})<L(\partial
\Sigma_{1}),
\]
and moreover the following conclusion (I) or (II) holds true.

(I) $\Sigma_{3}\in\mathcal{F}_{r}\left(  L,m\right)  .$

(II) $\Sigma_{3}\in\mathcal{C}^{\ast}\left(  L,m\right)  ,$ and both
$\mathcal{E}_{f_{3}}$ and $C_{f_{3}}^{\ast}\left(  \partial\Delta\right)  $
are the same singleton, and either $\#C_{f_{3}}^{\ast}\left(  \partial
\Delta\right)  \leq\#C_{f_{2}}^{\ast}\left(  \partial\Delta\right)  -1$ holds,
or $\#C_{f_{1}}^{\ast}\left(  \partial\Delta\right)  =\#C_{f_{0}}^{\ast
}\left(  \partial\Delta\right)  $ and one of (\ref{as2-11})--(\ref{23-3}) hold
with $f=f_{2}.$

If (I) holds, the proof is completed.

If (II) holds, then we repeat the the same argument which deduce $\Sigma_{3}$
from $\Sigma_{1}.$ This interate can only be executed a finite number of times
(II), and at last we obtain a surface $\Sigma_{k}=\left(  f_{k},\overline
{\Delta}\right)  \in\mathcal{C}^{\ast}(L,m)$ such that%
\[
H\left(  \Sigma_{k}\right)  \geq H(\Sigma),L(\Sigma_{k})<L(\partial\Sigma),
\]
and one of the following two alternatives holds:

($\mathfrak{a}$) $\Sigma_{k}\in\mathcal{C}^{\ast}\left(  L,1\right)  ,$
$\mathfrak{L}\left(  \partial\Sigma_{k}\right)  =1,\mathcal{\ }A\left(
\Sigma_{k}\right)  <4\pi\ $and $C_{f_{k}}^{\ast}\left(  \overline{\Delta
}\right)  =C_{f_{k}}^{\ast}\left(  \partial\Delta\right)  =\emptyset.$

($\mathfrak{b}$) $\#C_{f_{k}}^{\ast}\left(  \overline{\Delta}\right)
=\#C_{f_{k}}^{\ast}(\partial\Delta)=\#\mathcal{E}_{f_{k}}=1,$ $H(\Sigma
_{k})\geq H(\Sigma),L(\partial\Sigma_{k})<L\left(  \partial\Sigma\right)  ,$
and one of the three alternatives holds: $\Sigma_{k}\in\mathcal{C}^{\ast
}\left(  L,m-1\right)  ,A\left(  \Sigma_{k}\right)  <A\left(  f\right)
-4\pi,L\left(  \partial\Sigma_{k}\right)  <L\left(  \partial\Sigma\right)  $.

If ($\mathfrak{a}$) holds, then $\partial\Sigma_{k}$ is a simple convex circle
and, $f_{k}\left(  \overline{\Delta}\right)  $ as a set, is the closed disk on
$S$ enclosed by $\partial\Sigma_{k},$ and by argument principle we have
$\Sigma_{k}\in\mathcal{F}_{r}\left(  L,1\right)  \subset\mathcal{F}_{r}(L,m).$

If ($\mathfrak{b}$) holds, then we can repeat the whole above argument, which
deduces $\Sigma_{1}$ first from $\Sigma$ then deduces $\Sigma_{k}$ from
$\Sigma_{1},\ $to obtain a surface $\Sigma_{s}$ from $\Sigma_{k}$ satisfying
($\mathfrak{a}$) or ($\mathfrak{b}$). But this interate can only be executed a
finite number of times by ($\mathfrak{b}$) and at last we obtain a surface
$\Sigma_{t}$ satisfying ($\mathfrak{a}$).
\end{proof}

\end{document}